\documentclass{article}
\usepackage[a4paper]{geometry}
\usepackage[all]{xy}
\usepackage{amsmath,caption,enumitem,graphicx,amssymb,amsthm,mathrsfs,mathtools,pgffor,subcaption,tikz,tikz-3dplot}
\usepackage[initials]{amsrefs}

\newtheorem{theorem}{Theorem}[section]
\newtheorem{lemma}[theorem]{Lemma}

\theoremstyle{definition}
\newtheorem*{definition}{Definition}

\numberwithin{equation}{section}
\numberwithin{figure}{section}

\newcommand{\SL}{\textnormal{\textbf{SL}}_{\mathbf{2}}}
\newcommand{\PATHP}{\mathscr{P}}
\newcommand{\FR}{\textnormal{\textbf{FR}}}

\renewcommand{\leq}{\leqslant}
\renewcommand{\geq}{\geqslant}

\DeclareMathOperator{\sgn}{sgn}

\setlist[enumerate]{leftmargin=20pt,itemsep=0pt,topsep=0pt}
\setlist[enumerate,1]{label=\emph{(\roman*)}}

\makeatletter
\renewcommand\section{\@startsection {section}{1}{\z@}%
                                   {-3.5ex \@plus -1ex \@minus -.2ex}%
                                   {1.3ex \@plus.2ex}%
                                   {\normalfont\large\scshape}}
\makeatother

\title{ \vspace{-6ex}\bf \large Classifying $\textnormal{\textbf{SL}}_{\mathbf{2}}$-tilings}

\author{ \vspace{-3ex}\normalsize Ian Short\thanks{
		School of Mathematics and Statistics, The Open University, Milton Keynes, MK7 6AA, United Kingdom.\newline{}
		{}
		}}

\date{\vspace{-4ex}}

\usetikzlibrary{calc}
\usetikzlibrary{matrix}
\usetikzlibrary{decorations.markings}
\usetikzlibrary{arrows.meta}
\usetikzlibrary{arrows,shapes,positioning}
\tikzstyle arrowstyle=[scale=1]
\tikzstyle directed=[postaction={decorate,decoration={markings,
  mark=at position 0.5 with {\arrow[arrowstyle]{{latex}}}}}]
\tikzstyle reverse directed=[postaction={decorate,decoration={markings,
  mark=at position 0.5 with {\arrow[arrowstyle]{{latex reversed}}}}}]

\tikzset{pointer/.style 2 args={draw,fill,single arrow,
    single arrow tip angle=45,
    single arrow head extend=#1,
    single arrow head indent=0pt,
    inner sep=0pt,
    rotate=#2}}

\definecolor{grey}{RGB}{210,210,210}
\definecolor{gammacol}{RGB}{33,120,33}
\definecolor{deltacol}{RGB}{44,90,160}
\definecolor{alphacol}{RGB}{170,0,0}
\definecolor{betacol}{RGB}{171,55,200}

\pgfmathsetmacro{\rad}{1.0}
\pgfmathsetmacro{\radplus}{1.13}
\pgfmathsetmacro{\radplusminus}{1.1}

\def\hgline[#1](#2)(#3:#4:#5)
{
     \pgfmathsetmacro{\thetaone}{#3}
	\pgfmathsetmacro{\thetatwo}{#4}
	\pgfmathsetmacro{\theta}{(\thetaone+\thetatwo)/2}
	\pgfmathsetmacro{\phi}{abs(\thetaone-\thetatwo)/2}
	\pgfmathsetmacro{\close}{less(abs(\phi-90),0.0001)}
	\ifdim \close pt = 1pt
    		\draw[#1] ($(#2)+({#5*cos(\thetaone)},{#5*sin(\thetaone)})$)   --  ($(#2)+({#5*cos(\thetatwo)},{#5*sin(\thetatwo)})$);
	\else
    		\pgfmathsetmacro{\R}{#5*tan(\phi)}
		\pgfmathsetmacro{\startangle}{\thetaone-90}
	     \pgfmathsetmacro{\finishangle}{\startangle+2*\phi-180}
		\draw[#1]($(#2)+({#5*cos(\thetaone)},{#5*sin(\thetaone)})$) arc (\startangle:\finishangle:\R);
	\fi
}

\def\dirhgline[#1](#2,#3,#4)(#5:#6:#7)
{
     \pgfmathsetmacro{\thetaone}{#5}
	\pgfmathsetmacro{\thetatwo}{#6}
	\pgfmathsetmacro{\theta}{(\thetaone+\thetatwo)/2}
	\pgfmathsetmacro{\phi}{abs(\thetaone-\thetatwo)/2}
	\pgfmathsetmacro{\close}{less(abs(\phi-90),0.0001)}
	\ifdim \close pt = 1pt
    		\draw[#1] ($(0,0)+({#7*cos(\thetaone)},{#7*sin(\thetaone)})$)   --  ($(0,0)+({#7*cos(\thetatwo)},{#7*sin(\thetatwo)})$) node[pos=#2,pointer={#3}{#4}]{};
	\else
    		\pgfmathsetmacro{\R}{#7*tan(\phi)}
		\pgfmathsetmacro{\startangle}{\thetaone-90}
	     \pgfmathsetmacro{\finishangle}{\startangle+2*\phi-180}
		\draw[#1]($(0,0)+({#7*cos(\thetaone)},{#7*sin(\thetaone)})$) arc (\startangle:\finishangle:\R) node[pos=#2,pointer={#3}{#4}]{};
	\fi
}

\newcount\recurdepth
\def\fareyrecur[#1]#2#3#4#5#6#7{
  \recurdepth=#6
  \ifnum\the\recurdepth>1\relax
    \advance\recurdepth by-1\relax
    \edef\tempnum{\number\numexpr#2+#4\relax}
    \edef\tempden{\number\numexpr#3+#5\relax}
    \pgfmathparse{\tempnum/\tempden}\edef\temp{\pgfmathresult}

    \ifnum #7=0 
      \ifnum\tempnum>0
        \node[below=1pt,scale=1]at({(\temp)},0){$\frac{\tempnum}{\tempden}$};
      \else 
        \ifnum\the\recurdepth>1\relax 
          \edef\abstempnum{\number\numexpr -\tempnum\relax}
          \node[below=1pt,scale=1]at({(\temp)},0){$-\frac{\abstempnum}{\tempden}\phantom{-}$};
        \fi 
      \fi 
      \draw[#1] ({(\temp)},0) arc (180:0:{((#4/#5)-\temp)*0.5});
      \draw[#1] ({(\temp)},0) arc (0:180:{(\temp-(#2/#3))*0.5});
    \fi

    \ifnum #7=1 
      \draw[#1] ({(\temp)},0) arc (180:0:{((#4/#5)-\temp)*0.5});
      \draw[#1] ({(\temp)},0) arc (0:180:{(\temp-(#2/#3))*0.5});
    \fi	

    \ifnum #7=2 
       \pgfmathsetmacro{\angleA}{2*atan(#2/#3)-90}
 	  \pgfmathsetmacro{\angleB}{2*atan(#4/#5)-90}
       \pgfmathsetmacro{\angleC}{2*atan(\temp)-90}
       \hgline[#1](0,0)(\angleA:\angleC:\rad);
       \hgline[#1](0,0)(-\angleC:-\angleA:\rad);
       \hgline[#1](0,0)(180+\angleA:180+\angleC:\rad);
       \hgline[#1](0,0)(180-\angleC:180-\angleA:\rad);
       \hgline[#1](0,0)(\angleC:\angleB:\rad);
       \hgline[#1](0,0)(-\angleB:-\angleC:\rad);
       \hgline[#1](0,0)(180+\angleC:180+\angleB:\rad);
       \hgline[#1](0,0)(180-\angleB:180-\angleC:\rad);
    \fi

    \begingroup
      \edef\ttempup{\noexpand\fareyrecur[#1]{\tempnum}{\tempden}{#4}{#5}{\the\recurdepth}{#7}}
      \edef\ttempdown{\noexpand\fareyrecur[#1]{#2}{#3}{\tempnum}{\tempden}{\the\recurdepth}{#7}}
      \ttempup\ttempdown
    \endgroup
  \fi
}

\def\fareygraph[#1]#2#3#4{
  \draw[#1] ({#2},0) -- ({#3+1},0);
  \foreach \n in {#2,...,#3}{
    \draw[#1] ({\n},0.75) -- ({\n},0);
    \ifnum #4=0 
      \node[below=1pt,scale=1,black] at ({\n},0) {$\frac{\n}{1}$};	
    \fi
    \draw[#1] ({(\n)},0) arc (180:0:0.5);
    \fareyrecur[#1]{\n}{1}{\n+1}{1}{4}{#4}	
  }
  \edef\nplusone{\number\numexpr#3+1\relax}
  \draw[#1] ({\nplusone},0.75) -- ({\nplusone},0) node[below=4pt,scale=1] {$\frac{\nplusone}{1}$};	
}

\def\discfareygraph[#1]#2{
  \draw[#1] (0,0) circle (\rad);
  \hgline[#1](0,0)(0:90:\rad);
  \hgline[#1](0,0)(90:180:\rad);
  \hgline[#1](0,0)(180:270:\rad);
  \hgline[#1](0,0)(270:360:\rad);
  \hgline[#1](0,0)(90:270:\rad);
  \fareyrecur[#1]{0}{1}{1}{1}{#2}{2}
}

\begin{document}

\maketitle

\begin{abstract}
Recently there has been significant progress in classifying integer friezes and $\text{SL}_2$-tilings. Typically, combinatorial methods are employed, involving triangulations of regions and inventive counting techniques. Here we develop a unified approach to such classifications using the tessellation of the hyperbolic plane by ideal triangles induced by the Farey graph. We demonstrate that the geometric, numeric and combinatorial properties of the Farey graph are perfectly suited to classifying tame $\text{SL}_2$-tilings, positive integer $\text{SL}_2$-tilings, and tame integer friezes -- both finite and infinite. In so doing, we obtain geometric analogues of certain known combinatorial models for tilings involving triangulations, and we prove several new results of a similar type too. For instance, we determine those bi-infinite sequences of positive integers that are the quiddity sequence of some positive infinite frieze, and we give a simple combinatorial model for classifying tame integer friezes, which generalises the classical construction of Conway and Coxeter for positive integer friezes. 
\end{abstract}

\section{Introduction}\label{sec1}

The genesis of the subject of $\text{SL}_2$-tilings is Coxeter's work on frieze patterns \cite{Co1971}. To describe frieze patterns, we begin with Coxeter's original motivating example, shown in Figure~\ref{fig1}.

\begin{figure}[ht]
\begingroup
\[
  \vcenter{
  \xymatrix @-0.7pc @!0 {
    && 0 && 0 && 0 && 0 && 0 && 0 && 0 && 0 && 0 && 0 && 0 && \\
        &&& 1 && 1 && 1 && 1 && 1 && 1 && 1 && 1 && 1 && 1 && 1 && \\
    && 1 && 2 && 2 && 3 && 1 && 2 && 4 && 1 && 2 && 2 && 3 &&\\
\raisebox{-4pt}{$\dotsc$}        &&& 1 && 3 && 5 && 2 && 1 && 7 && 3 && 1 && 3 && 5 && 2 && \raisebox{-4pt}{$\dotsc$}  \\
    && 2 && 1 && 7 && 3 && 1 && 3 && 5 && 2 && 1 && 7 && 3 && \\
       &&& 1 && 2 && 4 && 1 && 2 && 2 && 3 && 1 && 2 && 4 && 1 && \\
        && 1 && 1 && 1 && 1 && 1 && 1 && 1 && 1 && 1 && 1 && 1 &&& \\
    &&& 0 && 0 && 0 && 0 && 0 && 0 && 0 && 0 && 0 && 0 && 0 && \\
                        }
          }
\]
\endgroup
\caption{A positive frieze}
\label{fig1}
\end{figure}
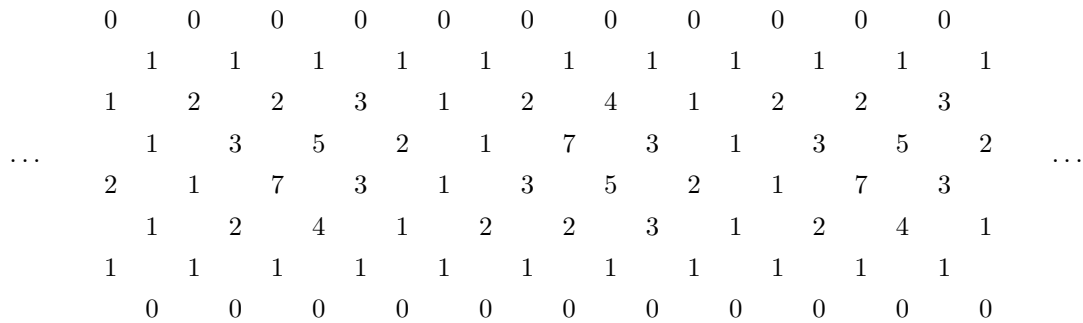

This is a bi-infinite strip of integers with bi-infinite rows of zeros at the top and bottom. Any four entries that form a diamond shape 
\[
\arraycolsep=3.6pt
\begin{matrix}
&b &\\
a && d\\
& c &
\end{matrix}
\]
satisfy the unimodular rule $ad-bc=1$. More generally, we define a \emph{positive frieze} to be a finite list of bi-infinite sequences of integers -- which we think of as the rows of an array such as Figure~\ref{fig1} -- where all entries are positive except the first and last rows of zeros, and any diamond of four integers satisfies the unimodular rule. The \emph{order} of the frieze is the number of rows minus one.

Coxeter proved in \cite{Co1971} that every positive frieze is periodic, and later he and Conway \cite{CoCo1973} classified positive friezes using \emph{triangulated polygons}. A triangulated polygon is a convex Euclidean polygon partitioned into triangles by non-crossing diagonals; an example is shown in Figure~\ref{fig2}.

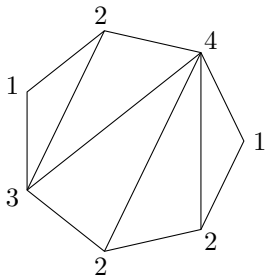
\begin{figure}[ht]
\centering
\begin{tikzpicture}
	\newdimen\R
	\R=1.5cm
	\pgfmathsetmacro{\theta}{360/7}
 	\draw (0:\R) \foreach[count=\i] \lab in {2,2,3,1,2,4,1} {
                -- (-\i*\theta:\R) node[shift=({-\i*\theta:6pt})] {$\lab$}
            }-- cycle (90:\R) ;
     \draw (-\theta:\R) -- (-6*\theta:\R);
     \draw (-6*\theta:\R) -- (-3*\theta:\R);
     \draw (-3*\theta:\R) -- (-5*\theta:\R);
      \draw (-2*\theta:\R) -- (-6*\theta:\R);
\end{tikzpicture}
\caption{The triangulated polygon corresponding to the positive frieze of Figure~\ref{fig1}}
\label{fig2}
\end{figure}

The integers at the vertices of the polygon count the number of triangles in the partition that are incident to each vertex. Reading these integers in a clockwise fashion around the polygon, starting from the rightmost vertex, we obtain the sequence $1,2,2,3,1,2,4$, which is the periodic part of the third row of Figure~\ref{fig1}. Of course, we obtain a cyclic permutation of this triangle-counting sequence by starting from a different vertex. The \emph{quiddity cycle} of a positive frieze is the cyclically-equivalent set of finite sequences that are the periodic part of the third row of the frieze. Conway and Coxeter proved that there is a one-to-one correspondence between triangulated $n$-gons and positive friezes of order $n$, in which the triangle-counting cycle of an $n$-gon corresponds to the quiddity cycle of a positive frieze.

The subject of frieze patterns has grown enormously since Conway and Coxeter's work, not least because of important connections between frieze patterns and cluster algebras; see, for example, \cite{CaCh2006}. Assem, Reutenauer and Smith \cite{AsReSm2010} introduced the concept of $\text{SL}_2$-tilings to this expanding field in their study of friezes and cluster algebras. The review article by Morier-Genoud \cite{Mo2015} discusses the significance of frieze patterns in diverse disciplines.

Our starting point is the work of Morier-Genoud, Ovsienko and Tabachnikov \cite{MoOvTa2015}, who reproved Conway and Coxeter's original result using the Farey graph and studied certain periodic $\text{SL}_2$-tilings. We take these geometric methods further, and offer a unified approach to classifying $\text{SL}_2$-tilings, bringing together recent work of Bergeron and Reutenauer on tame $\text{SL}_2$-tilings \cite{BeRe2010}, of Bessenrodt, Holm and J\o rgensen \cites{BeHoJo2017,HoJo2013} on positive integer $\text{SL}_2$-tilings, of Baur, Parsons and Tschabold  \cites{BaPaTs2016,BaPaTs2018,Ts2019} on infinite friezes, of Morier-Genoud, Ovsienko and Tabachnikov on antiperiodic tilings \cite{MoOvTa2015}, and of Ovsienko \cite{Ov2018} on tame friezes with positive quiddity sequences. We demonstrate that the geometric, numeric and combinatorial properties of the Farey graph provide deep insight into $\text{SL}_2$-tilings.

This extended introduction describes the full suite of results for this classification programme, illustrating the theorems with examples and figures.

Fundamental to this approach is the special linear group of degree 2 over the integers, namely
\[
\text{SL}_2(\mathbb{Z}) = \left\{ \begin{pmatrix}a&b\\ c& d\end{pmatrix} : a,b,c,d\in\mathbb{Z},\, ad-bc=1\right\}.
\]
Informally speaking, an $\text{SL}_2$-tiling is a bi-infinite matrix such that any two-by-two submatrix belongs to $\text{SL}_2(\mathbb{Z})$. An example of an $\text{SL}_2$-tiling is shown in Figure~\ref{fig3}. 

\begin{figure}[ht]
\centering
\begingroup
$
  \vcenter{
  \xymatrix @-0.2pc @!0 {
          & &  &  & \vdots &  &  & &\\
         & 15 & 11& 7 & 3 & 5 & 7 & 9 & \\
   & -11\phantom{-} &  -8\phantom{-} & -5\phantom{-} & -2\phantom{-} & -3\phantom{-} & -4\phantom{-}  & -5\phantom{-} &\\
    &  7    &   5  &	3 & 1 & 1 & 1  & 1 &  \\
\dotsb  & -3\phantom{-} &   -2\phantom{-}  & -1\phantom{-} & 0 & 1 & 2  &  3 &\dotsb \\
       & -1\phantom{-}   &  -1\phantom{-}  &  -1\phantom{-} & -1\phantom{-} & -3\phantom{-} & -5\phantom{-}  & -7\phantom{-} &\\
      &5    &  4  &  3 & 2 & 5 & 8 & 11 &\\
      & -9\phantom{-} & -7\phantom{-}& -5\phantom{-} & -3\phantom{-} & -7\phantom{-} & -11\phantom{-} & -15\phantom{-} & \\
       &  &  &  &   \vdots &  &  & & 
            }
          }
$
\endgroup
\caption{An $\text{SL}_2$-tiling}
\label{fig3}
\end{figure}

Now we give a more precise definition of an $\text{SL}_2$-tiling, and to this end we consider functions $\mathbf{M}\colon \mathbb{Z}\times\mathbb{Z}\longrightarrow\mathbb{Z}$, which represent bi-infinite matrices. It simplifies our presentation to write $m_{i,j}$ for the entry $\mathbf{M}(i,j)$, for~$i,j\in\mathbb{Z}$.

\begin{definition}
An \emph{$\textnormal{SL}_2$-tiling} is a function $\mathbf{M}\colon \mathbb{Z}\times\mathbb{Z}\longrightarrow\mathbb{Z}$ such that
\[
\begin{pmatrix}m_{i,j} & m_{i,j+1}\\ m_{i+1,j} & m_{i+1,j+1} \end{pmatrix}\in \text{SL}_2(\mathbb{Z}),
\]
for $i,j\in\mathbb{Z}$.
\end{definition}

We think of $m_{i,j}$ as the entry of $\mathbf{M}$ in the $i$th row and $j$th column, with the row index increasing downwards and the column index increasing from left to right -- the usual conventions for matrices.

We restrict our attention to a class of $\text{SL}_2$-tilings called \emph{tame} tilings, which possess rigidity properties essential to our approach.

\begin{definition}
An $\textnormal{SL}_2$-tiling $\mathbf{M}$ is \emph{tame} if
\[
\det\begin{pmatrix}m_{i,j} & m_{i,j+1} & m_{i,j+2}\\ m_{i+1,j} & m_{i+1,j+1} & m_{i+1,j+2}\\  m_{i+2,j} & m_{i+2,j+1} & m_{i+2,j+2}\end{pmatrix}=0,
\]
for $i,j\in\mathbb{Z}$.
\end{definition}

Our principal tool for classifying $\text{SL}_2$-tilings is the \emph{Farey graph}. This is an infinite graph embedded in the extended complex plane $\mathbb{C}\cup\{\infty\}$. In future we write $\mathbb{C}_\infty$ for $\mathbb{C}\cup\{\infty\}$ (and use similar notation for $\mathbb{R}_\infty$ and $\mathbb{Q}_\infty$). To describe the Farey graph precisely, we define $\mathbb{H}$ to be the (open) upper half-plane, which is a model of the hyperbolic plane when it is endowed with the Riemannian metric $|dz|/\text{Im}\, z$ (however, hyperbolic geometry will not feature explicitly in our arguments). Each hyperbolic line in this model of the hyperbolic plane is either the upper half of a circle centred on the real axis or a half-line in $\mathbb{H}$ orthogonal to the real axis. Hyperbolic lines of the former type join one real number (on the boundary of $\mathbb{H}$) to another, and hyperbolic lines of the latter type join a real number to $\infty$.

We use the term \emph{reduced rational} to describe an expression $a/b$ in which $a$ and $b$ are coprime integers. Each rational number can be expressed as a reduced rational in precisely two ways; for example, the rational number $2/3$ can also be written as $(-2)/(-3)$. For convenience, we consider the expressions $1/0$ and $(-1)/0$ to be reduced rationals, both representing the point $\infty$.

The Farey graph $\mathscr{F}$ is the graph with vertices $\mathbb{Q}_\infty$ and edges comprising those pairs $a/b$ and $c/d$ of reduced rationals for which $ad-bc=\pm 1$. We represent the edge incident to $a/b$ and $c/d$ by the unique hyperbolic line in $\mathbb{H}$ between those two boundary points. The collection of all edges creates a tessellation of $\mathbb{H}$ by triangles, part of which is shown in Figure~\ref{fig4}.

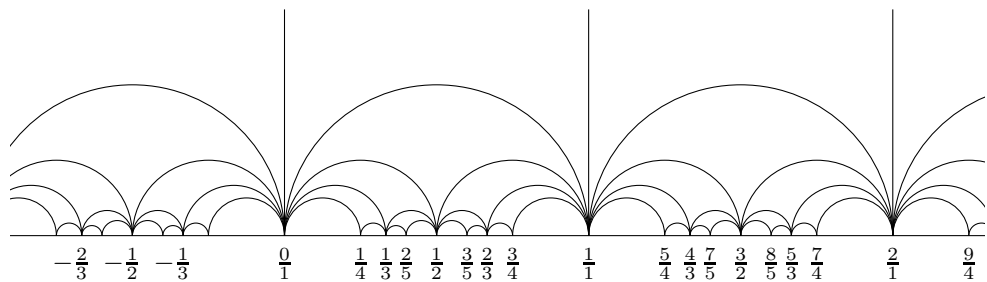
\begin{figure}[ht]
\centering
\begin{tikzpicture}[scale=4]
\begin{scope}
    \clip(-0.9,-0.2) rectangle (2.3,0.75);
	\fareygraph[thin]{-1}{2}{0}
\end{scope}
\end{tikzpicture}
\caption{Part of the Farey graph}
\label{fig4}
\end{figure}

Next we discuss an action of $\text{SL}_2(\mathbb{Z})$ on the Farey graph. Given $A\in \text{SL}_2(\mathbb{Z})$ and $z\in\mathbb{C}_\infty$, we define
\[
A(z) = \frac{az+b}{cz+d}, \quad\text{where}\quad A=\begin{pmatrix}a& b\\ c&d\end{pmatrix},
\]
with the usual conventions concerning $\infty$. This formula specifies an action of $\text{SL}_2(\mathbb{Z})$ on $\mathbb{C}_\infty$ in the standard way. It also specifies an action of $\text{SL}_2(\mathbb{Z})$ on the extended real line $\mathbb{R}_\infty$, on the extended rationals $\mathbb{Q}_\infty$, and on the upper half-plane $\mathbb{H}$ as a group of orientation-preserving hyperbolic isometries. The action is not faithful because $A$ and $-A$ determine the same map; however, the action of the quotient group $\text{PSL}_2(\mathbb{Z})$ (the modular group) is faithful.

Suppose now that $u_1,u_2,v_1,v_2$ are integers that satisfy $u_1v_2-u_2v_1=1$. Let
\[
\begin{pmatrix}u_1' \\ u_2'\end{pmatrix}=A\begin{pmatrix}u_1 \\ u_2\end{pmatrix}
\quad\text{and}\quad
\begin{pmatrix}v_1' \\ v_2'\end{pmatrix}=A\begin{pmatrix}v_1 \\ v_2\end{pmatrix}.
\]
Then 
\[
\begin{pmatrix}u_1' & v_1'\\ u_2' & v_2'\end{pmatrix}=A\begin{pmatrix}u_1 & v_1\\ u_2 & v_2\end{pmatrix}\in\text{SL}_2(\mathbb{Z}),
\]
so $u_1'v_2'-u_2'v_1'=1$. In particular, if $u=u_1/u_2$ and $v=v_1/v_2$ are adjacent vertices of the Farey graph $\mathscr{F}$, and $A\in\text{SL}_2(\mathbb{Z})$, then $A(u)$ and $A(v)$ are also adjacent vertices of $\mathscr{F}$. Consequently, we see that $\text{SL}_2(\mathbb{Z})$ acts on $\mathscr{F}$ too, and in fact each element of $\text{SL}_2(\mathbb{Z})$ induces a graph automorphism of $\mathscr{F}$.

Given two vertices $u$ and $v$ of $\mathscr{F}$, we write $u\sim v$ if $u$ and $v$ are adjacent. A \emph{bi-infinite path} in $\mathscr{F}$ is a sequence $\dotsc,v_{-1},v_0,v_1,v_2,\dotsc$ of vertices of $\mathscr{F}$ such that $v_i\sim v_{i+1}$, for $i\in\mathbb{Z}$. We denote this path by $\langle\,\dotsc,v_{-1},v_0,v_1,v_2,\dotsc\rangle$. We also consider \emph{finite paths}, defined in the obvious way, with similar notation. Each vertex $v_i$ can be written as a reduced rational in exactly two ways, and it is possible to choose such representations $v_i=a_i/b_i$, for $i\in\mathbb{Z}$, such that $a_ib_{i+1}-a_{i+1}b_i=1$, for $i\in\mathbb{Z}$. In fact, there are exactly two ways of doing this; the other way is $v_i=(-a_i)/(-b_i)$, for $i\in\mathbb{Z}$.

In this paper we use paths in $\mathscr{F}$ to classify $\text{SL}_2$-tilings. We now describe the details of this correspondence.

Evidently, if $\mathbf{M}$ is an $\text{SL}_2$-tiling, then so too is the bi-infinite matrix $-\mathbf{M}$, obtained by taking the negative of each of the entries of $\mathbf{M}$. For our purposes, it is helpful to identify these two tilings, as follows. We define $\SL$ to be the quotient set of the set of all tame $\text{SL}_2$-tilings  by the equivalence relation that identifies $\mathbf{M}$ and $-\mathbf{M}$. And we write $\pm \mathbf{M}$ for an element of $\SL$, where $\mathbf{M}$ is understood to be a tame $\text{SL}_2$-tiling.

Next, consider the collection $\PATHP$ of all bi-infinite paths in $\mathscr{F}$. We will define a function $\widetilde{\Phi}$ from $\PATHP\times\PATHP$ to $\SL$. Consider a pair of bi-infinite paths $(\gamma,\delta)$. The vertices of $\gamma$ are reduced rationals $a_i/b_i$, for $i\in\mathbb{Z}$, and the vertices of $\delta$ are reduced rationals $c_j/d_j$, for $j\in\mathbb{Z}$. We choose these representations such that $a_ib_{i+1}-a_{i+1}b_i=1$ and $c_jd_{j+1}-c_{j+1}d_j=1$, for $i,j\in\mathbb{Z}$. Let us now define $\mathbf{M}\colon \mathbb{Z}\times \mathbb{Z}\longrightarrow \mathbb{Z}$ by 
\[
m_{i,j}=a_id_j-b_ic_j,
\]
where, as usual, $m_{i,j}=\mathbf{M}(i,j)$. This is the fundamental formula underlying our approach; it is much the same as the formula used in \cite[Theorem~3]{MoOvTa2015}, but without the modulus signs. One can check that $\mathbf{M}$ is a tame $\text{SL}_2$-tiling (we do so in Section~\ref{sec2}), and we define $\widetilde{\Phi}(\gamma,\delta)=\pm\mathbf{M}$.

Notice that switching from $a_i/b_i$ to $(-a_i)/(-b_i)$ for all the vertices of $\gamma$ changes the corresponding $\text{SL}_2$-tiling $\mathbf{M}$ to $-\mathbf{M}$. This justifies the need to identify $\pm\mathbf{M}$ in $\SL$.

Let us consider an example, using the bi-infinite paths
\[
\gamma = \left\langle\,\dotsc ,-2,-1,0,1,2,\dotsc \right\rangle 
\quad\text{and}\quad
\delta = \left\langle\,\dotsc ,\tfrac34,\tfrac23,\tfrac12,0,-\tfrac12,-\tfrac23,-\tfrac34,\dotsc \right\rangle
\]
illustrated in Figure~\ref{fig5}. The path $\gamma$ has vertices $a_i/b_i$, and $\delta$ has vertices $c_j/d_j$, for $i,j\in\mathbb{Z}$, where $a_i=(-1)^ii$, $b_i=(-1)^i$, $c_j=-j$ and $d_j=|j|+1$. The signs of these numbers have been chosen carefully to ensure that $a_ib_{i+1}-a_{i+1}b_i=1$ and $c_jd_{j+1}-c_{j+1}d_j=1$, for $i,j\in\mathbb{Z}$. Then the corresponding matrix $\mathbf{M}$ is that given in Figure~\ref{fig3}, with entries
\[
m_{i,j} = (-1)^i(i(|j|+1)+j),
\]
for $i,j\in\mathbb{Z}$. In general, each $0$ entry of $\mathbf{M}$ corresponds to an intersection point of $\gamma$ and $\delta$; in this particular case there is only one such point.

\begin{figure}[ht]
\centering
\begin{tikzpicture}[scale=4]
\begin{scope}
    \clip(-1.6,-0.2) rectangle (1.6,0.75);
    \fareygraph[thin,grey]{-2}{2}{1}

	\foreach \n in {-2,...,1}{
		\draw[gammacol,thick] (\n,0) arc (180:0:0.5)node[pos=0.5,pointer={1.0}{0}]{};
	}
	\foreach \n in {1,...,10}{
		\ifnum \n < 3 
			\draw[deltacol,thick] ({\n/(\n+1)},0) arc (0:180:{0.5/(\n*(\n+1))}) node[pos=0.5,pointer={{1.0/(1*\n)}}{180}]{};
			\draw[deltacol,thick] ({-(\n-1)/\n},0) arc (0:180:{0.5/(\n*(\n+1))}) node[pos=0.5,pointer={{1.0/(1*\n)}}{180}]{};
		\else
			\draw[deltacol,thick] ({\n/(\n+1)},0) arc (0:180:{0.5/(\n*(\n+1))}); 
			\draw[deltacol,thick] ({-(\n-1)/\n},0) arc (0:180:{0.5/(\n*(\n+1))});
		\fi
	}
	
    \node[gammacol,below] at (-1,0) {$-1\phantom{-}$};  	
    \node[below] at (0,0) {$0$};
    \node[gammacol,below] at (1,0) {$1$};
    \node[deltacol,below] at (0.5,0) {$\tfrac12$};
    \node[deltacol,below] at (2/3,0) {$\tfrac23$};
    \node[deltacol,below] at (3/4,0) {$\tfrac34$};
    \node[deltacol,below] at (-0.5,0) {$-\tfrac12\phantom{-}$};
    \node[deltacol,below] at (-2/3,0) {$-\tfrac23\phantom{-}$};
    
    \node[gammacol,above] at (0.5,0.5) {$\gamma$};
    \node[deltacol,above] at (-0.25,0.25) {$\delta$};
\end{scope}
\end{tikzpicture}
\caption{The paths $\gamma = \left\langle\,\dotsc ,-2,-1,0,1,2,\dotsc \right\rangle$  and $\delta = \left\langle\,\dotsc ,\tfrac34,\tfrac23,\tfrac12,0,-\tfrac12,-\tfrac23,-\tfrac34,\dotsc \right\rangle$}
\label{fig5}
\end{figure}

Since the image of a bi-infinite path under an element of $\text{SL}_2(\mathbb{Z})$ is also a bi-infinite path, we see that $\text{SL}_2(\mathbb{Z})$ acts on $\PATHP$, and so it acts on $\PATHP\times\PATHP$ too, by the formula $A(\gamma,\delta)=(A\gamma,A\delta)$, for $A\in\text{SL}_2(\mathbb{Z})$ and $\gamma,\delta\in\PATHP$. Let $\pi\colon \PATHP\times\PATHP\longrightarrow(\PATHP\times\PATHP)/\textnormal{SL}_2(\mathbb{Z})$ be the quotient map.

In Section~\ref{sec2} we will verify that $\widetilde{\Phi}(A\gamma,A\delta)=\widetilde{\Phi}(\gamma,\delta)$, for $A\in\text{SL}_2(\mathbb{Z})$ and $\gamma,\delta\in\PATHP$. It follows that $\widetilde{\Phi}$ induces a map $\Phi$ from $(\PATHP\times\PATHP)/\textnormal{SL}_2(\mathbb{Z})$ to $\SL$ such that the following diagram commutes.

\begin{center}
\begin{tikzpicture}
  \matrix (m) [matrix of math nodes,row sep=3em,column sep=4em,minimum width=2em]
  {
     \PATHP\times\PATHP & \SL \\
     & \\
     (\PATHP\times\PATHP)/\text{SL}_2(\mathbb{Z}) &  \\};
  \path[-stealth]
    (m-1-1) edge node [left] {$\pi$} (m-3-1)
            edge node [above] {$\widetilde{\Phi}$} (m-1-2)
    (m-3-1) edge node [below right] {$\Phi$} (m-1-2);
\end{tikzpicture}
\end{center}

This brings us to our first theorem.

\begin{theorem}\label{thm1}
The map $\Phi\colon (\PATHP\times\PATHP)/\textnormal{SL}_2(\mathbb{Z})\longrightarrow \SL $ is a one-to-one correspondence.
\end{theorem}

Let us briefly (and informally) discuss the inverse function of $\Phi$; a more formal definition will follow later. Consider a tame $\text{SL}_2$-tiling $\mathbf{M}$. Choose any two consecutive rows of $\mathbf{M}$, and form a bi-infinite sequence of rational numbers by dividing each entry in the top row by the corresponding entry in the bottom row. This sequence is a bi-infinite path in the Farey graph $\mathscr{F}$, and another bi-infinite path can be obtained by choosing two columns instead of rows. After adjusting one of the two paths by applying a suitable element of $\text{SL}_2(\mathbb{Z})$, we obtain a pair $(\gamma,\delta)$ of bi-infinite paths in $\mathscr{F}$ that satisfy $\widetilde{\Phi}(\gamma,\delta)=\pm \mathbf{M}$.

The characterisation of tame $\text{SL}_2$-tilings provided by Theorem~\ref{thm1} is comparable to that of \cite{BeRe2010}*{Proposition~3} (in two dimensions), but framed in the context of the Farey graph. The significance of Theorem~\ref{thm1} is that it gives us geometric insight into the structure of tame $\text{SL}_2$-tilings, and by restricting $\widetilde{\Phi}$ to subcollections of $\PATHP\times\PATHP$ we can obtain precise descriptions of various classes of tilings.

For a first application of this procedure, we classify the set $\SL^+$ of positive $\text{SL}_2$-tilings ($\text{SL}_2$-tilings with positive integer entries), all of which are necessarily tame. A classification of positive $\text{SL}_2$-tilings has appeared before, using combinatorial techniques, in \cite{BeHoJo2017}. Later (in Section~\ref{sec9}) we will demonstrate that the geometric model of positive $\text{SL}_2$-tilings that we obtain using the Farey graph gives rise to the combinatorial model of \cite{BeHoJo2017}.

The set $\SL^+$ can be identified with a subset of $\SL$ in the obvious way (under which a positive $\text{SL}_2$-tiling $\mathbf{M}$ is identified with $\pm\mathbf{M}$).

For the next definition we identify $\mathbb{R}_\infty$ with the unit circle $\mathbb{S}$ in the complex plane using the modified Cayley transform $z\longmapsto (iz+1)/(z+i)$, which maps $\mathbb{R}_\infty$ to $\mathbb{S}$, and sends $\infty$ to $i$, $0$ to $-i$, $1$ to $1$, and $-1$ to $-1$.

\begin{definition}
A bi-infinite sequence of distinct points $\dotsc,v_{-1},v_0,v_1,v_2,\dotsc$ in $\mathbb{R}_\infty$ is in \emph{clockwise order} if the corresponding sequence of points $\dotsc,\overline{v}_{-1},\overline{v}_0,\overline{v}_1,\overline{v}_2,\dotsc$ in $\mathbb{S}$ is such that $\overline{v}_i,\overline{v}_j,\overline{v}_k$ is in clockwise order around $\mathbb{S}$, for all $i,j,k\in\mathbb{Z}$ with $i<j<k$.

A bi-infinite path $\gamma=\langle\, \dotsc,v_{-1},v_0,v_1,v_2,\dotsc\rangle$ in $\mathscr{F}$ is a \emph{clockwise} bi-infinite path if the sequence $\dotsc,v_{-1},v_0,v_1,v_2,\dotsc$ is in clockwise order. 
\end{definition}

The definition says, informally speaking, that a clockwise bi-infinite path traverses $\mathbb{R}_\infty$ clockwise and does not complete more than one full cycle. We define clockwise order for finite and half-infinite sequences, and clockwise finite and half-infinite paths, in a similar way. 

Each clockwise bi-infinite path $\gamma$ converges (in $\mathbb{R}_\infty$) in the backward direction to a point $\gamma_{-\infty}$ and in the forward direction to a point $\gamma_\infty$. These are called the \emph{backward limit} and \emph{forward limit} of $\gamma$, respectively, and they may be equal. Consider the collection of all pairs of clockwise bi-infinite paths $(\gamma,\delta)$ for which $\gamma_{-\infty},\gamma_\infty,\delta_{-\infty},\delta_\infty$ is in clockwise order, where possibly $\gamma_\infty=\delta_{-\infty}$ and possibly $\gamma_{-\infty}=\delta_\infty$ (but $\gamma_\infty\neq \gamma_{-\infty}$ and $\delta_\infty\neq \delta_{-\infty}$). Let $(\PATHP\times\PATHP)^+$ denote the collection of all such pairs. Observe that $\text{SL}_2(\mathbb{Z})$ acts on $(\PATHP\times\PATHP)^+$. 

By restricting the function $\Phi$ to $(\PATHP\times\PATHP)^+/\text{SL}_2(\mathbb{Z})$ we obtain the following result.

\begin{theorem}\label{thm2}
The map $\Phi\colon (\PATHP\times\PATHP)^+/\textnormal{SL}_2(\mathbb{Z})\longrightarrow \SL^+ $ is a one-to-one correspondence.
\end{theorem}

We demonstrate the correspondence of this theorem with two examples, illustrated in Figure~\ref{fig6}. Each of the two subfigures displays a copy of the Farey graph mapped onto the unit disc by the modified Cayley transform $z\longmapsto (iz+1)/(z+i)$, with vertex labels from $\mathbb{R}_\infty$. A pair of  clockwise bi-infinite paths is highlighted in each Farey graph.

\begin{figure}[ht]
\centering
\begin{tikzpicture}[scale=2.5]

\begin{scope}
\discfareygraph[thin,grey]{4};

\node at ({2*atan(2)-90}:\radplusminus) {$2$};
\node at ({-2*atan(2)-90}:\radplus) {$-2$};
\node at ({2*atan(1/2)-90}:\radplusminus) {$\tfrac12$};
\node at ({-2*atan(1/2)-90}:\radplus) {$-\tfrac12$};

\node[gammacol] at (0:\radplusminus) {$1$};
\node[gammacol] at ({2*atan(2/3)-90}:\radplusminus) {$\tfrac23$};
\node[gammacol] at ({2*atan(3/2)-90}:\radplusminus) {$\tfrac32$};

\node[deltacol] at (180:\radplus) {$-1$};
\node[deltacol] at ({-2*atan(2/3)-90}:\radplus) {$-\tfrac23$};
\node[deltacol] at ({-2*atan(3/2)-90}:\radplus) {$-\tfrac32$};

\foreach \n in {1,...,12}{
	\hgline[gammacol,thick](0,0)(2*atan((2*\n-1)/\n)-90:2*atan((2*\n+1)/(\n+1))-90:\rad);
	\hgline[gammacol,thick](0,0)(2*atan((\n+1)/(2*\n+1))-90:2*atan(\n/(2*\n-1))-90:\rad);
	\hgline[deltacol,thick](0,0)(2*atan((2*\n-1)/\n)+90:2*atan((2*\n+1)/(\n+1))+90:\rad);
	\hgline[deltacol,thick](0,0)(2*atan((\n+1)/(2*\n+1))+90:2*atan(\n/(2*\n-1))+90:\rad);
}

\dirhgline[gammacol,draw=none,thick](0.535,0.7,274)(0:2*atan(2/3)-90:\rad);
\dirhgline[gammacol,draw=none,thick](0.45,0.7,270)(2*atan(2/3)-90:0:\rad);
\dirhgline[deltacol,draw=none,thick](0.535,0.7,94)(180:2*atan(2/3)+90:\rad);
\dirhgline[deltacol,draw=none,thick](0.45,0.7,86)(2*atan(2/3)+90:180:\rad);

\node[gammacol] at (22:0.8*\rad) {$\gamma$};
\node[deltacol] at (202:0.8*\rad) {$\delta$};

\node at (-1.15,-1.15) {(a)}; 
\end{scope}

\begin{scope}[xshift=3cm]
\discfareygraph[thin,grey]{4};

\node[] at (0:\radplusminus) {$1$};
\node[] at (180:\radplus) {$-1$};

\node[gammacol] at (90:\radplusminus) {$\infty$};
\node[gammacol] at ({2*atan(2)-90}:\radplusminus) {$2$};
\node[gammacol] at ({-2*atan(2)-90}:\radplus) {$-2$};
\node[gammacol] at ({2*atan(3/2)-90}:\radplusminus) {$\tfrac32$};
\node[gammacol] at ({-2*atan(3/2)-90}:\radplus) {$-\tfrac32$};

\node[deltacol] at (-90:\radplusminus) {$0$};
\node[deltacol] at ({2*atan(1/2)-90}:\radplusminus) {$\tfrac12$};
\node[deltacol] at ({-2*atan(1/2)-90}:\radplus) {$-\tfrac12$};
\node[deltacol] at ({2*atan(2/3)-90}:\radplusminus) {$\tfrac23$};
\node[deltacol] at ({-2*atan(2/3)-90}:\radplus) {$-\tfrac23$};

\foreach \n in {1,...,20}{
	\hgline[deltacol,thick](0,0)(2*atan((\n-1)/\n)-90:2*atan(\n/(\n+1))-90:\rad);
	\hgline[deltacol,thick](0,0)(-2*atan(\n/(\n+1))-90:-2*atan((\n-1)/\n)-90:\rad);
	\hgline[gammacol,thick](0,0)(2*atan((\n-1)/\n)+90:2*atan(\n/(\n+1))+90:\rad);
	\hgline[gammacol,thick](0,0)(-2*atan(\n/(\n+1))+90:-2*atan((\n-1)/\n)+90:\rad);
}

\dirhgline[deltacol,draw=none,thick](0.5,0.8,204)(-90:2*atan(1/2)-90:\rad);
\dirhgline[deltacol,draw=none,thick](0.5,0.4,240)(2*atan(1/2)-90:2*atan(2/3)-90:\rad);
\dirhgline[deltacol,draw=none,thick](0.48,0.8,155)(-2*atan(1/2)-90:-90:\rad);
\dirhgline[deltacol,draw=none,thick](0.5,0.4,122)(-2*atan(2/3)-90:-2*atan(1/2)-90:\rad);

\dirhgline[gammacol,draw=none,thick](0.48,0.8,25)(90:2*atan(2)-90:\rad);
\dirhgline[gammacol,draw=none,thick](0.5,0.4,-60)(2*atan(3/2)-90:2*atan(2)-90:\rad);
\dirhgline[gammacol,draw=none,thick](0.5,0.8,-24)(90:-2*atan(2)-90:\rad);
\dirhgline[gammacol,draw=none,thick](0.5,0.4,58)(-2*atan(2)-90:-2*atan(3/2)-90:\rad);

 \node[deltacol] at (-50:0.56*\rad) {$\delta$};
 \node[gammacol] at (130:0.56*\rad) {$\gamma$};

\node at (-1.15,-1.15) {(b)}; 
\end{scope}

\end{tikzpicture}
\caption{Two pairs of clockwise bi-infinite paths in the Farey graph}
\label{fig6}
\end{figure}
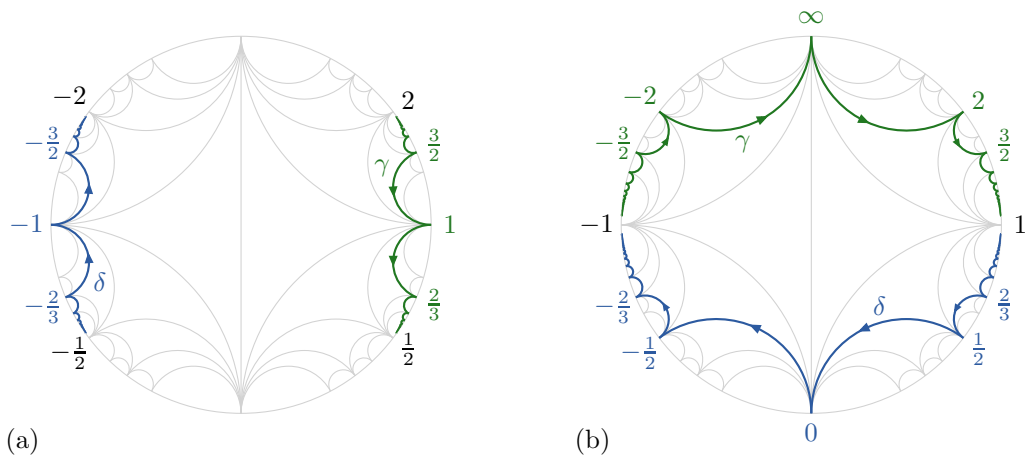

The paths shown in Figure~\ref{fig6}(a) are
\[
\gamma = \left\langle\,\dotsc ,\tfrac74,\tfrac53,\tfrac32,1,\tfrac23,\tfrac35,\tfrac47,\dotsc \right\rangle 
\quad\text{and}\quad
\delta = \left\langle\,\dotsc ,-\tfrac47,-\tfrac35,-\tfrac23,-1,-\tfrac32,-\tfrac53,-\tfrac74,\dotsc \right\rangle. 
\]
The path $\gamma$ has vertices $a_i/b_i$ and $\delta$ has vertices $c_j/d_j$, for $i,j\in\mathbb{Z}$, where $a_i=\tfrac12(3|i|-i+2)$, $b_i=\tfrac12(3|i|+i+2)$, $c_j=-\tfrac12(3|j|+j+2)$ and $d_j=\tfrac12(3|j|-j+2)$. Here we have $a_ib_{i+1}-a_{i+1}b_i=1$ and $c_jd_{j+1}-c_{j+1}d_j=1$, for $i,j\in\mathbb{Z}$. The limits of the two sequences are $\gamma_{-\infty}=2$, $\gamma_\infty=\tfrac12$, $\delta_{-\infty}=-\tfrac12$ and $\delta_\infty=-2$, and the corresponding positive $\text{SL}_2$-tiling $\mathbf{M}$ has entries 
\[
m_{i,j} = a_id_j-b_ic_j=\tfrac12\left((3|i|+2)(3|j|+2)+ij\right),
\]
for $i,j\in\mathbb{Z}$. This tiling is illustrated in Figure~\ref{fig7}(\subref{fig5a}). 

The paths shown in Figure~\ref{fig6}(b) are
\[
\gamma = \left\langle\,\dotsc ,-\tfrac43,-\tfrac32,-2,\infty,2,\tfrac32,\tfrac43,\dotsc \right\rangle
\quad\text{and}\quad
\delta = \left\langle\,\dotsc ,\tfrac34,\tfrac23,\tfrac12,0,-\tfrac12,-\tfrac23,-\tfrac34,\dotsc \right\rangle . 
\]
The path $\gamma$ has vertices $a_i/b_i$ and $\delta$ has vertices $c_j/d_j$, for $i,j\in\mathbb{Z}$, where $a_i=|i|+1$, $b_i=i$, $c_j=-j$ and $d_j=|j|+1$, so $a_ib_{i+1}-a_{i+1}b_i=1$ and $c_jd_{j+1}-c_{j+1}d_j=1$. The limits of the two sequences are $\gamma_{-\infty}=\delta_\infty=-1$ and $\gamma_\infty=\delta_{-\infty}=1$, and the corresponding positive $\text{SL}_2$-tiling $\mathbf{M}$ has entries 
\[
m_{i,j} = (|i|+1)(|j|+1)+ij,
\]
for $i,j\in\mathbb{Z}$. This tiling is illustrated in Figure~\ref{fig7}(\subref{fig5b}). 

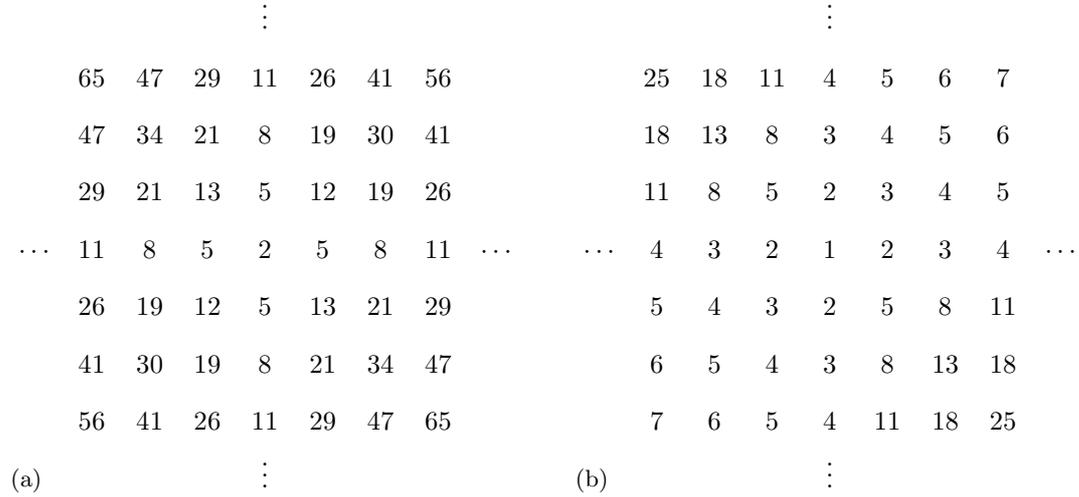
\begin{figure}[ht]
\begin{subfigure}[b]{0.5\textwidth}
\begingroup
$
  \vcenter{
  \xymatrix @-0.2pc @!0 {
          & &  &  & \vdots &  &  & &\\
         & 65 & 47& 29 & 11 & 26 & 41 & 56 & \\
   & 47 &  34 & 21 & 8 & 19 & 30  & 41 &\\
    &  29    &   21  &	13 & 5 & 12 & 19  & 26 &  \\
\dotsb  & 11 &   8  & 5 & 2 & 5 & 8  &  11 &\dotsb \\
       & 26   &  19  &  12 & 5 & 13 & 21  & 29 &\\
      &41    &  30  &  19 & 8 & 21 & 34 & 47 &\\
      & 56 & 41 & 26 & 11 & 29 & 47 & 65 & \\
       &  &  &  &   \vdots &  &  & & 
            }
          }
$
\endgroup
\vspace*{-0.6cm}
\caption{}
\label{fig5a}
\end{subfigure}
\begin{subfigure}[b]{0.5\textwidth}
\begingroup
$
  \vcenter{
  \xymatrix @-0.2pc @!0 {
          & &  &  & \vdots &  &  & &\\
         & 25 & 18& 11 & 4 & 5 & 6 & 7 & \\
   & 18 &  13 & 8 & 3 & 4 & 5  & 6 &\\
    &  11    &   8  &	5 & 2 & 3 & 4  &5 &  \\
\dotsb  & 4 &   3  & 2 & 1 & 2 & 3  &  4 &\dotsb \\
       & 5   &  4  &  3 & 2 & 5 & 8  & 11 &\\
      &6    &  5  &  4 & 3 & 8 & 13 & 18 &\\
      & 7 & 6& 5 & 4 & 11 & 18 & 25 & \\
       &  &  &  &   \vdots &  &  & & 
            }
          }
$
\endgroup
\vspace*{-0.6cm}
\caption{}
\label{fig5b}
\end{subfigure}
\caption{Positive $\text{SL}_2$-tilings corresponding to the two pairs of paths of Figure~\ref{fig6}}
\label{fig7}
\end{figure}

To analyse these two tilings, we define a function $\Delta\colon \mathbb{Q}_{\infty}\times\mathbb{Q}_{\infty} \longrightarrow \mathbb{Z}$ by
\[
\Delta(a/b,c/d)=|ad-bc|, 
\]
where $a/b$ and $c/d$ are reduced rationals. This function has been considered before in a related context, in \cite{MoOvTa2015}. It is not a metric on $\mathbb{Q}_{\infty}$. Evidently $\Delta(x,y)=0$ if and only if $x=y$, and $\Delta(x,y)=1$ if and only if $x$ and $y$ are adjacent vertices of $\mathscr{F}$. Furthermore, if $(\gamma,\delta)\in(\mathscr{P}\times\mathscr{P})^+$, where $\gamma$ has vertices $a_i/b_i$ and $\delta$ has vertices $c_j/d_j$, for $i,j\in\mathbb{Z}$, then we will see in Section~\ref{sec3} that $m_{i,j}=a_id_j-b_ic_j$ is positive, so $m_{i,j}=\Delta(a_i/b_i,c_j/d_j)$.

Now, in Figure~\ref{fig7}(\subref{fig5a}) the (unique) smallest entry is 2. Accordingly, 2 is the least value taken by $\Delta(x,y)$ for vertices $x$ of $\gamma$ and $y$ of $\delta$, where $\gamma$ and $\delta$ are the paths shown in Figure~\ref{fig6}(a). This value is achieved when $x=1$ and $y=-1$, and it is achieved at no other two vertices, one from $\gamma$ and the other from $\delta$ (for a proof of a more general result of this type, see Theorem~\ref{thmJ}).

Consider now the paths $\gamma$ and $\delta$ of Figure~\ref{fig6}(b) and the corresponding tiling of Figure~\ref{fig7}(\subref{fig5b}). These two paths come within a $\Delta$-distance of $1$ from each other at the vertices $0$ and $\infty$, so the corresponding tiling has an entry 1. Again, this is the unique smallest entry, because the $\Delta$-distance between any other two vertices, one from $\gamma$ and the other from $\delta$, is greater than 1. 

The paths $\gamma$ and $\delta$ of Figure~\ref{fig6}(b) are particularly special because they satisfy $\gamma_\infty=\delta_{-\infty}$ and $\gamma_{-\infty}=\delta_\infty$. Moreover, both these points are \emph{rational} numbers. We shall see later that in the special cases when $\gamma_\infty=\delta_{-\infty}$ or $\gamma_{-\infty}=\delta_\infty$, the rationality or irrationality of each of these numbers is highly significant. For instance, if one of them is irrational, then there are infinitely many entries 1 in the corresponding $\text{SL}_2$-tiling (see Theorem~\ref{thm88}). In Section~\ref{sec9} we discuss these properties further, and relate Theorem~\ref{thm2} to the combinatorial models used to classify positive $\text{SL}_2$-tilings in \cites{BeHoJo2017,HoJo2013}.

Let us move on to another example of application of Theorem~\ref{thm1}, this time of \emph{infinite friezes}, studied by Baur, Parsons and Tschabold in \cites{BaPaTs2016,BaPaTs2018} and Tschabold in \cite{Ts2019}. We encountered Coxeter's frieze patterns earlier; infinite friezes are similar types of integer arrays, but instead of a border row of zeros at the bottom they continue downwards indefinitely. An example is shown in Figure~\ref{fig8}. 

\begin{figure}[ht]
\begingroup
\[
  \vcenter{
  \xymatrix @-0.5pc @!0 {
    && 0 && 0 && 0 && 0 && 0 && 0 && 0 & \\
    &&& 1 && 1 && 1 && 1 && 1 && 1 && 1 && \\
    && -1\phantom{-} && 1 && -1\phantom{-} && 2 && -1\phantom{-} && 1 && -1\phantom{-} &\\        
  \cdots  &&& -2\phantom{-} && -2\phantom{-} && -3\phantom{-} && -3\phantom{-} && -2\phantom{-} && -2\phantom{-} && -2\phantom{-} &&\cdots  \\
    && -3\phantom{-} && 3 && -5\phantom{-} && 4 && -5\phantom{-} && 3 && -3\phantom{-} & \\    
     &&& 5 && 8 && 7 && 7 && 8 && 5 && 5 &&  \\
    &&  &&  &&&  &  \qquad\vdots &&  &&  &&  &  
                        }
          }
\]
\endgroup
\caption{An infinite frieze in standard form}
\label{fig8}
\end{figure}

There is a bi-infinite row of 0s followed by successive bi-infinite rows of integers, and each row is shifted relative to the rows above and below it in such a way that any four entries $a$, $b$, $c$ and $d$ in a diamond shape satisfy the unimodular rule $ad-bc=1$, which we met earlier for Coxeter's frieze patterns. For our purposes, it is  convenient to consider infinite friezes as types of $\text{SL}_2$-tilings, and we can do this by, informally speaking, rotating an infinite frieze through $45^\circ$ clockwise to give `half' the entries  of a bi-infinite matrix $\mathbf{M}$ (the entries $m_{i,j}$, for $i\geq j$), with $0$s down the leading diagonal ($m_{i,i}=0$, for $i\in\mathbb{Z}$).

We now complete the other entries of $\mathbf{M}$ by making it antisymmetric: $m_{j,i}=-m_{i,j}$, for $i,j\in\mathbb{Z}$. The resulting bi-infinite matrix $\mathbf{M}$ is clearly an $\text{SL}_2$-tiling, and in fact if $\mathbf{M}$ is tame, then it is the unique tame $\text{SL}_2$-tiling with entries $m_{i,j}$, for $i\geq j$. 

This uniqueness property is proven in, for example, \cite[Proposition~21]{BeRe2010} (with the assumption, not needed here, that $m_{i+1,i-1}\neq 0$, for $i\in\mathbb{Z}$); it also follows quickly from Theorem~\ref{thm1}. Briefly, if $\mathbf{M}$ is tame, then we can choose two paths with vertices $a_i/b_i$ and $c_j/d_j$, for $i,j\in\mathbb{Z}$, where $a_ib_{i+1}-a_{i+1}b_i=1$ and $c_jd_{j+1}-c_{j+1}d_j=1$, such that $m_{i,j}=a_id_j-b_ic_j$. The condition $m_{i,i}=0$ implies that either $a_i=c_i$ and $b_i=d_i$ for $i\in\mathbb{Z}$, or $a_i=-c_i$ and $b_i=-d_i$ for $i\in\mathbb{Z}$, from which it follows that $m_{j,i}=-m_{i,j}$, for $i,j\in\mathbb{Z}$.  

Carrying out this process of completion for the infinite frieze in Figure~\ref{fig8} gives the $\text{SL}_2$-tiling displayed in Figure~\ref{fig9}.

\begin{figure}[ht]
\begingroup
\[
  \vcenter{
  \xymatrix @-0.2pc @!0 {
          & &  &  & \vdots &  &  & &\\
         &0 & -1\phantom{-}& -1\phantom{-} & 2 & 5 & -7\phantom{-} & -12\phantom{-} & \\
   & 1 &  0 & -1\phantom{-} & 1 & 3 & -4\phantom{-}  & -7\phantom{-} &\\
    &  1    &   1  &	0 & -1\phantom{-} & -2\phantom{-} & 3  & 5 &  \\
\dotsb  & -2\phantom{-} &   -1\phantom{-}  & 1 & 0 & -1\phantom{-} & 1  &  2 &\dotsb \\
       & -5\phantom{-}   &  -3\phantom{-}  &  2 & 1 & 0 & -1\phantom{-}  & -1\phantom{-} &\\
      &7    &  4  &  -3\phantom{-} & -1\phantom{-} & 1 & 0 & -1\phantom{-} &\\
      & 12 & 7& -5\phantom{-} & -2\phantom{-} & 1 & 1 & 0 & \\
       &  &  &  &   \vdots &  &  & & 
            }
          }
\]
\endgroup
\caption{An infinite frieze in matrix form}
\label{fig9}
\end{figure}

This discussion motivates the following definition, which uses slightly different terminology to that of \cites{BaPaTs2016,BaPaTs2018,Ts2019} suitable for our purposes.

\begin{definition}
An \emph{infinite frieze} is an $\text{SL}_2$-tiling $\mathbf{M}$ that satisfies $m_{j,i}=-m_{i,j}$, for $i,j\in\mathbb{Z}$. A \emph{positive infinite frieze} is an infinite frieze $\mathbf{M}$ for which $m_{i,j}>0$, whenever $i>j$.
\end{definition}

Of principal interest to us are \emph{tame} infinite friezes. These can be characterised as tame $\text{SL}_2$-tilings that satisfy $m_{i,i}=0$, for $i\in\mathbb{Z}$, because, as we have seen, this condition implies the stronger antisymmetry condition $m_{j,i}=-m_{i,j}$, for $i,j\in\mathbb{Z}$.

It is sometimes convenient to consider an infinite frieze to be in the form of Figure~\ref{fig8}, where the entries $m_{i,j}$ with $i<j$ are omitted; in this case we say that the frieze is in \emph{standard form}. The entries of the second row of an infinite frieze in standard form are either all $1$ or all $-1$. It is the third row of the frieze that is of particular interest to us; if all entries of the second row are 1, then the third row is called the \emph{quiddity sequence} of the infinite frieze.

\begin{definition}
Suppose that $\mathbf{M}$ is an infinite frieze with $m_{i+1,i}=1$, for $i\in\mathbb{Z}$. The \emph{quiddity sequence} of $\mathbf{M}$ is the sequence $m_{i+1,i-1}$, for $i\in\mathbb{Z}$.

If $\mathbf{M}$ is an infinite frieze with $m_{i+1,i}=-1$, for $i\in\mathbb{Z}$, then the \emph{quiddity sequence} of $\mathbf{M}$ is defined to be the quiddity sequence of $-\mathbf{M}$.
\end{definition}

Let us examine the relationship between infinite friezes and pairs of bi-infinite paths in $\mathscr{F}$. Any infinite frieze $\mathbf{M}$ satisfies $m_{i,i}=0$, for $i\in\mathbb{Z}$. Using the formula $m_{i,j}=a_id_j-b_ic_j$, with the usual notation for bi-infinite paths, we have that $a_id_i-b_ic_i=0$, and hence $a_i/b_i=c_i/d_i$, for $i\in\mathbb{Z}$. Thus, in this special case, the two bi-infinite paths coincide. 

Now, the second row of $\mathbf{M}$ (with entries $m_{i+1,i}$, for $i\in\mathbb{Z}$) consists entirely of $1$s or entirely of $-1$s. Since $a_{i+1}b_{i}-a_{i}b_{i+1}=-1$, for $i\in\mathbb{Z}$, we see that the formula $m_{i,j}=a_ib_j-a_jb_i$ for $i,j\in\mathbb{Z}$ gives rise to an infinite frieze whose second row comprises $-1$s, and the formula $m_{i,j}=a_jb_i-b_ja_i$ for $i,j\in\mathbb{Z}$ gives rise to an infinite frieze whose second row comprises $1$s. It does not matter which formula is used for computing images of $\widetilde{\Phi}$, since switching from one to the other merely changes the sign of the corresponding $\text{SL}_2$-tiling. For consistency, we use the latter formula.

We define $\FR_\infty$ to be the collection of all tame infinite friezes, with, as usual, the friezes $\mathbf{M}$ and $-\mathbf{M}$ identified. By identifying $\PATHP$ with a subcollection of $\PATHP\times\PATHP$ using the map $\gamma \longmapsto (\gamma,\gamma)$, we can think of $\Phi$ as a function from $\PATHP/\textnormal{SL}_2(\mathbb{Z})$ to $\FR_\infty$.

It is now a small matter to deduce the following theorem as a consequence of Theorem~\ref{thm1}.

\begin{theorem}\label{thm3}
The map $\Phi\colon\PATHP/\textnormal{SL}_2(\mathbb{Z})\longrightarrow \FR_\infty$ is a one-to-one correspondence.
\end{theorem}

Thus we see that while \emph{two} bi-infinite paths are needed to specify a tame $\text{SL}_2$-tiling, only \emph{one} bi-infinite path is needed to specify a tame infinite frieze. For example, consider the path
\[
\gamma=\langle\,\dotsc,\tfrac35,\tfrac23,\tfrac12,1,0,-1,-\tfrac12,-\tfrac23,-\tfrac35,\dotsc\rangle
\]
illustrated in Figure~\ref{fig10}. Let $F_i$ denote the $i$th Fibonacci number, where $F_0=0$, $F_1=1$ and $F_{i+1}=F_{i}+F_{i-1}$, for $i=1,2,\dotsc$. Let $(s_i)$ denote the bi-infinite sequence with $s_0,s_1,s_2,\dotsc$ equal to $-1,-1,1,1,-1,-1,\dotsc$ and $s_i=-s_{-i}$ for $i<0$. Then $\gamma$ has vertices $a_i/b_i$, for $i\in\mathbb{Z}$, where $a_i=s_iF_{|i|}$ and $b_i=-s_{|i|}F_{|i|+1}$, and it can be verified that $a_ib_{i+1}-a_{i+1}b_i=1$, for $i\in\mathbb{Z}$. The corresponding tame integer frieze has entries
\[
m_{i,j}=a_jb_i-b_ja_i= s_is_{|j|}F_{|i|}F_{|j|+1}-s_{|i|}s_jF_{|j|}F_{|i|+1},
\]
for $i,j\in\mathbb{Z}$. This is the infinite frieze shown in Figures~\ref{fig8} and~\ref{fig9}. 

\begin{figure}[ht]
\centering
\begin{tikzpicture}[scale=4]
\begin{scope}
    \clip(-1.3,-0.2) rectangle (1.3,0.75);
    \fareygraph[thin,grey]{-2}{2}{1}

	\foreach[count=\i] \x/\r/\d in {0/0.5/180,1/0.5/180,-0.5/0.25/0,1/0.25/0,-0.5/0.083/180,0.666/0.083/180,-0.6/0.0333/0,0.666/0.0333/0,-0.6/0.0125/0,0.625/0.0125/0,0.625/0.0125/0,-0.615/0.00481/0,0.625/0.00481/0}{
		\ifnum \i<7	
			\draw[gammacol,thick] (\x,0) arc (0:180:\r) node[pos=0.5,pointer={2*(\r)^(1/2)}{\d}]{};
		\else
			\draw[gammacol,thick] (\x,0) arc (0:180:\r);
		\fi
	}

	\foreach \x/\lab in {0.6/$\tfrac35$,0.666/$\tfrac23$,0.5/$\tfrac12$,1/$1$,0/$0$,-1/$-1\phantom{-}$,-0.5/$-\tfrac12\phantom{-}$,-0.666/$-\tfrac23\phantom{-}$}{
		 \node[gammacol,below] at (\x,0) {\lab};
	}

\node[above,gammacol] at (0.5,0.5) {$\gamma$};

\node at (0,0.2) {\footnotesize$2$};
\node at (0.92,0.06) {\footnotesize$-1$};
\node at (-1.08,0.06) {\footnotesize$-1$};
\node at (0.47,0.08) {\footnotesize$1$};
\node at (-0.47,0.08) {\footnotesize$1$};

\end{scope}
\end{tikzpicture}
\caption{The bi-infinite path $\gamma=\langle\,\dotsc,\tfrac35,\tfrac23,\tfrac12,1,0,-1,-\tfrac12,-\tfrac23,-\tfrac35,\dotsc\rangle$}
\label{fig10}
\end{figure}
 
The integer $2$ is marked just above the vertex $0$ of the path $\gamma$ in Figure~\ref{fig10} to indicate that, in navigating the Farey graph, $\gamma$ takes the second right turn at $0$. And the integer $-1$ is marked at the vertex $1$ because $\gamma$ takes the first left turn at $1$. Continuing in this way we can construct a bi-infinite sequence of integers, called the \emph{itinerary} for $\gamma$, which comprises directions for $\gamma$ to navigate the Farey graph, where we use positive integers for right turns and negative integers for left turns (and $0$ for about turns). 

We can formalise this definition as follows. Here we use the facts that the neighbours of $\infty$ in $\mathscr{F}$ are the integers, and elements of $\text{SL}_2(\mathbb{Z})$ are automorphisms of $\mathscr{F}$, so they preserve adjacency.

\begin{definition}
Let $\gamma=\langle\,\dotsc,v_{-1},v_0,v_1,v_2,\dotsc\rangle$ be a bi-infinite path in $\mathscr{F}$. For each index $i$, choose $g_i\in\text{SL}_2(\mathbb{Z})$ such that $g_i(v_{i-1})=0$ and $g_i(v_i)=\infty$, and define $e_i=g_i(v_{i+1})$. The bi-infinite sequence of integers $[\,\dotsc,e_{-1},e_0,e_1,e_2,\dotsc]$ is called the \emph{itinerary} for~$\gamma$. The itinerary $[e_1,e_2,\dots,e_{n-1}]$ for a finite path $\langle v_0,v_1,\dots,v_n\rangle$ is defined in a similar way.
\end{definition}

We use square-bracket notation for itineraries because that notation is used in the theory of continued fractions, and itineraries are intimately related to integer continued fractions. Indeed, continued fractions featured significantly in Coxeter's original work on positive friezes \cite{Co1971}, and they play an important role here too. For reasons of concision, however, we have elected to present our work without appealing to continued fractions. (See \cite{MoOv2019} for a detailed account of the relationship between continued fractions, frieze patterns, triangulations and relations in the modular group.)

Observe that the itinerary of a bi-infinite path specifies the path uniquely, up to $\text{SL}_2(\mathbb{Z})$ equivalence of paths.  

The particular path illustrated in Figure~\ref{fig10} has itinerary $[\,\dotsc,-1,1,-1,2,-1,1,-1,\dotsc]$, where $2$ is the entry of the sequence at index $0$. This sequence coincides with the quiddity sequence of the corresponding frieze (see the third row of Figure~\ref{fig8}), and in fact this observation holds more generally (see Theorem~\ref{thmZ}). As a consequence, we deduce the relatively elementary fact that \emph{any} bi-infinite sequence of integers is the quiddity sequence of some tame infinite frieze. 

Next we consider  positive infinite friezes (all of which are necessarily tame). Let $\PATHP^+$ be the set of clockwise bi-infinite paths, and let $\FR^+_\infty$ denote the collection of positive infinite friezes (which we identify with a subset of $\FR_\infty$ by the map $\mathbf{M}\longmapsto \pm\mathbf{M}$). Recall that the limits $\gamma_{-\infty}$ and $\gamma_{\infty}$ of a path $\gamma$ from $\PATHP^+$ may be equal. Given the results we have seen so far, the next theorem should come as no surprise. 

\begin{theorem}\label{thm4}
The map $\Phi\colon\PATHP^+/\textnormal{SL}_2(\mathbb{Z})\longrightarrow \FR_\infty^+$ is a one-to-one correspondence. 
\end{theorem}

We demonstrate the correspondence of this theorem with two examples, illustrated in Figure~\ref{fig11}. Each of the two subfigures shows a clockwise bi-infinite path in the Farey graph.

\begin{figure}[ht]
\centering
\begin{tikzpicture}[scale=2.5]

\begin{scope}
\discfareygraph[thin,grey]{4};

\node[] at (0:\radplus) {$1$};
\node[] at (180:\radplus) {$-1$};

\node[gammacol] at (-90:\radplus) {$0$};
\node[gammacol] at ({2*atan(1/2)-90}:\radplus) {$\tfrac12$};
\node[gammacol] at ({-2*atan(1/2)-90}:\radplus) {$-\tfrac12$};
\node[gammacol] at ({2*atan(2/3)-90}:\radplus) {$\tfrac23$};
\node[gammacol] at ({-2*atan(2/3)-90}:\radplus) {$-\tfrac23$};

\foreach \n in {1,...,20}{
	\hgline[gammacol,thick](0,0)(2*atan((\n-1)/\n)-90:2*atan(\n/(\n+1))-90:\rad);
	\hgline[gammacol,thick](0,0)(-2*atan(\n/(\n+1))-90:-2*atan((\n-1)/\n)-90:\rad);
}

\dirhgline[gammacol,draw=none,thick](0.5,0.8,204)(-90:2*atan(1/2)-90:\rad);
\dirhgline[gammacol,draw=none,thick](0.5,0.4,240)(2*atan(1/2)-90:2*atan(2/3)-90:\rad);
\dirhgline[gammacol,draw=none,thick](0.48,0.8,155)(-2*atan(1/2)-90:-90:\rad);
\dirhgline[gammacol,draw=none,thick](0.5,0.4,122)(-2*atan(2/3)-90:-2*atan(1/2)-90:\rad);

 \node[gammacol] at (-50:0.56*\rad) {$\gamma$};

\node at (-1.15,-1.15) {(a)}; 
\end{scope}

\begin{scope}[xshift=3cm]
\begin{scope}[scale=1.35,shift={(-0.7,-0.4)}]  	
    \clip(-0.2,-0.5) rectangle (1.2,1);
    \fareygraph[thin,grey]{-2}{2}{1};
	
	\foreach[count=\i] \x/\r/\d in {0.6/0.05/180,0.666/0.0208/180,1/0.166/180,1/0.5/0,0.5/0.25/180}{
		\ifnum \i>2	
			\draw[gammacol,thick] (\x,0) arc (0:180:\r) node[pos=0.5,pointer={1.9*\r}{\d}]{};
		\else
			\draw[gammacol,thick] (\x,0) arc (0:180:\r);
		\fi
	}
	
	\foreach \x/\lab in {0/0,1/1,0.5/$\tfrac12$,0.666/$\tfrac23$,0.6/$\tfrac35$}{
		 \node[gammacol,below] at (\x,0) {\lab};
	}
	
	\draw[thin,draw=grey] (0.618,0) -- (0.618,-0.3) node[below]{\footnotesize$\tfrac12(\sqrt5-1)$};

 \node[gammacol,above] at (0.5,0.5) {$\gamma$};
	
\end{scope}

\node at (-1.15,-1.15) {(b)}; 
\end{scope}

\end{tikzpicture}
\caption{Two pairs of clockwise bi-infinite paths in the Farey graph}
\label{fig11}
\end{figure}
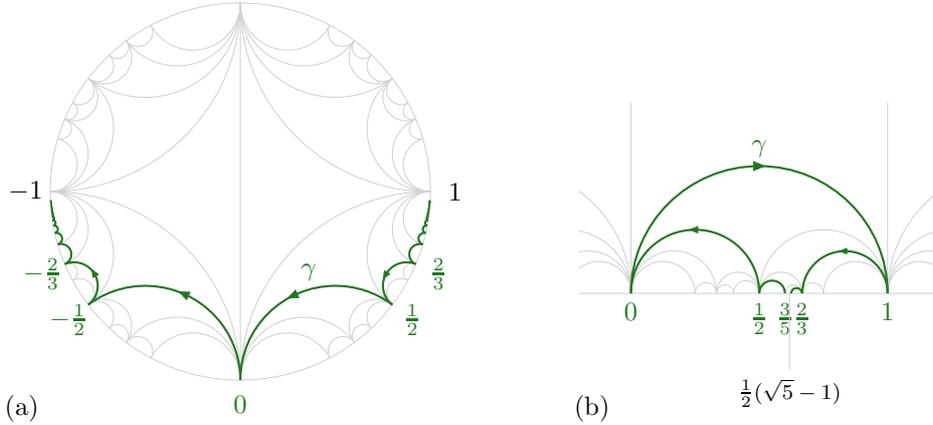

Figure~\ref{fig11}(a) illustrates the path
\[
\gamma =\left\langle\,\dotsc ,\tfrac34,\tfrac23,\tfrac12,0,-\tfrac12,-\tfrac23,-\tfrac34,\dotsc \right\rangle. 
\]
This path has vertices $a_i/b_i$,  where $a_i=-i$ and $b_i=|i|+1$, so $a_ib_{i+1}-a_{i+1}b_i=1$, for $i\in\mathbb{Z}$, and the corresponding positive infinite frieze $\mathbf{M}$ has entries
\[
m_{i,j}=a_jb_i-b_ja_i=(|j|+1)i-(|i|+1)j,
\]
for $i,j\in\mathbb{Z}$. It is displayed in Figure~\ref{fig12}(a). Notice that the quiddity sequence of this frieze is $\dotsc,2,2,4,2,2,\dotsc$, which is the itinerary for $\gamma$.

Figure~\ref{fig11}(b) illustrates the clockwise bi-infinite path
\[
\gamma =\left\langle\,\dotsc ,\tfrac8{13},\tfrac35,\tfrac12,0,1,\tfrac23,\tfrac58,\dotsc \right\rangle. 
\]
This path has vertices $a_i/b_i$, for $i\in\mathbb{Z}$, where 
\[
a_i=\begin{cases}-F_{2i-1},& i>0,\\  F_{-2i},& i\leq 0, \end{cases}
\quad\text{and}\quad
b_i=\begin{cases}-F_{2i},& i>0,\\  F_{-2i+1},& i\leq 0, \end{cases}
\]
and it can be verified that $a_ib_{i+1}-a_{i+1}b_i=1$, for $i\in\mathbb{Z}$. The corresponding positive infinite frieze $\mathbf{M}$ with entries $m_{i,j}=a_jb_i-b_ja_i$, for $i,j\in\mathbb{Z}$, is displayed in Figure~\ref{fig12}(b). Again, the quiddity sequence of the frieze (namely $\dotsc,3,3,1,2,3,3,\dotsc$) is the itinerary for $\gamma$. 

For the path $\gamma$ in Figure~\ref{fig11}(a), the two limits $\gamma_{-\infty}=1$ and $\gamma_\infty=-1$  are distinct. In contrast, for the path $\gamma$ in Figure~\ref{fig11}(b), $\gamma_{-\infty}$ and $\gamma_{\infty}$ both equal  $\tfrac12(\sqrt5-1)$, an \emph{irrational} number (the reciprocal of the golden ratio). In Section~\ref{sec8}, where we explore circumstances such as these more carefully, we shall see that in this case the entry 1 appears in infinitely many rows of the infinite frieze in standard form (see Theorem~\ref{thm99}); in this particular case there is a 1 in every row.

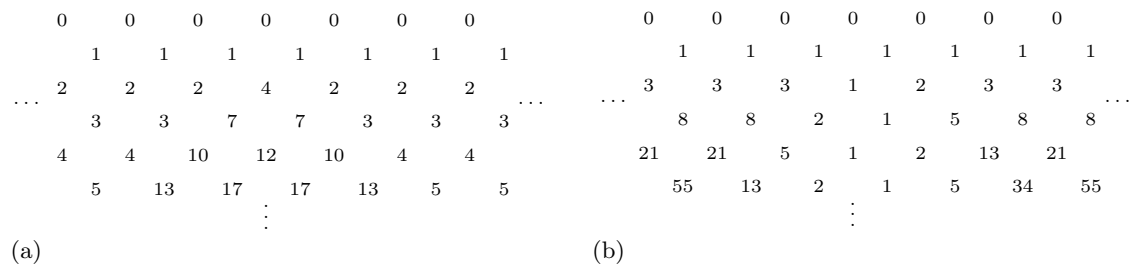
\begin{figure}[ht]
\begin{subfigure}[b]{0.515\textwidth}
\begingroup
\hspace*{-2pt}{\scriptsize$
  \vcenter{
  \xymatrix @-0.94pc @!0 {
    & 0 && 0 && 0 && 0 && 0 && 0 && 0 & \\
    && 1 && 1 && 1 && 1 && 1 && 1 && 1 && \\
 \raisebox{-4pt}{$\dotsc$}   & 2 && 2 && 2 && 4 && 2 && 2 && 2 & &\hspace*{-4pt}\raisebox{-4pt}{$\dotsc$}  \\        
    && 3 && 3 && 7 && 7 && 3 && 3 && 3 && \\
    & 4 && 4 &&10 && 12 && 10 && 4 && 4 & \\  
        && 5 && 13 && 17 && 17 && 13 && 5 && 5 && \\
    &  &&  &&&  & \vdots &&  &&  &&  &  
                        }
          }
$}
\endgroup
\vspace*{-0.2cm}
\caption{}
\label{fig10a}
\end{subfigure}
\begin{subfigure}[b]{0.485\textwidth}
\begingroup
{\scriptsize$
  \vcenter{
  \xymatrix @-0.94pc @!0 {
    & 0 && 0 && 0 && 0 && 0 && 0 && 0 & \\
    && 1 && 1 && 1 && 1 && 1 && 1 &&  1 &&\\
\raisebox{-4pt}{$\dotsc$}   &3& & 3 &&3 && 1 && 2 && 3 && 3 && \hspace*{-4pt}\raisebox{-4pt}{$\dotsc$}  \\        
    &&8&& 8 && 2 && 1 && 5 && 8 && 8 &&   \\
    &21 && 21 && 5 &&1 && 2 && 13 && 21 & \\   
      &&55  && 13 && 2 && 1 && 5 && 34 && 55 &&   \\
    &  &&  &&&  & \vdots &&  &&  & &
                        }
          }
$}
\endgroup
\vspace*{-0.2cm}
\caption{}
\label{fig10b}
\end{subfigure}
\caption{Positive infinite friezes corresponding to the paths of Figure~\ref{fig11}}
\label{fig12}
\end{figure}

Next we look at how to classify positive infinite friezes by their quiddity sequences. We recall from the start of this introduction that Conway and Coxeter used triangulated polygons to determine the (finite) quiddity cycles of positive friezes. 

\begin{definition}
A \emph{cycle sequence} is a finite sequence of positive integers obtained by removing the final term from any one of the cyclically-equivalent sequences that makes up the quiddity cycle of some positive frieze. We say that a bi-infinite sequence of integers $w$ is \emph{acyclic} if no cycle sequence appears as a finite subsequence of contiguous terms of~$w$.
\end{definition}

For example, the quiddity cycle for the frieze of Figure~\ref{fig1} consists of $1,2,2,3,1,2,4$ and its cyclically-equivalent sequences. Thus $1,2,2,3,1,2$ is a cycle sequence, and so is $2,2,3,1,2,4$. The phrase `cycle sequence' is used because such sequences of positive integers correspond to itineraries of finite paths that are simple cycles in the Farey graph, as we shall see. 

\begin{theorem}\label{thm5}
A bi-infinite sequence of positive integers is the quiddity sequence of a positive infinite frieze if and only if it is acyclic.
\end{theorem}

For an example, the bi-infinite sequence $\dotsc,8,8,1,2,2,3,1,2,8,8,\dotsc$ is not the quiddity sequence of a positive infinite frieze because it is not acyclic -- it contains the cycle sequence $1,2,2,3,1,2$ as a subsequence of contiguous terms. On the other hand, because every cycle sequence has an entry 1 (which can be verified combinatorially using triangulated polygons), we see that any bi-infinite sequence of integers greater than 1 is acyclic and therefore is the quiddity sequence of some positive infinite frieze.

Another application of our procedure for classifying $\text{SL}_2$-tilings concerns \emph{tame friezes}. A tame frieze is an integer frieze that is subject to a tameness condition, which we will come to shortly. And an integer frieze is simply a bi-infinite strip of integers of the same form as a positive frieze, but with integer entries, not necessarily positive. See Figure~\ref{fig13} for an example.

\begin{figure}[ht]
\begingroup
\[
  \vcenter{
  \xymatrix @-0.5pc @!0 {
&& 0 && 0 && 0 &&  0 && 0 && 0 && 0 && 0 &\\
& 1 && 1 && 1 && 1 && 1 && 1 && 1 && 1 && \\
&& 3 && 1 && 1 &&  -1\phantom{-} && 1 && 1 && 3 && 1 &\\
\dotsb& 2 && 2 && 0 && -2\phantom{-} && -2\phantom{-} && 0 && 2 && 2 &&\dotsb \\
&& 1 && -1\phantom{-} && -1\phantom{-} &&  -3\phantom{-} && -1\phantom{-} && -1\phantom{-} && 1 && -1\phantom{-} &\\
& -1\phantom{-} && -1\phantom{-} && -1\phantom{-} && -1\phantom{-} && -1\phantom{-} && -1\phantom{-} && -1\phantom{-} && -1\phantom{-} && \\
&& 0 && 0 && 0 &&  0 && 0 && 0 && 0 && 0 &
                        }
          }
\]
\endgroup
\caption{A tame  frieze in standard form}
\label{fig13}
\end{figure}

Just as for positive friezes, the \emph{order} of an integer frieze is the number of rows minus one. Observe that the second and second-last rows are each composed entirely of 1s or entirely of $-1$s. 

A feature of our approach is to consider friezes as types of $\text{SL}_2$-tilings; let us see how to do this for an integer frieze of order $n$. Following the procedure for infinite friezes, we begin by rotating the frieze through $45^\circ$ clockwise to give a `diagonal strip' of integers, which we will complete to a full bi-infinite matrix $\mathbf{M}$. Let us assume that the top row of $0$s of the integer frieze, once rotated, becomes the leading diagonal of terms $m_{i,i}$, for $i\in\mathbb{Z}$, of $\mathbf{M}$, in which case the bottom row of $0$s becomes the diagonal of terms $m_{i+n,i}$, for $i\in\mathbb{Z}$, and the other entries are $m_{i,j}$, for $0< i-j< n$.

Two possibilities arise. Either the entries from the second row of the integer frieze (all $1$ or all $-1$) differ in sign from the entries in the second-last row, or else the entries in these two rows are all the same. In the first case, we complete $\mathbf{M}$ by defining $m_{i+n,j}=m_{i,j+n}=m_{i,j}$ for $i,j\in\mathbb{Z}$, and in the second case we define $m_{i+n,j}=m_{i,j+n}=-m_{i,j}$ for $i,j\in\mathbb{Z}$. Clearly, the resulting bi-infinite matrix is an $\text{SL}_2$-tiling. We will prove that if $\mathbf{M}$ is tame then it is the unique completion of the integer frieze to a tame $\text{SL}_2$-tiling. 

Suppose then that $\mathbf{M}$ is \emph{any} completion of the integer frieze $m_{i,j}$, for $0\leq i-j\leq n$, to a tame $\text{SL}_2$-tiling. Since $m_{i,i}=0$, for $i\in\mathbb{Z}$, we see that $\mathbf{M}$ is an infinite frieze, so there is a single bi-infinite path $\gamma$ in $\mathscr{F}$ with vertices $a_i/b_i$ (where $a_ib_{i+1}-a_{i+1}b_i=1$) such that $m_{i,j}=a_jb_i-b_ja_i$, for $i,j\in\mathbb{Z}$. Since $m_{i+n,i}=0$, for $i\in\mathbb{Z}$, it follows that $a_{i+n}b_i-a_ib_{i+n}=0$, and hence $a_{i+n}b_i=a_ib_{i+n}$, for $i\in\mathbb{Z}$. Consequently, either $a_{i+n}=a_i$ and $b_{i+n}=b_i$ for $i\in\mathbb{Z}$, or $a_{i+n}=-a_i$ and $b_{i+n}=-b_i$ for $i\in\mathbb{Z}$. Using the formula $m_{i,j}=a_jb_i-b_ja_i$, we see that in the first case $m_{i+n,j}=m_{i,j+n}=m_{i,j}$ for $i,j\in\mathbb{Z}$, and in the second case $m_{i+n,j}=m_{i,j+n}=-m_{i,j}$ for $i,j\in\mathbb{Z}$, so the completion of the integer frieze is indeed unique.

Figure~\ref{fig13A} exhibits the completion of the integer frieze of Figure~\ref{fig13}. The diagonals of 0s are shaded grey to highlight the strips of Figure~\ref{fig13} that make up the tiling.

\begin{figure}[ht]
\centering
\begingroup
{\scriptsize
\[
  \vcenter{
  \xymatrix @-0.7pc @!0 {
&&&&&&&& \vdots &&&&&&&&\\
&\textcolor{grey}{0}&-1\phantom{-}&-3\phantom{-}&-2\phantom{-}&1&1&\textcolor{grey}{0}&-1\phantom{-}&-3\phantom{-}&-2\phantom{-}&1&1&\textcolor{grey}{0}&-1\phantom{-}&-3\phantom{-}&\\
&1&\textcolor{grey}{0}&-1\phantom{-}&-1\phantom{-}&0&1&1&\textcolor{grey}{0}&-1\phantom{-}&-1\phantom{-}&0&1&1&\textcolor{grey}{0}&-1\phantom{-}&\\
&3&1&\textcolor{grey}{0}&-1\phantom{-}&-1\phantom{-}&2&3&1&\textcolor{grey}{0}&-1\phantom{-}&-1\phantom{-}&2&3&1&\textcolor{grey}{0}&\\
&2&1&1&\textcolor{grey}{0}&-1\phantom{-}&1&2&1&1&\textcolor{grey}{0}&-1\phantom{-}&1&2&1&1&\\
&-1\phantom{-}&0&1&1&\textcolor{grey}{0}&-1\phantom{-}&-1\phantom{-}&0&1&1&\textcolor{grey}{0}&-1\phantom{-}&-1\phantom{-}&0&1&\\
&-1\phantom{-}&-1\phantom{-}&-2\phantom{-}&-1\phantom{-}&1&\textcolor{grey}{0}&-1\phantom{-}&-1\phantom{-}&-2\phantom{-}&-1\phantom{-}&1&\textcolor{grey}{0}&-1\phantom{-}&-1\phantom{-}&-2\phantom{-}&\\
&\textcolor{grey}{0}&-1\phantom{-}&-3\phantom{-}&-2\phantom{-}&1&1&\textcolor{grey}{0}&-1\phantom{-}&-3\phantom{-}&-2\phantom{-}&1&1&\textcolor{grey}{0}&-1\phantom{-}&-3\phantom{-}&\\
\dotsb&1&\textcolor{grey}{0}&-1\phantom{-}&-1\phantom{-}&0&1&1&\textcolor{grey}{0}&-1\phantom{-}&-1\phantom{-}&0&1&1&\textcolor{grey}{0}&-1\phantom{-}&\dotsb \\
&3&1&\textcolor{grey}{0}&-1\phantom{-}&-1\phantom{-}&2&3&1&\textcolor{grey}{0}&-1\phantom{-}&-1\phantom{-}&2&3&1&\textcolor{grey}{0}&\\
&2&1&1&\textcolor{grey}{0}&-1\phantom{-}&1&2&1&1&\textcolor{grey}{0}&-1\phantom{-}&1&2&1&1&\\
&-1\phantom{-}&0&1&1&\textcolor{grey}{0}&-1\phantom{-}&-1\phantom{-}&0&1&1&\textcolor{grey}{0}&-1\phantom{-}&-1\phantom{-}&0&1&\\
&-1\phantom{-}&-1\phantom{-}&-2\phantom{-}&-1\phantom{-}&1&\textcolor{grey}{0}&-1\phantom{-}&-1\phantom{-}&-2\phantom{-}&-1\phantom{-}&1&\textcolor{grey}{0}&-1\phantom{-}&-1\phantom{-}&-2\phantom{-}&\\
&\textcolor{grey}{0}&-1\phantom{-}&-3\phantom{-}&-2\phantom{-}&1&1&\textcolor{grey}{0}&-1\phantom{-}&-3\phantom{-}&-2\phantom{-}&1&1&\textcolor{grey}{0}&-1\phantom{-}&-3\phantom{-}&\\
&1&\textcolor{grey}{0}&-1\phantom{-}&-1\phantom{-}&0&1&1&\textcolor{grey}{0}&-1\phantom{-}&-1\phantom{-}&0&1&1&\textcolor{grey}{0}&-1\phantom{-}&\\
&3&1&\textcolor{grey}{0}&-1\phantom{-}&-1\phantom{-}&2&3&1&\textcolor{grey}{0}&-1\phantom{-}&-1\phantom{-}&2&3&1&\textcolor{grey}{0}&\\
&&&&&&&& \vdots &&&&&&&&
            }
          }
\]
}
\endgroup
\caption{A tame  frieze of order 6 in matrix form}
\label{fig13A}
\end{figure}

We use the completed form of an integer frieze to define tame friezes more formally.

\begin{definition}
Let $n$ be an integer greater than 1. A \emph{tame frieze} of order $n$ is a tame $\text{SL}_2$-tiling $\mathbf{M}$ such that $m_{i,i}=m_{i+n,i}=0$, for $i\in\mathbb{Z}$.
\end{definition}

A tame frieze is said to be in standard form if it is expressed as in Figure~\ref{fig13}. With this representation, many of the entries of the frieze are omitted -- they can be recovered using either the  periodicity rules  $m_{i+n,j}=m_{i,j+n}=m_{i,j}$ for $i,j\in\mathbb{Z}$, or the antiperiodicity rules $m_{i+n,j}=m_{i,j+n}=-m_{i,j}$ for $i,j\in\mathbb{Z}$ (whichever of the two sets of rules is appropriate for that particular frieze).

Notice that if $\mathbf{M}$ is a tame frieze of order $n$, then it is also a tame frieze of order $kn$, for any positive integer $k$. 

We consider now how to represent any tame frieze $\mathbf{M}$ using the Farey graph. We have seen that any bi-infinite path $\gamma$ in $\mathscr{F}$ with vertices $a_i/b_i$ such that $m_{i,j}=a_jb_i-b_ja_i$ must satisfy $a_{i+n}/b_{i+n}=a_i/b_i$, for $i\in\mathbb{Z}$. Hence $\gamma$ is a periodic path with period $n$ (where $n$ is not necessarily the \emph{smallest} period of the path). Conversely, it can easily be shown (and will be, later) that if $\gamma$ is a periodic bi-infinite path with period $n$, then the formulas $m_{i,j}=a_jb_i-b_ja_i$, for $i,j\in\mathbb{Z}$, define a tame frieze of order $n$.

We will relate periodic bi-infinite paths to closed paths.

\begin{definition}
A \emph{closed path} in $\mathscr{F}$ is a finite path $\langle v_0,v_1,\dots,v_n\rangle$ for which $v_0=v_n$. A \emph{simple closed path} in $\mathscr{F}$ is a closed path $\langle v_0,v_1,\dots,v_n\rangle$ in $\mathscr{F}$ for which the vertices $v_0,v_1,\dots,v_{n-1}$ are distinct.
\end{definition}

We define a \emph{clockwise} simple closed path in $\mathscr{F}$ to be a  simple closed path $\langle v_0,v_1,\dots,v_n\rangle$ in $\mathscr{F}$ for which the sequence $v_0,v_1,\dots,v_{n-1}$ is in clockwise order.

Let $\mathscr{C}_n$ denote the collection of closed paths of length $n$ in $\mathscr{F}$. We can identify $\mathscr{C}_n$ with the collection of periodic bi-infinite paths of period $n$ in $\mathscr{F}$ by mapping the closed path $\langle v_0,v_1,\dots,v_n\rangle$ to the periodic bi-infinite path $\langle\, \dots,v_{-1},v_0,v_1,v_2,\dotsc\rangle$, where $v_i= v_j$ if $i\equiv j\pmod{n}$. 

We denote by $\FR_n$ the quotient set of the collection of tame friezes of order $n$ under the usual equivalence relation that identifies an $\text{SL}_2$-tiling with its negative. By thinking of $\mathscr{C}_n$ as a collection of periodic bi-infinite paths, we can consider it to be a subcollection of $\mathscr{P}$ (and hence a subcollection of $\mathscr{P}\times \mathscr{P}$ using the identification $\gamma\longmapsto (\gamma,\gamma)$). Since $\text{SL}_2(\mathbb{Z})$ acts on $\mathscr{C}_n$, we can consider $\Phi$ to be a function from $\mathscr{C}_n/\textnormal{SL}_2(\mathbb{Z})$ to $\FR_n$.

\begin{theorem}\label{thm6}
The map $\Phi\colon\mathscr{C}_n/\textnormal{SL}_2(\mathbb{Z})\longrightarrow \FR_n$ is a one-to-one correspondence. 
\end{theorem}

Positive friezes are special types of tame friezes, so we can restrict the function $\Phi$ further to classify positive friezes. To do this, however, we first need to revisit the definition of a positive frieze in the context of $\text{SL}_2$-tilings. A \emph{positive frieze} of order $n$, then, is a tame frieze $\mathbf{M}$ of order $n$ that satisfies $m_{i,j}>0$, for $0<i-j<n$. Notice that the entries of $\mathbf{M}$ are \emph{not} all positive; after all, $m_{i,i}=m_{i+n,i}=0$, for $i\in\mathbb{Z}$, and, moreover, the antiperiodicity rules $m_{i+n,j}=m_{i,j+n}=-m_{i,j}$, for $i,j\in\mathbb{Z}$, imply that $\mathbf{M}$ comprises alternating diagonal `strips' of positive and negative integers separated by diagonals of $0$s. 

Let $\FR_n^+$ denote the set of positive friezes, which we identify with a subset of $\FR_n$ using the embedding $\mathbf{M}\longmapsto \pm\mathbf{M}$. We also define $\mathscr{C}_n^+$ to be the collection of clockwise simple closed paths of length $n$ in $\mathscr{F}$, a subcollection of $\mathscr{C}_n$. The group $\text{SL}_2(\mathbb{Z})$ acts on $\mathscr{C}_n^+$, so we can consider the restriction of $\Phi$ to the set $\mathscr{C}_n^+/\textnormal{SL}_2(\mathbb{Z})$.

\begin{theorem}\label{thm7}
The map $\Phi\colon\mathscr{C}_n^+/\textnormal{SL}_2(\mathbb{Z})\longrightarrow \FR_n^+$ is a one-to-one correspondence. 
\end{theorem}

Theorem~\ref{thm7} is equivalent to \cite[Proposition~2.2.1]{MoOvTa2015}, which is itself a reformulation using the Farey graph of Conway and Coxeter's original classification of positive friezes with triangulated polygons. In brief, any triangulated polygon can be realised as a \emph{hyperbolic} triangulated polygon within $\mathscr{F}$,  and the vertices of a hyperbolic triangulated polygon in $\mathscr{F}$ listed in clockwise order specify a clockwise simple closed path in $\mathscr{F}$. We expand on this correspondence in Section~\ref{sec6}; a full explanation is given in \cite{MoOvTa2015}. The correspondence allows us to use the Farey graph and the class $\mathscr{C}_n^+$ in place of the class of triangulated polygons. 

Let us now work the other way and refashion Theorem~\ref{thm6} using triangulated polygons rather than closed paths in the Farey graph, thereby generalising Conway and Coxeter's original classification of positive friezes to the larger class of tame friezes. Observe that any closed path in $\mathscr{F}$ lies within some hyperbolic triangulated polygon in $\mathscr{F}$. Thus Theorem~\ref{thm6} shows that there is a correspondence between closed paths in triangulated polygons  -- where we allow ourselves the freedom to traverse the diagonals of a triangulated polygon -- and tame friezes.

For example, consider the closed path $\gamma = \langle \infty, 2,1,\infty,0,-1,\infty\rangle$ illustrated in Figure~\ref{fig14}(a). We can extend this path to a periodic bi-infinite path $\widetilde{\gamma}$ in the agreed way. The itinerary for $\widetilde{\gamma}$ is the periodic bi-infinite sequence with periodic part $[1,1,-1,1,1,3]$. This finite sequence of integers encodes a set of directions for navigating a closed path in a triangulated pentagon, as shown in Figure~\ref{fig14}(b), where the labels are the terms of the itinerary for $\widetilde{\gamma}$. The corresponding tame frieze is that displayed in Figure~\ref{fig13}; observe that the quiddity sequence of the frieze and the itinerary for $\widetilde{\gamma}$ coincide.

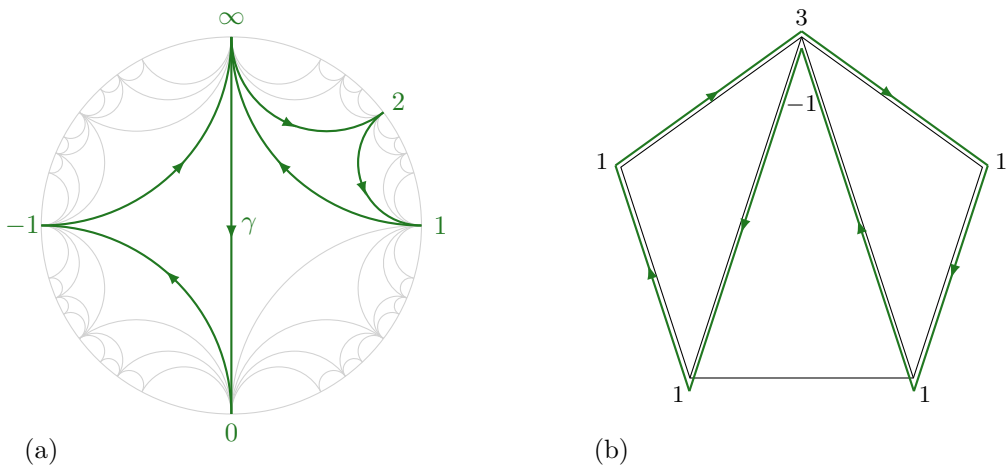
\begin{figure}[ht]
\centering

\begin{tikzpicture}[scale=2.5]

\begin{scope}
\discfareygraph[thin,grey]{4};
\hgline[gammacol,thick,reverse directed](0,0)(90:180:\rad);
\hgline[gammacol,thick,reverse directed](0,0)(180:270:\rad);
\hgline[gammacol,thick,reverse directed](0,0)(90:360:\rad);
\hgline[gammacol,thick,reverse directed](0,0)(0:{90-2*atan(1/2)}:\rad);
\hgline[gammacol,thick,reverse directed](0,0)({90-2*atan(1/2)}:90:\rad);

\draw[gammacol,thick,reverse directed] (-90:\rad) -- (90:\rad) node[pos=0.5,right] {$\gamma$};

\pgfmathsetmacro{\R}{\rad+0.10}

\node[gammacol] at (90:\R) {$\infty$};
\node[gammacol] at (180:\R) {$-1$};
\node[gammacol] at (-90:\R) {$0$};
\node[gammacol] at (0:\R) {$1$};
\node[gammacol] at ({2*atan(2)-90}:\R) {$2$};

\node at (-1,-1.2) {(a)}; 
\end{scope}

\begin{scope}[xshift=3cm]
\begin{scope}[rotate=90]
\draw (0:\rad) \foreach \x in {72,144,...,360} {
                -- (\x:\rad)
            }-- cycle (90:\rad);	
\draw (144:\rad) -- (0:\rad); 
\draw (0:\rad) -- (-144:\rad);

\pgfmathsetmacro{\gap}{0.03}
\pgfmathsetmacro{\R}{\rad+\gap}
\pgfmathsetmacro{\L}{2*\R*sin(36)}
\pgfmathsetmacro{\r}{\rad-2*\gap}

\draw[gammacol,thick,reverse directed] (0:\R) --++(126:\L);
\draw[gammacol,thick,reverse directed] (0:\R)++(126:\L)--++(198:\L+1.6*\gap);
\draw[gammacol,thick,reverse directed] (0:\R)++(126:\L)++(198:\L+1.6*\gap)-- (0:\r);
\draw[gammacol,thick,directed] (0:\R) --++(-126:\L);
\draw[gammacol,thick,directed] (0:\R)++(-126:\L)--++(-198:\L+1.6*\gap);
\draw[gammacol,thick,directed] (0:\R)++(-126:\L)++(-198:\L+1.6*\gap)-- (0:\r);

\pgfmathsetmacro{\M}{\rad+3.5*\gap}
\node at (0:\M) {\small$3$};
\node at (72:\M) {\small$1$};
\node at (144:\M) {\small$1$};
\node at (-144:\M) {\small$1$};
\node at (-72:\M) {\small$1$};
\node at (0:\rad-12*\gap) {\small$-1$};
\end{scope}
\node at (-1,-1.2) {(b)}; 
\end{scope}

\end{tikzpicture}
\caption{(a) The closed path $\gamma=\langle \infty, 2,1,\infty,0,-1,\infty\rangle$ on $\mathscr{F}$\quad (b) A representation of $\gamma$ as a closed path in a triangulated Euclidean polygon }
\label{fig14}
\end{figure}

In \cite[Section~7]{CuHo2019} an intriguing combinatorial model for classifying tame friezes is developed, which uses  triangulations of polygons with integer labels attached to the triangles. It would be of interest to relate Theorem~\ref{thm6} to this model.

We finish this introduction with two other classes of $\text{SL}_2$-tilings that have been studied before and which can be considered within the geometric framework provided by the Farey graph. The first is a particular class of antiperiodic $\text{SL}_2$-tilings considered in \cite{MoOvTa2015}. Let $r$ and $s$ be positive integers greater than 1. Let $\mathbf{M}$ be an $\text{SL}_2$-tiling that satisfies the antiperiodicity rules $m_{i+r,j}=m_{i,j+s}=-m_{i,j}$, for $i,j\in\mathbb{Z}$, and which has an $r\times s$ submatrix of positive integers; that is, there are integers $k$ and $\ell$ for which $m_{i,j}>0$, for $i=k+1,k+2,\dots,k+r$ and $j=\ell+1,\ell+2,\dots,\ell+s$. It is proved in \cite[Proposition~3.4.1]{MoOvTa2015} that all such $\text{SL}_2$-tilings are tame. We denote the class of these tilings under the usual equivalence relation that identifies $\mathbf{M}$ and $-\mathbf{M}$ by $\SL^+(r,s)$. Let $(\mathscr{C}_r\times \mathscr{C}_s)^+$ be the collection of pairs $(\gamma,\delta)$ of clockwise simple closed paths in $\mathscr{C}_r^+\times \mathscr{C}_s^+$ that do not intersect one another; this collection is invariant under the action of $\text{SL}_2(\mathbb{Z})$.

\begin{theorem}\label{thm8}
The map $\Phi\colon(\mathscr{C}_r\times \mathscr{C}_s)^+/\textnormal{SL}_2(\mathbb{Z})\longrightarrow \SL^+(r,s)$ is a one-to-one correspondence. 
\end{theorem}

Theorem~\ref{thm8} is equivalent to \cite[Theorem~2]{MoOvTa2015}; in the context of this paper, that theorem (effectively) prescribes two non-intersecting clockwise simple closed paths by their quiddity cycles.

A variety of other sets of periodic and antiperiodic $\text{SL}_2$-tilings can be classified in a similar spirit to Theorem~\ref{thm8}. For example, the class $\mathscr{C}_r\times \mathscr{C}_s$ corresponds to the set of tame $\text{SL}_2$-tilings $\mathbf{M}$ that satisfy one of $m_{i+r,j}=m_{i,j}$ or $m_{i+r,j}=-m_{i,j}$ for all $i,j\in\mathbb{Z}$, and one of  $m_{i,j+s}=m_{i,j}$ or $m_{i,j+s}=-m_{i,j}$ for all $i,j\in\mathbb{Z}$. Furthermore, one can refine this correspondence to distinguish periodic and antiperiodic rules, both horizontally and vertically (the method involves analysing clockwise simple closed subpaths of closed paths).

We pursue this no further, and instead turn our attention to another  class of $\text{SL}_2$-tilings that has been considered before (in \cite{Ov2018,MoOv2019}) -- the class $\FR_n^0$ of those tame friezes of order $n$ for which the quiddity sequence comprises positive integers (but the other entries of the frieze are not necessarily positive). This class can be identified with a subclass of $\FR_n$ through the embedding $\mathbf{M} \longmapsto \pm \mathbf{M}$, in which case we see that $\FR_n^+ \subset \FR_n^0 \subset \FR_n$. The collection   of closed paths corresponding to $\FR_n^0$  is the collection $\mathscr{C}_n^0$ of closed paths $\langle v_0,v_1,\dots,v_n\rangle$ of length $n$ for which $v_{i-1},v_{i},v_{i+1}$ is in clockwise order, for $i=1,2,\dots,n$, where $v_{n+1}=v_1$. (Such a path is not necessarily a `clockwise' path, according to our definition earlier, since it might wind round the unit circle several times and thereby fail the condition that $v_i,v_j,v_k$ is in clockwise order whenever $i<j<k$.) Observe that $\text{SL}_2(\mathbb{Z})$ acts on $\mathscr{C}_n^0$, giving us the following result (due to the referee).

\begin{theorem}\label{thm9}
The map $\Phi\colon\mathscr{C}_n^0/\textnormal{SL}_2(\mathbb{Z})\longrightarrow \FR_n^0$ is a one-to-one correspondence. 
\end{theorem}

An elegant combinatorial model for classifying $\FR_n^0$ was developed in \cite{Ov2018}, using dissections of polygons into $3m$-gons (triangles, hexagons, and so forth). In \cite[Sections~4 and~5]{MoOv2019} it is explained how this combinatorial model corresponds to the sort of model described by Theorem~\ref{thm9}; indeed, \cite[Theorem~5.13]{MoOv2019} is almost equivalent to Theorem~\ref{thm9}, although framed somewhat differently.

\subsubsection*{Acknowledgements}

I gratefully acknowledge Peter J{\o}rgensen for introducing me to friezes and $\text{SL}_2$-tilings, and for discussions in which many of the ideas presented here were conceived. I also thank the referee for recommending the inclusion of Theorems~\ref{thm7}--\ref{thm9}, as well as several other valuable contributions. 

\section{Tame tilings}\label{sec2}

Here we prove Theorem~\ref{thm1}; that is, we prove that the function $\Phi\colon (\PATHP\times\PATHP)/\text{SL}_2(\mathbb{Z})\longrightarrow \SL$ is a bijection. Some of our arguments could be shortened by extracting results  from the literature. However, we have elected not to use other sources in order to keep the discussion self-contained. 

We begin by stating (but not proving) a well-known elementary lemma that characterises tame $\text{SL}_2$-tilings.

\begin{lemma}\label{lem1}
An $\textnormal{SL}_2$-tiling $\mathbf{M}$ is tame if and only if there are bi-infinite sequences of integers $(u_i)$ and $(v_j)$ such that
\[
m_{i+1,j}+m_{i-1,j}=u_im_{i,j}\quad\text{and}\quad m_{i,j+1}+m_{i,j-1}=v_jm_{i,j},
\]
for $i,j\in\mathbb{Z}$.
\end{lemma}

Informally speaking, this lemma states that, for the matrix $\mathbf{M}$,  `row $i+1$ plus row $i-1$ equals $u_i$ times row $i$' and `column $j+1$ plus column $j-1$ equals $v_j$ times column $j$'.

The next lemma is also elementary.

\begin{lemma}\label{lem2}
Let $a_1/b_1$, $a_2/b_2$ and $a_3/b_3$ be reduced rationals such that $a_1b_2-a_2b_1=1$ and $a_2b_3-a_3b_2=1$. Then there is an integer $k$ such that $a_3+a_1=ka_2$ and $b_3+b_1=kb_2$.
\end{lemma}
\begin{proof}
The result can easily be established if $a_2=0$ or $b_2=0$, so let us suppose that neither $a_2$ nor $b_2$ is 0. Combining the equations $a_1b_2-a_2b_1=1$ and $a_2b_3-a_3b_2=1$ we obtain $a_2(b_3+b_1)=b_2(a_3+a_1)$. Since $a_2$ and $b_2$ are coprime we see that $a_3+a_1=ka_2$, for some integer $k$. Consequently $b_3+b_1=kb_2$, as required.
\end{proof}

We will find it convenient to define 
\[
J=\begin{pmatrix}0 & 1\\ -1 & 0 \end{pmatrix},\quad\text{so}\quad J^{-1}=\begin{pmatrix}0 & -1\\ 1 & 0 \end{pmatrix}.
\]
We write $A^T$ for the transpose matrix of $A\in\text{SL}_2(\mathbb{Z})$. Then clearly $J^T=J^{-1}=-J$. Also, we often have recourse to the useful formula $JA^{-1}J^{-1}=J^{-1}A^{-1}J=A^{T}$, for $A\in\text{SL}_2(\mathbb{Z})$.

Let us now prove that the function $\widetilde{\Phi}\colon\mathscr{P}\times\mathscr{P}\longrightarrow \SL$ is well-defined. Consider a pair $(\gamma, \delta)$ of bi-infinite paths in the Farey graph $\mathscr{F}$. We continue the notation of the introduction, writing the vertices of $\gamma$ as reduced rationals $a_i/b_i$, for $i\in\mathbb{Z}$, and writing the vertices of $\delta$ as reduced rationals $c_j/d_j$, for $j\in\mathbb{Z}$. We choose these representations such that $a_ib_{i+1}-a_{i+1}b_i=1$ and $c_jd_{j+1}-c_{j+1}d_j=1$, for $i,j\in\mathbb{Z}$. Now we define $\mathbf{M}\colon \mathbb{Z}\times \mathbb{Z}\longrightarrow \mathbb{Z}$ by 
\[
m_{i,j}=\begin{pmatrix}a_i & b_i\end{pmatrix}J\begin{pmatrix}c_j \\ d_j\end{pmatrix}=a_id_j-b_ic_j,
\]
where $m_{i,j}=\mathbf{M}(i,j)$. Then
\begin{equation}\label{eqn2}
\begin{pmatrix}m_{i,j} & m_{i,j+1}\\ m_{i+1,j} & m_{i+1,j+1} \end{pmatrix}=
\begin{pmatrix}a_i & b_i \\a_{i+1} & b_{i+1}\end{pmatrix}J
\begin{pmatrix}c_{j} & c_{j+1}\\ d_j & d_{j+1}\end{pmatrix}
\in \text{SL}_2(\mathbb{Z}).
\end{equation}
Hence $\mathbf{M}$ is an $\text{SL}_2$-tiling. Next, observe that, by Lemma~\ref{lem2}, for each $i\in\mathbb{Z}$ there is an integer $u_i$ such that $a_{i+1}+a_{i-1}=u_ia_i$ and $b_{i+1}+b_{i-1}=u_ib_i$. Hence
\[
m_{i+1,j}+m_{i-1,j}=(a_{i+1}d_j-b_{i+1}c_j)+(a_{i-1}d_j-b_{i-1}c_j)=u_ia_id_j-u_ib_ic_j=u_im_{i,j},
\]
for $j\in\mathbb{Z}$. Similarly, there are integers $v_j$ with $m_{i,j+1}+m_{i,j-1}=v_jm_{i,j}$, for $i,j\in\mathbb{Z}$. Hence $\mathbf{M}$ is tame, by Lemma~\ref{lem1}.

Notice that if we write the vertices of $\gamma$ in the alternative form $(-a_i)/(-b_i)$, for $i\in\mathbb{Z}$, then we obtain the matrix $-\mathbf{M}$ in place of $\mathbf{M}$ (and a similar comment applies to $\delta$), so $\widetilde{\Phi}$ is indeed well-defined.

The next task is to check that two pairs of bi-infinite paths that lie in the same orbit of $\text{SL}_2(\mathbb{Z})$ are mapped  to the same element of $\SL$. To see this, suppose that $A\in\text{SL}_2(\mathbb{Z})$, and 
\[
\begin{pmatrix}a_i' \\ b_i'\end{pmatrix}=A\begin{pmatrix}a_i \\ b_i\end{pmatrix},
\quad
\begin{pmatrix}c_j' \\ d_j'\end{pmatrix}=A\begin{pmatrix}c_j \\ d_j\end{pmatrix}
\quad\text{and}\quad
m'_{i,j}=a_i'd_j'-b_i'c_j',
\]
for $i,j\in\mathbb{Z}$. Then we can use the formula $JA^{-1}J^{-1}=A^T$ to give
\[
m_{i,j}' = \begin{pmatrix}a_i' & b_i'\end{pmatrix}J\begin{pmatrix}c_j' \\ d_j'\end{pmatrix}
=\begin{pmatrix}a_i & b_i\end{pmatrix}A^TJA\begin{pmatrix}c_j \\ d_j\end{pmatrix}=m_{i,j},
\]
as required. It follows, then, that we have a well-defined function $\Phi\colon (\PATHP\times\PATHP)/\text{SL}_2(\mathbb{Z})\longrightarrow \SL$, which maps the orbit of  $(\gamma, \delta)$ under $\text{SL}_2(\mathbb{Z})$  to  $\pm\mathbf{M}$.

For the second part of the proof of Theorem~\ref{thm1}, we will define a function 
\[
\Psi\colon \SL\longrightarrow (\PATHP\times\PATHP)/\text{SL}_2(\mathbb{Z})
\]
that will prove to be the inverse function of $\Phi$. The next lemma is fundamental to defining $\Psi$.

\begin{lemma}\label{lem3}
Let $\mathbf{M}$ be a tame $\textnormal{SL}_2$-tiling with 
\[
M= \begin{pmatrix}m_{0,0} & m_{0,1}\\ m_{1,0} & m_{1,1} \end{pmatrix}. 
\]
Define $a_i=m_{i,0}$, $b_i=m_{i,1}$,  $c_j=m_{0,j}$ and $d_j=m_{1,j}$, for $i,j\in\mathbb{Z}$. Then 
\[
m_{i,j} = \begin{pmatrix}a_i & b_i\end{pmatrix}M^{-1}\begin{pmatrix}c_j \\ d_j\end{pmatrix},
\]
for $i,j\in\mathbb{Z}$.
\end{lemma}
\begin{proof}
Let $(u_i)$ be the bi-infinite sequence from Lemma~\ref{lem1}. Then 
\[
\begin{pmatrix}m_{i,j} \\ m_{i+1,j}\end{pmatrix} =U_i\begin{pmatrix}m_{i-1,j} \\ m_{i,j}\end{pmatrix},\quad\text{where}\quad U_i = \begin{pmatrix}0 & 1 \\ -1 & u_i\end{pmatrix},
\]
for $i,j\in\mathbb{Z}$. We will prove the lemma when $i>0$ (the case $i\leq 0$ can be handled similarly). For $i>0$,
\[
\begin{pmatrix}m_{i,j} \\ m_{i+1,j}\end{pmatrix} =U_i\dotsb U_1 \begin{pmatrix}c_j \\ d_j\end{pmatrix}.
\]
Applying this with $j=0$ and $j=1$ we obtain
\[
\begin{pmatrix}a_i &  b_i\\a_{i+1}& b_{i+1}\end{pmatrix}=U_i\dotsb U_1M.
\]
Hence
\[
\begin{pmatrix}m_{i,j} \\ m_{i+1,j}\end{pmatrix}=\begin{pmatrix}a_i &  b_i\\a_{i+1}& b_{i+1}\end{pmatrix}M^{-1}\begin{pmatrix}c_j \\ d_j\end{pmatrix}.
\]
Taking the first component of each side of this vector equation gives the result.
\end{proof}

Consider now a tame $\text{SL}_2$-tiling $\mathbf{M}$. Let
\[
M= \begin{pmatrix}m_{0,0} & m_{0,1}\\ m_{1,0} & m_{1,1} \end{pmatrix}. 
\]
We define two bi-infinite sequences of reduced rationals $a_i/b_i$ and $c_j/d_j$, for $i,j\in\mathbb{Z}$, by 
\begin{equation}\label{eqn3}
\begin{pmatrix}a_i\\ b_i \end{pmatrix}= \begin{pmatrix}m_{i,0}\\ m_{i,1} \end{pmatrix}
\quad\text{and}\quad
\begin{pmatrix}c_j\\ d_j \end{pmatrix}= M^TJ^{-1}\begin{pmatrix}m_{0,j}\\ m_{1,j} \end{pmatrix}.
\end{equation}
Then, as vertices of the Farey graph, $a_i/b_i \sim a_{i+1}/b_{i+1}$ for $i\in\mathbb{Z}$, since $a_ib_{i+1}-a_{i+1}b_i=m_{i,0}m_{i+1,1}-m_{i+1,0}m_{i,1}=1$. Hence $a_i/b_i$, for $i\in\mathbb{Z}$, determines a bi-infinite path in the Farey graph, as does $c_j/d_j$, for $j\in\mathbb{Z}$ (for similar reasons). Clearly, $-\mathbf{M}$ gives rise to the same pair of bi-infinite paths, so we obtain a function from $\SL$ to $\PATHP\times\PATHP$. We define $\Psi\colon \SL \longrightarrow (\PATHP\times\PATHP)/\text{SL}_2(\mathbb{Z})$ to be the composition of the function just constructed with the projection $\PATHP\times\PATHP\longrightarrow(\PATHP\times\PATHP)/\text{SL}_2(\mathbb{Z})$.

It remains to prove that $\Phi$ and $\Psi$ are inverse functions. Consider, then, a tame $\text{SL}_2$-tiling $\mathbf{M}$. Under $\Psi$, the pair $\pm\mathbf{M}$ maps to the $\text{SL}_2(\mathbb{Z})$-orbit of the pair of bi-infinite paths with vertices $a_i/b_i$ and $c_j/d_j$ given by equations~\eqref{eqn3}. And under $\Phi$ this orbit maps to $\pm\mathbf{M}'$, where
\[
m'_{i,j} = \begin{pmatrix}a_i& b_i \end{pmatrix}J\begin{pmatrix}c_j\\ d_j\end{pmatrix}=\begin{pmatrix}m_{i,0}& m_{i,1} \end{pmatrix}JM^TJ^{-1}\begin{pmatrix}m_{0,j}\\ m_{1,j} \end{pmatrix}.
\]
Now, $JM^TJ^{-1}=M^{-1}$, so we see that
\[
m'_{i,j}=\begin{pmatrix}m_{i,0}& m_{i,1} \end{pmatrix}M^{-1}\begin{pmatrix}m_{0,j}\\ m_{1,j} \end{pmatrix}=m_{i,j},
\]
by Lemma~\ref{lem3}. Hence $\Phi \circ \Psi$ is the identity function. 

Next, consider a pair of bi-infinite paths $\gamma$ and $\delta$ in $\mathscr{F}$ with vertices $a_i/b_i$ and $c_j/d_j$, respectively, where $a_{i}b_{i+1}-a_{i+1}b_{i}=1$ and $c_{j}d_{j+1}-c_{j+1}d_{j}=1$, for $i,j\in\mathbb{Z}$. Let $\mathbf{M}$ be the tame $\text{SL}_2$-tiling given by $m_{i,j}=a_id_j-b_ic_j$, for $i,j\in\mathbb{Z}$. Under $\Psi\circ \Phi$, the $\text{SL}_2(\mathbb{Z})$-orbit of $(\gamma,\delta)$ maps to the $\text{SL}_2(\mathbb{Z})$-orbit of $(\gamma',\delta')$, where   $\gamma'$ and $\delta'$ have vertices $a_i'/b_i'$ and $c_j'/d_j'$, respectively, that satisfy
\[
\begin{pmatrix}a_i'\\ b_i' \end{pmatrix}= \begin{pmatrix}m_{i,0}\\ m_{i,1} \end{pmatrix}=\begin{pmatrix}c_0 & d_0\\ c_1 & d_1 \end{pmatrix}J^{-1}\begin{pmatrix}a_i\\ b_i \end{pmatrix}
\]
and
\[
\begin{pmatrix}c_j'\\ d_j' \end{pmatrix}= M^TJ^{-1}\begin{pmatrix}m_{0,j}\\ m_{1,j} \end{pmatrix}=M^TJ^{-1}\begin{pmatrix}a_0 & b_0\\ a_1 & b_1 \end{pmatrix}J\begin{pmatrix}c_j\\ d_j \end{pmatrix},
\]
for $i,j\in\mathbb{Z}$. Now, applying equation~\eqref{eqn2} with $i=j=0$ gives 
\[
M=
\begin{pmatrix}a_0 & b_0 \\a_{1} & b_{1}\end{pmatrix}J
\begin{pmatrix}c_{0} & c_{1}\\ d_0 & d_{1}\end{pmatrix}.
\]
Using this equation, and the formula $JA^{-1}J^{-1}=J^{-1}A^{-1}J=A^T$, for $A\in\text{SL}_2(\mathbb{Z})$, one can check that 
\[
\begin{pmatrix}c_0 & d_0\\ c_1 & d_1 \end{pmatrix}J^{-1}=M^TJ^{-1}\begin{pmatrix}a_0 & b_0\\ a_1 & b_1 \end{pmatrix}J.
\]
Let $A$ be the element of $\text{SL}_2(\mathbb{Z})$ given by either side of this equation. Then 
\[
\begin{pmatrix}a_i'\\ b_i' \end{pmatrix}=A\begin{pmatrix}a_i\\ b_i \end{pmatrix}
\quad\text{and}\quad
\begin{pmatrix}c_j'\\ d_j' \end{pmatrix}= A\begin{pmatrix}c_j\\ d_j \end{pmatrix},
\]
for $i,j\in\mathbb{Z}$. It follows that $(\gamma,\delta)$ and $(\gamma',\delta')$ belong to the same orbit of $\text{SL}_2(\mathbb{Z})$, so $\Psi\circ \Phi$ is the identity function. 

In summary, $\Phi$ and $\Psi$ are inverse functions, so $\Phi$ is a one-to-one correspondence, which proves Theorem~\ref{thm1}.

\section{Positive tilings}\label{sec3}

In this section we prove Theorem~\ref{thm2}, which says that $\Phi\colon (\PATHP\times\PATHP)^+/\text{SL}_2(\mathbb{Z})\longrightarrow \SL^+$ is a one-to-one correspondence. To establish this, we first  show that $\Phi$ maps $(\PATHP\times\PATHP)^+/\text{SL}_2(\mathbb{Z})$ into $\SL^+$, and then we show that $\Psi$ maps $\SL^+$ into $(\PATHP\times\PATHP)^+/\text{SL}_2(\mathbb{Z})$.

First, however, it is helpful to note that a clockwise bi-infinite path $\gamma=\langle\,\dotsc v_{-1},v_0,v_1,v_2,\dotsc\rangle$ in $\mathscr{F}$ has one of three types: (i) a decreasing bi-infinite sequence of rational numbers; (ii) $v_i<v_{i-1}<\dotsb < \gamma_{-\infty}\leq \gamma_\infty< \dotsb <v_{i+2}<v_{i+1}$, for some integer $i$; or (iii) $v_i=\infty$ for some integer $i$ and $v_{i-1}<v_{i-2}<\dotsb < \gamma_{-\infty}\leq \gamma_\infty< \dotsb <v_{i+2}<v_{i+1}$. Of course, in all three cases $v_j$ and $v_{j+1}$ are adjacent in $\mathscr{F}$, for $j\in\mathbb{Z}$.

Returning to the proof of Theorem~\ref{thm2}, consider a pair of clockwise bi-infinite paths $(\gamma,\delta)$ in the Farey graph, where, in the notation of the introduction, the limits $\gamma_{-\infty},\gamma_{\infty},\delta_{-\infty},\delta_{\infty}$ are in clockwise order in $\mathbb{R}_\infty$ (after identifying $\mathbb{R}_\infty$ with the unit circle by the modified Cayley transform $z\longmapsto (iz+1)/(z+i)$) and $\gamma_\infty\neq\gamma_{-\infty}$ and $\delta_\infty\neq\delta_{-\infty}$. By applying an element of $\text{SL}_2(\mathbb{Z})$ we can assume that the vertex of $\delta$ with index 1 is $\infty$. It follows that $\gamma$ is a decreasing sequence of reduced rationals $a_i/b_i$, where $i\in\mathbb{Z}$. Let us choose all the denominators $b_i$ to be positive, in which case the fact that the path is decreasing implies that $a_{i}b_{i+1}-a_{i+1}b_{i}>0$, for $i\in\mathbb{Z}$. Consequently, we see that $a_{i}b_{i+1}-a_{i+1}b_{i}=1$, for  $i\in\mathbb{Z}$.

Let $c_j/d_j$, for $j\in\mathbb{Z}$, be reduced rationals that form the path $\delta$, where $c_1=-1$ and $d_1=0$. We choose the denominators $d_j$ to be positive if $j<1$, and negative if $j>1$. Again, we can check that $c_{j}d_{j+1}-d_{j+1}c_{j}=1$, for $j\in\mathbb{Z}$. Now, because $a_i/b_i >c_j/d_j$, for $i\in\mathbb{Z}$ and $j<1$, we obtain
\[
m_{i,j} = a_id_j-b_ic_j = b_id_j\left(\frac{a_i}{b_i}-\frac{c_j}{d_j}\right)>0.
\] 
Similarly, $m_{i,j}>0$ if $j\geq 1$. In this way we see that $\Phi$ maps $(\PATHP\times\PATHP)^+/\text{SL}_2(\mathbb{Z})$ into $\SL^+$.

It remains to prove that $\Psi$ maps $\SL^+$ into $(\PATHP\times\PATHP)^+/\text{SL}_2(\mathbb{Z})$. Suppose then that $\mathbf{M}$ is a positive $\text{SL}_2$-tiling. We will need the inequalities 
\begin{equation}\label{eqn4}
\frac{m_{i,0}}{m_{i,1}}>\frac{m_{i+1,0}}{m_{i+1,1}}
\quad\text{and}\quad
\frac{m_{0,j}}{m_{1,j}}>\frac{m_{0,j+1}}{m_{1,j+1}},
\end{equation}
for $i,j\in\mathbb{Z}$, which follow immediately from the equations $m_{i,j}m_{i+1,j+1}-m_{i,j+1}m_{i+1,j}=1$ and the fact that the entries of $\mathbf{M}$ are positive.

Observe that the image of  $\pm\mathbf{M}$ under $\Psi$ is the $\text{SL}_2(\mathbb{Z})$-orbit of the pair of bi-infinite paths $\gamma$ and $\delta$ with vertices $a_i/b_i$ and $c_j/d_j$, respectively, given by equations~\eqref{eqn3}. By the left-hand inequality of \eqref{eqn4}, the first path $\gamma$ is a decreasing sequence of positive rational numbers, so it is a clockwise bi-infinite path in $\mathscr{F}$. For the second path $\delta$, we have
\[
c_j=m_{1,0}m_{0,j}-m_{0,0}m_{1,j}\quad\text{and}\quad d_j=m_{1,1}m_{0,j}-m_{0,1}m_{1,j},
\]
for $j\in\mathbb{Z}$. In particular, $c_0=0$ and $d_0=1$, and $c_1=-1$ and $d_1=0$.  Observe that 
\[
d_j=m_{1,j}m_{1,1}\left(\frac{m_{0,j}}{m_{1,j}}-\frac{m_{0,1}}{m_{1,1}}\right).
\]
Using the right-hand inequality of \eqref{eqn4}, we see that $d_j>0$ for $j<1$, and $d_j<0$ for $j>1$. From the equations $c_jd_{j+1}-c_{j+1}d_j=1$, for $j\in\mathbb{Z}$, we can infer that the sequence $c_0/d_0, c_{-1}/d_{-1},\dotsc$ is increasing and the sequence $c_2/d_2,c_3/d_3,\dotsc$ is decreasing.

Next, since $a_id_j-b_ic_j=m_{i,j}>0$, for $i,j\in\mathbb{Z}$, we see that $a_i/b_i>c_j/d_j$ for $j<1$, and $a_i/b_i<c_j/d_j$ for $j>1$. It follows that the limits $\gamma_{-\infty}$, $\gamma_{\infty}$, $\delta_{-\infty}$ and $\delta_{\infty}$ of $\gamma$ and $\delta$ are real numbers that satisfy $\delta_{-\infty}\leq \gamma_{\infty}<\gamma_{-\infty}\leq \delta_{\infty}$. Therefore $\delta$, like $\gamma$, is a clockwise bi-infinite path in $\mathscr{F}$, and the $\text{SL}_2(\mathbb{Z})$-orbit of $(\gamma,\delta)$ lies in $(\PATHP\times\PATHP)^+/\text{SL}_2(\mathbb{Z})$, as required. 

This completes the proof of Theorem~\ref{thm2}.

\section{Tame infinite friezes}\label{sec4}

In this section we prove Theorem~\ref{thm3}, which says that $\Phi\colon \PATHP/\text{SL}_2(\mathbb{Z})\longrightarrow \FR_\infty$ is a bijective function. This map sends the $\text{SL}_2(\mathbb{Z})$-orbit of a bi-infinite path $\gamma$ (with vertices $a_i/b_i$ such that $a_ib_{i+1}-a_{i+1}b_i=1$) to $\pm\mathbf{M}$, where $\mathbf{M}$ is the tame $\text{SL}_2$-tiling  given by 
\[
m_{i,j}=a_jb_i-b_ja_i,
\]
for $i,j\in\mathbb{Z}$. Since $m_{j,i}=-m_{i,j}$, for $i,j\in\mathbb{Z}$, we see that $\mathbf{M}$ is a tame infinite frieze.

Conversely, if $\mathbf{M}$ is a tame infinite frieze, then it is a tame $\text{SL}_2$-tiling, so we can find two bi-infinite paths $\gamma$ and $\delta$ in $\mathscr{F}$ for which, with the usual notation,
\[
m_{i,j}=a_id_j-b_ic_j,
\]
for $i,j\in\mathbb{Z}$. Since $m_{i,i}=0$, for $i\in\mathbb{Z}$, it follows that $a_i/b_i=c_i/d_i$, so $\gamma=\delta$. Therefore the preimage of $\pm\mathbf{M}$ under $\Phi$ is an element of $\PATHP/\text{SL}_2(\mathbb{Z})$. This completes the proof of Theorem~\ref{thm3}.

We conclude this section with the following theorem, referred to already in the introduction.

\begin{theorem}\label{thmZ}
Let $\gamma$ be a bi-infinite path in $\mathscr{F}$ and let $\mathbf{M}$ be a tame infinite frieze such that $\widetilde{\Phi}(\gamma)=\pm \mathbf{M}$. Then the itinerary for $\gamma$ is equal to the quiddity sequence of $\mathbf{M}$ and $-\mathbf{M}$.
\end{theorem}

\begin{proof}
Let $\gamma=\langle\,\dotsc,v_{-1},v_0,v_1,v_2,\dotsc\rangle$ with itinerary $[\,\dotsc,e_{-1},e_0,e_1,e_2,\dotsc]$. As usual, we write $v_i$ as a reduced rational $a_i/b_i$, for $i\in\mathbb{Z}$, where $a_ib_{i+1}-a_{i+1}b_i=1$.

Consider some particular integer $i$. By applying an element of $\text{SL}_2(\mathbb{Z})$ to $\gamma$ (which preserves its itinerary) we can assume that $v_{i-1}=0$ and $v_i=\infty$, in which case $v_{i+1}=e_i$. In these circumstances either
\begin{align*}
&a_{i-1}=0,\quad  &&a_{i}=-1,\quad &&a_{i+1}=-e_i,\\
&b_{i-1}=1 ,\quad &&b_{i}=0,\quad &&b_{i+1}=-1,
\end{align*}
or else the sign of each of these numbers is reversed. In both cases we have
\[
m_{i+1,i-1}=a_{i-1}b_{i+1}-a_{i+1}b_{i-1}=e_i.
\]
Thus the $i$th entry of the quiddity sequence $m_{j+1,j-1}$, for $j\in\mathbb{Z}$, is equal to the $i$th entry of the itinerary for $\gamma$. It follows that the quiddity sequence and itinerary coincide. 
 \end{proof}

\section{Positive infinite friezes}\label{sec5}

In this section we prove Theorem~\ref{thm4}, which says that the map $\Phi\colon\PATHP^+/\text{SL}_2(\mathbb{Z})\longrightarrow \FR_\infty^+$ is a one-to-one correspondence. 

Consider a clockwise bi-infinite path $\gamma$, with vertices given by reduced rationals $a_i/b_i$, for $i\in\mathbb{Z}$. By applying an element of $\text{SL}_2(\mathbb{Z})$, we can assume that the vertex of $\gamma$ with index 1 is $\infty$, and $a_1=1$ and $b_1=0$. Let us choose the denominators $b_i$ to be positive if $i>1$, and negative if $i<1$. Then we can deduce that $a_ib_{i+1}-a_{i+1}b_i=1$, for $i\in\mathbb{Z}$.

Suppose now that $i,j\in\mathbb{Z}$ with $i>j$. If $i$ and $j$ are both greater than 1 or both less than 1, then $a_i/b_i < a_j/b_j$, so the tame infinite frieze $\mathbf{M}$ corresponding to $\gamma$ satisfies
\[
m_{i,j}=a_jb_i-b_ja_i=b_ib_j\left(\frac{a_j}{b_j}-\frac{a_i}{b_i}\right) >0.
\]
Similarly it can be shown that $m_{i,j}>0$ if $i>1>j$, or if $i>j$ and one of $i$ or $j$ equals 1. Hence $\mathbf{M}$ is a positive infinite frieze.

Conversely, suppose that $\mathbf{M}$ is a positive infinite frieze, and let $\gamma$ be the bi-infinite path in $\mathscr{F}$ given by equation~\eqref{eqn3} with vertices $a_i/b_i$, where $a_i=m_{i,0}$ and $b_i=m_{i,1}$, for $i\in\mathbb{Z}$ (which is in the preimage of $\pm\mathbf{M}$ under $\widetilde{\Phi}$). Observe that $a_0=0$, $b_0=-1$, $a_1=1$ and $b_1=0$. Also, $a_i$ and $b_i$ are negative for $i<0$ and they are positive for $i>1$. For $i>1$ we have
\[
\frac{a_i}{b_i}-\frac{a_{i+1}}{b_{i+1}}=\frac{1}{b_ib_{i+1}}(a_ib_{i+1}-a_{i+1}b_i)=\frac{1}{b_ib_{i+1}}>0,
\]
so $a_2/b_2,a_3/b_3,\dotsc$ is a decreasing sequence. Similarly, it can be shown that $a_0/b_0,a_{-1}/b_{-1},\dotsc$ is an increasing sequence. Furthermore, if $i>1$ and $j<0$ then $b_ib_j$ is negative, so
\[
\frac{a_i}{b_i}-\frac{a_j}{b_j} = -\frac{1}{b_ib_{j}}(a_jb_i-b_ja_i)=-\frac{m_{i,j}}{b_ib_{j}}>0.
\]
Hence $a_i/b_i>a_j/b_j$. It follows that $\gamma$ is a clockwise bi-infinite path. 

This completes the proof of Theorem~\ref{thm4}.

\section{Quiddity sequences of positive infinite friezes}\label{sec6}

Here we prove Theorem~\ref{thm5}, which characterises the quiddity sequences of positive infinite friezes. In the course of the proof, and elsewhere in this work, we often use the following elementary lemma, the proof of which is omitted.

\begin{lemma}\label{lemA}
Let $x$ and $y$ be adjacent vertices of $\mathscr{F}$. Then any finite path whose initial and final vertices lie in different components of $\mathbb{R}_\infty-\{x,y\}$ must pass through one or both of $x$ and $y$.
\end{lemma}

Next we define hyperbolic triangulated polygons, which are hyperbolic versions of Euclidean triangulated polygons. 

\begin{definition}
A \emph{hyperbolic triangulated polygon} comprises the sets of vertices and edges obtained from the union of finitely many triangles in the Farey graph, where any two of these triangles can be connected to one another by a finite sequence of successively adjacent triangles from the union.  
\end{definition}

There is another way to think about hyperbolic triangulated polygons, using simple closed paths in $\mathscr{F}$. Recall that a simple closed path in $\mathscr{F}$ is a path $\langle v_0,v_1,\dots,v_n\rangle$ in $\mathscr{F}$ such that $v_0=v_n$ and $v_0,v_1,\dots,v_{n-1}$ are distinct. The length of this path is $n$. It can easily be shown that the vertices of a hyperbolic triangulated polygon, listed in clockwise order, determine a clockwise simple closed path of length at least three (see \cite[Section~2.2]{MoOvTa2015} for an explanation). Conversely, any clockwise simple closed path of length at least three determines a hyperbolic triangulated polygon with vertices given by the vertices of the path, and the edges of the triangulated polygon are those edges of $\mathscr{F}$ that connect pairs of vertices from the path.

There is a correspondence between hyperbolic and Euclidean triangulated polygons, which is illustrated in Figure~\ref{fig18} and can be described (in brief) as follows. To obtain a Euclidean triangulated polygon from a hyperbolic triangulated polygon $\mathscr{T}$ we map the Poincar\'e disc model of the hyperbolic plane (which we are using) to the Klein disc model of the hyperbolic plane in the standard way (by projecting stereographically onto a hemisphere above the disc and then projecting orthogonally back onto the disc). This projection preserves the unit circle pointwise and it sends lines in the Poincar\'e model to straight Euclidean lines in the Klein model. Thus $\mathscr{T}$ is taken to a Euclidean triangulated polygon with the same associated triangle-counting cycle.

\begin{figure}[ht]
	\centering
	\begin{tikzpicture}
	\begin{scope}[scale=1.2,xshift=0cm]
	\newdimen\R
	\R=1.5cm
	\pgfmathsetmacro{\theta}{360/7}
	
	\draw[thick] (0:\R) \foreach[count=\i] \lab in {2,2,3,1,2,4,1} {
		-- (-\i*\theta:\R) node[shift=({-\i*\theta:6pt})] {$\lab$}
	}-- cycle (90:\R) ;
	\draw[thick] (-\theta:\R) -- (-6*\theta:\R);
	\draw[thick] (-6*\theta:\R) -- (-3*\theta:\R);
	\draw[thick] (-3*\theta:\R) -- (-5*\theta:\R);
	\draw[thick] (-2*\theta:\R) -- (-6*\theta:\R);
	\end{scope}
	
	\begin{scope}[scale=2,xshift=3cm]	
	\discfareygraph[thin,grey]{4};
	
	\hgline[thick](0,0)(0:90:\rad);
	\hgline[thick](0,0)(90:180:\rad);
	\hgline[thick](0,0)(180:270:\rad);
	\hgline[thick](0,0)(270:360:\rad);
	\hgline[thick](0,0)(270:323.13:\rad);
	\hgline[thick](0,0)(323.13:337.37:\rad);
	\hgline[thick](0,0)(337.37:360:\rad);
	\hgline[thick](0,0)(337.37:343.74:\rad);
	\hgline[thick](0,0)(343.74:360:\rad);
	\hgline[thick](0,0)(323.13:360:\rad);
	\draw[thick] (0,\rad) -- (0,-\rad);
	
	\node[] at (180:\radplusminus) {$1$};
	\node[] at (90:\radplusminus) {$2$};
	\node[] at (0:\radplusminus) {$4$};
	\node[yshift=2pt] at (343.74:\radplusminus) {\footnotesize$1$};
	\node[] at (337.37:\radplusminus) {\footnotesize$2$};
	\node[] at (323.13:\radplusminus) {$2$};
	\node[] at (270:\radplusminus) {$3$};

	\end{scope}
	\end{tikzpicture}
	\caption{Euclidean and hyperbolic triangulated polygons with the same triangle-counting cycles}
	\label{fig18}
\end{figure}
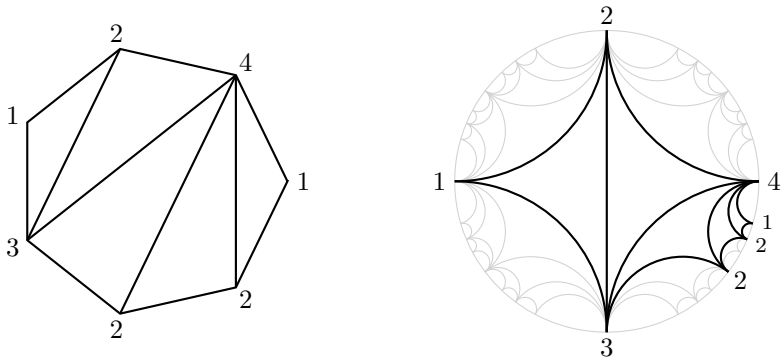

Reversing this procedure, given a Euclidean triangulated polygon, one can obtain a hyperbolic triangulated polygon with the same triangle-counting sequence by an inductive process, starting from a single triangle and then gluing further triangles one at time. This process is formalised in \cite{MoOv2019} (in particular, Section~2.4 of that paper), where an algorithm using Farey sums is used to achieve this purpose.

In this manner we see that for each Euclidean triangulated polygon there is a hyperbolic triangulated polygon of the same combinatorial type, and vice versa (this was observed in \cite{MoOvTa2015}). From now on we use hyperbolic triangulated polygons rather than Euclidean triangulated polygons.

Recall from the introduction that the itinerary of a finite path $\gamma=\langle v_0,v_1,\dots,v_n\rangle$ is the sequence $[e_1,e_2,\dots,e_{n-1}]$, where, for each index $i$, the integer $e_i$ is the image of $v_{i+1}$ under a map $g\in\text{SL}_2(\mathbb{Z})$ that satisfies $g(v_{i-1})=0$ and $g(v_i)=\infty$. Also, recall that a cycle sequence is  obtained by removing the final term from a quiddity sequence of some positive frieze. 

\begin{lemma}\label{lemB}
A finite path $\gamma$ in $\mathscr{F}$ is a clockwise simple closed path if and only if the itinerary of $\gamma$ is a cycle sequence.
\end{lemma}

\begin{proof}
Let $\gamma=\langle v_0,v_1,\dots,v_n\rangle$, with itinerary $[e_1,e_2,\dots,e_{n-1}]$. 

Suppose first that $\gamma$ is a clockwise simple closed path. Let $\mathscr{T}$ be the hyperbolic triangulated polygon with vertices those of $\gamma$. Let $t_i$ be the number of triangles in $\mathscr{T}$ incident to $v_i$. Then $t_1,t_2,\dots,t_{n-1}$ is a cycle sequence.

Consider any integer $i$ between (and including) 1 and $n-1$. By applying an element of $\text{SL}_2(\mathbb{Z})$, we can assume that $v_{i-1}=0$, $v_i=\infty$ and $v_{i+1}=e_i$. The sequence $v_{i-1},v_i,v_{i+1}$ is in clockwise order, so $e_i$ is a positive integer, and the remaining vertices of $\gamma$ lie between $0$ and $e_i$. By applying Lemma~\ref{lemA} to the path $\langle v_{i+1},v_{i+2},\dots,v_n,v_1,\dots,v_{i-1}\rangle$ with $x=\infty$ and $y$ equal to each of the integers $1,2,\dots,e_i-1$ in turn, we see that $\gamma$ must pass through all the vertices $0,1,\dots,e_i$. It follows that there are precisely $e_i$ triangles in $\mathscr{T}$ that are incident to $v_i$. Hence $t_i=e_i$, and therefore the itinerary of $\gamma$ is a cycle sequence. 

Conversely, suppose that $e_1,e_2,\dots,e_{n-1}$ is a cycle sequence. Then there is some hyperbolic triangulated polygon $\mathscr{T}$ with vertices $w_1,w_2,\dots,w_{n}$ such that $\delta=\langle w_0,w_1,\dots,w_n\rangle$ (where $w_0=w_n$) is a clockwise simple closed path and such that there are $e_i$ triangles in $\mathscr{T}$ incident to $w_i$, for $i=1,2,\dots,n-1$. 

By applying an element of $\text{SL}_2(\mathbb{Z})$ to $\gamma$ we can assume that $v_0=w_0$ and $v_1=w_1$. Now choose $g\in\text{SL}_2(\mathbb{Z})$ such that $g(v_0)=0$ and $g(v_1)=\infty$. Then $g(v_2)=e_1$, by definition of $e_1$, but also $g(w_2)=e_1$, because there are $e_1$ triangles in $\mathscr{T}$ incident to $w_1$. Hence $v_2=w_2$. Repeating this argument we see that $\gamma$ and $\delta$ are equal, so $\gamma$ is a clockwise simple closed path.
\end{proof}

We can now prove Theorem~\ref{thm5}, which says that a bi-infinite sequence of positive integers is the quiddity sequence of a positive infinite frieze if and only if it is acyclic. 

Choose any bi-infinite sequence of positive integers $\dotsc,e_{-1},e_0,e_1,e_2,\dotsc$ and let 
\[
\gamma=\langle\,\dotsc,v_{-1},v_0,v_1,v_2,\dotsc\rangle
\]
 be a bi-infinite path in $\mathscr{F}$ with itinerary $[\,\dotsc,e_{-1},e_0,e_1,e_2,\dotsc]$. For each $j\in\mathbb{Z}$, let $I_j$ denote the open interval of those points $x$ for which $v_{j-1},x,v_{j}$ are in clockwise order in $\mathbb{R}_\infty$. Since $e_j>0$ for $j\in\mathbb{Z}$ we can see (by mapping $v_{j-1}$ to $0$, $v_{j}$ to $\infty$, and $v_{j+1}$ to $e_j$) that $I_j$ and $I_{j+1}$ are disjoint. Two possibilities arise: either all the intervals $I_j$ are disjoint, or they are not. Let us consider each case in turn.

If the intervals $I_j$ are all disjoint, then $\gamma$ is a clockwise bi-infinite path (as defined in the introduction). It follows that $\gamma$ does not contain a clockwise simple closed path as a subpath of contiguous vertices of $\gamma$, so the itinerary of $\gamma$ does not contain a cycle sequence as a subsequence of contiguous terms, by Lemma~\ref{lemB}. Therefore the quiddity sequence of the corresponding frieze is acyclic.

Suppose now that the intervals $I_j$ are not disjoint. Then there are indices $j$ and $k$ with $k>j+1$ for which $v_k$ is either equal to $v_j$, or else it lies in $I_j$. However, by Lemma~\ref{lemA}, the path $\gamma$ cannot enter $I_j$ without first passing through $v_{j-1}$ or $v_{j}$. Thus if $v_k\in I_j$, then there is some index $l$ with $j+1<l<k$ for which $v_l$ is equal to one of $v_{j-1}$ or $v_{j}$. Thus we obtain a clockwise simple closed path as a subpath of contiguous vertices of $\gamma$. We can now apply Lemma~\ref{lemB} to see that the itinerary of $\gamma$ contains a cycle sequence, so it is not acyclic. 

This completes the proof of Theorem~\ref{thm5}.

\section{Tame  friezes}\label{sec7}

Here we prove that $\Phi$ is a bijective map from $\mathscr{C}_n/\text{SL}_2(\mathbb{Z})$ to $\mathbf{FR}_n$ (Theorem~\ref{thm6}). 

Suppose that $\langle v_0,v_1,\dots,v_n\rangle$ is a closed path in $\mathscr{F}$ (so $v_n=v_0$). As stated in the introduction, we identify this closed path with the periodic bi-infinite path $\langle\,\dotsc,v_{-1},v_0,v_1,v_2,\dotsc\rangle$, where $v_i= v_j$ if $i\equiv j\pmod{n}$. Using the usual notation, we write $v_i=a_i/b_i$ with $a_{i}b_{i+1}-a_{i+1}b_{i}=1$, for $i\in\mathbb{Z}$. For each integer $i$, we have $v_{i+n}=v_i$, so $a_{i+n}/b_{i+n}=a_i/b_i$, and hence $m_{i+n,i}=a_ib_{i+n}-b_ia_{i+n}=0$. Since $m_{i,i}=0$ also, we see that $\mathbf{M}$ is a tame frieze of order $n$.

Conversely, given a tame frieze $\mathbf{M}$ of order $n$ we have already seen in the introduction that the preimage of $\pm \mathbf{M}$ under $\Phi$ is the $\text{SL}_2(\mathbb{Z})$-orbit of a periodic bi-infinite path of period $n$. This completes the proof of Theorem~\ref{thm6}.


\section{Positive friezes}\label{sec7a}

We prove here that $\Phi\colon \mathscr{C}_n^+/\text{SL}_2(\mathbb{Z})\longrightarrow\mathbf{FR}_n^+$ is a one-to-one correspondence (Theorem~\ref{thm7}). 

To this end, suppose that $\gamma=\langle v_0,v_1,\dots,v_n\rangle$ is a clockwise simple closed path in $\mathscr{F}$, and consider the periodic bi-infinite path $\langle\,\dotsc,v_{-1},v_0,v_1,v_2,\dotsc\rangle$ defined by the property that $v_i= v_j$ if $i\equiv j\pmod{n}$. As ever, we write $v_i$ as a reduced rational $a_i/b_i$, for $i\in\mathbb{Z}$, where $a_ib_{i+1}-a_{i+1}b_i=1$, and we let $\mathbf{M}$ be the tame frieze of order $n$ with entries $m_{i,j}= a_jb_i-b_ja_i$, for $i,j\in\mathbb{Z}$.

Choose two particular integers $i$ and $j$ with $0<i-j<n$. By applying an element of $\text{SL}_2(\mathbb{Z})$, we can assume that $a_j=1$ and $b_j=0$ (and $v_j=\infty$). Then  $b_{j+1}=a_jb_{j+1}-a_{j+1}b_j=1$. Since $\gamma$ is a clockwise simple closed path of length $n$, we see that $v_{j+1}>v_{j+2}>\dots>v_i$. A simple inductive argument shows that $b_{j+1},b_{j+2},\dots,b_i$ are all positive, so $m_{i,j}=b_i>0$.

In summary, $\mathbf{M}$ is a tame frieze of order $n$ that satisfies $m_{i,j}>0$, for $0<i-j<n$; hence $\mathbf{M}$ is a positive frieze of order $n$.

Conversely, suppose that $\mathbf{M}$ is a positive frieze of order $n$. By  Theorem~\ref{thm6} we can find a closed path  $\gamma=\langle v_0,v_1,\dots,v_n\rangle$, where $v_0=v_n$ and $v_i=a_i/b_i$, for $i=0,1,\dots,n$, with $a_{i}b_{i+1}-a_{i+1}b_{i}=1$, for $i=0,1,\dots,n-1$, such that $m_{i,j}=a_jb_i-b_ja_i$, for $0\leq i,j\leq n$. Furthermore, by applying an element of $\text{SL}_2(\mathbb{Z})$ we can assume that $v_0=v_n=\infty$ and $a_0=1$ and $b_0=0$. Observe that, for $0<i<n$, 
\[
b_i = a_0b_i-b_0a_i=m_{i,0}>0.
\]
Hence
\[
\frac{a_i}{b_i}-\frac{a_{i+1}}{b_{i+1}}=\frac{1}{b_ib_{i+1}}(a_ib_{i+1}-a_{i+1}b_i)=\frac{1}{b_ib_{i+1}}>0,
\]
for $0<i<n-1$. That is, $v_1>v_2>\dots>v_{n-1}$, so $\gamma$ is a clockwise simple closed path, as required. This completes the proof of Theorem~\ref{thm7}.
 
\section{Antiperiodic tilings}\label{sec7c} 

This section has a proof that $\Phi$ is a bijective map from $(\mathscr{C}_r\times\mathscr{C}_s)^+/\text{SL}_2(\mathbb{Z})$ to $\SL^+(r,s)$ (Theorem~\ref{thm8}). 

Before proving this theorem we make a simple observation. Let $S$ be the shift map acting on the collection of bi-infinite paths $\mathscr{P}$ in the Farey graph; $S$ sends the path with vertices $v_i$ to the path with vertices $v_i'=v_{i-1}$. Note that $S$ is a bijective map, so, for integers $p$ and $q$, we can define an action of $(S^p,S^q)$ on $(\mathscr{P}\times\mathscr{P})/\text{SL}_2(\mathbb{Z})$ in the obvious way. 

Next, we define $T^{p,q}$ to be the self-map of $\SL$ that sends $\pm\mathbf{M}$ to $\pm\mathbf{M}'$, where $m'_{i,j}=m_{i-p,j-q}$, for $i,j\in\mathbb{Z}$. It is easily verified that
\[
\Phi\circ (S^p,S^q)= T^{p,q}\circ \Phi.
\]
Since $(\mathscr{C}_r\times\mathscr{C}_s)^+/\text{SL}_2(\mathbb{Z})$ is invariant under $(S^p,S^q)$ and $\SL^+(r,s)$ is invariant under $T^{p,q}$, we can use the displayed equation above to simplify our proof of Theorem~\ref{thm8}.

Let us now prove that theorem. Suppose then that $(\gamma,\delta)\in (\mathscr{C}_r\times\mathscr{C}_s)^+$; that is, suppose that $\gamma\in\mathscr{C}_r^+$, $\delta\in\mathscr{C}_s^+$, and $\gamma$ and $\delta$ do not intersect. We write $\gamma=\langle\,\dotsc,v_{-1},v_0,v_1,v_2,\dotsc\rangle$, where $\langle v_0,v_1,\dots,v_r\rangle$ is a clockwise simple closed path and $v_i=v_j$ for $i\equiv j\pmod{r}$, and $\delta=\langle\,\dotsc,w_{-1},w_0,w_1,w_2,\dotsc\rangle$, where $\langle w_0,w_1,\dots,w_s\rangle$ is a clockwise simple closed path and $w_i=w_j$ for $i\equiv j\pmod{s}$. We use the usual notation $v_i=a_i/b_i$ and $w_j=c_j/d_j$. Since $\gamma$ and $\delta$ do not intersect, we can apply Lemma~\ref{lemA} to see that there are two consecutive vertices $v_{p-1}$ and $v_p$ of $\gamma$ and two consecutive vertices $w_{q-1}$ and $w_q$ of $\delta$ such that $\delta$ is contained in one of the components of $\mathbb{R}_\infty - \{v_{p-1},v_p\}$ and $\gamma$ is contained in one of the components of $\mathbb{R}_\infty - \{w_{q-1},w_q\}$. Then, by applying $(S^p,S^q)$ to $(\gamma,\delta)$, we can assume that $p=q=1$.

Next, by applying an element of $\text{SL}_2(\mathbb{Z})$, we can assume that $w_{0}=\infty$ and $w_1=0$. Then $w_1,w_2,\dots,w_{s-1}$ is a decreasing sequence of negative numbers, and of course $w_{s}=\infty$. Recall that there are two possible choices for the integer sequences $c_j$ and $d_j$ that differ only in sign. We choose the sequence for which $c_{s}=-1$ (and $d_{s}=0$). That $w_1,w_2,\dots,w_{s-1}$ is a decreasing sequence, together with the equations $c_jd_{j+1}-c_{j+1}d_j=1$, then implies that $d_1,d_2,\dots,d_{s-1}$ are all positive. Furthermore, we can see that $c_{j+s}=-c_j$ and $d_{j+s}=-d_j$, for $j\in\mathbb{Z}$ (the sign changes each cycle).

Similarly, it can be shown that $v_1,v_2\dots,v_{r}$ is a decreasing sequence of positive numbers, that $b_1,b_2,\dots,b_{r}$ are  positive numbers, and that $a_{i+r}=-a_i$ and $b_{i+r}=-b_i$, for $i\in\mathbb{Z}$.

Let $\mathbf{M}$ be the tame $\text{SL}_2$-tiling with entries $m_{i,j}=a_id_j-b_ic_j$, for $i,j\in\mathbb{Z}$. Then $m_{i+r,j}=a_{i+r}d_j-b_{i+r}c_j=-a_id_j+b_ic_j=-m_{i,j}$, and likewise
$m_{i,j+s}=-m_{i,j}$, for $i,j\in\mathbb{Z}$. Furthermore, for $i=1,2,\dots, r$ and $j=1,2,\dots,s-1$, we have $v_i>w_j$, so
\[
m_{i,j}=a_id_j-b_ic_j= \frac{1}{b_id_j}\left(\frac{a_i}{b_i}-\frac{c_j}{d_j}\right)>0.
\]
Also, $m_{i,s}=a_id_{s}-b_ic_{s}=b_i>0$. Hence $m_{i,j}>0$, for $i=1,2,\dots,r$ and $j=1,2,\dots,s$, so $\pm\mathbf{M}\in\SL^+(r,s)$.

Conversely, consider a tame $\text{SL}_2$-tiling $\mathbf{M}\in \SL^+(r,s)$. Then $\mathbf{M}$ satisfies the antiperiodicity rules $m_{i+r,j}=m_{i,j+s}=-m_{i,j}$, for $i,j\in\mathbb{Z}$. Also, by applying a suitable shift map $T^{p,q}$ to $\mathbf{M}$, we can assume that $m_{i,j}>0$, for $i=1,2,\dots,r$ and $j=1,2,\dots,s$. Let $\gamma$ and $\delta$ be the bi-infinite paths in $\mathscr{F}$ defined by equations~\eqref{eqn3}, with the usual notation for paths. Then $a_{i+r}=m_{i+r,0}=-m_{i,0}=-a_i$ and $b_{i+r}=m_{i+r,1}=-m_{i,1}=-b_i$, for $i\in\mathbb{Z}$, so $\gamma$ is periodic with period $r$. Similarly, $\delta$ is periodic with period $s$. Observe that $\gamma$ and $\delta$ do not intersect, because $\mathbf{M}$ has no entries $0$.

Now, $a_ib_{i+1}-a_{i+1}b_i=1$, for $i=1,2,\dots,r-1$, and $b_i$ and $b_{i+1}$ are positive, so, after dividing throughout by $b_ib_{i+1}$ we see that 
\[
\frac{a_1}{b_1}>\frac{a_2}{b_2}>\dots>\frac{a_r}{b_r}.
\]
We know that $v_0=v_r$ and hence $\langle v_0,v_1,\dots,v_r\rangle$ is a clockwise simple closed path. In a similar manner it can be shown that $\langle w_0,w_1,\dots,w_s\rangle$ is a clockwise simple closed path. Hence $(\gamma,\delta)\in (\mathscr{C}_r\times \mathscr{C}_s)^+$.

This completes the proof of Theorem~\ref{thm8}.

\section{Tame friezes with positive quiddity sequences}\label{sec7b} 
 
In this section we prove Theorem~\ref{thm9}, which says that  $\Phi$ is a bijective map from $\mathscr{C}_n^0/\text{SL}_2(\mathbb{Z})$ to $\FR_n^0$. We use the following lemma.

\begin{lemma}\label{lemK}
Let $\gamma=\langle\,\dotsc,v_{-1},v_0,v_1,v_2,\dotsc\rangle$ be a bi-infinite path, with $v_i=a_i/b_i$ and $a_ib_{i+1}-a_{i+1}b_i=1$, for $i\in\mathbb{Z}$, and let $\mathbf{M}$ be the corresponding infinite frieze with entries $m_{i,j}=a_jb_i-b_ja_i$, for $i,j\in\mathbb{Z}$. For each integer $i$, the sequence $v_{i-1},v_i,v_{i+1}$ is in clockwise order if and only if $m_{i+1,i-1}>0$.	
\end{lemma}

\begin{proof}
Choose any integer $i$. By applying an element of $\text{SL}_2(\mathbb{Z})$ to $\gamma$, which will not affect $\mathbf{M}$, we can specify that $v_i=\infty$. Let us also fix signs such that $a_i=1$ and $b_i=0$. Then $b_{i-1}=-(a_{i-1}b_{i}-a_{i}b_{i-1})=-1$ and $b_{i+1}=a_{i}b_{i+1}-a_{i+1}b_{i}=1$. Hence $v_{i-1}=-a_{i-1}$ and $v_{i+1}=a_{i+1}$. The sequence $v_{i-1},v_i,v_{i+1}$ is in clockwise order if and only if $v_{i-1}<v_{i+1}$; that is, if and only if $-a_{i-1}<a_{i+1}$. Since
\[
m_{i+1,i-1}=a_{i-1}b_{i+1}-b_{i-1}a_{i+1}=a_{i-1}+a_{i+1},
\]
we see that $v_{i-1},v_i,v_{i+1}$  is in clockwise order if and only if $m_{i+1,i-1}>0$.
\end{proof}

Let us now prove Theorem~\ref{thm9}. Applying Theorem~\ref{thm6}, we choose a periodic bi-infinite path $\gamma=\langle\,\dotsc,v_{-1},v_0,v_1,v_2,\dotsc\rangle$ of period $n$ and a corresponding tame frieze $\mathbf{M}$ of order $n$ with entries $m_{i,j}=a_jb_i-b_ja_i$. By Lemma~\ref{lemK}, the sequence $v_{i-1},v_i,v_{i+1}$ is in clockwise order for all $i\in\mathbb{Z}$ if and only if $m_{i+1,i-1}>0$ for all $i\in\mathbb{Z}$. That is, $\gamma \in \mathscr{C}_n^0$ if and only if $\mathbf{M}\in \FR_n^0$. This completes the proof of Theorem~\ref{thm9}.

\section{Combinatorics of positive infinite friezes}\label{sec8}

In this section we resume the discussion of the introduction relating our geometric characterisation of positive infinite friezes using the Farey graph to the combinatorial model of such friezes from \cites{BaPaTs2016,BaPaTs2018}. The class of combinatorial objects that we need are triangulations of infinitely-many-sided polygons -- or apeirogons, as they are often called. 

\begin{definition}
A \emph{hyperbolic triangulated apeirogon} comprises the sets of vertices and edges obtained from the union of infinitely many triangles in the Farey graph, where any two of these triangles can be connected to one another by a finite sequence of successively adjacent triangles from the union.  
\end{definition}

We usually omit the adjective `hyperbolic' and simply write `triangulated apeirogon'.

To appreciate the structure of triangulated apeirogons, it helps to consider the dual graph of the Farey graph, which is an infinite trivalent tree. Each vertex of this tree corresponds to a triangle of the Farey graph, and a triangulated apeirogon corresponds to an infinite connected subtree of the dual graph.

Let us explore a method for constructing a triangulated apeirogon $\mathscr{A}$ from a clockwise bi-infinite path $\gamma=\langle\,\dotsc,v_{-1},v_0,v_1,v_2\dotsc\rangle$. For each integer $j$, the open interval $I_j$ in $\mathbb{R}_\infty$ that runs clockwise from $v_{j-1}$ to $v_{j+1}$ contains $v_j$ and all but finitely many neighbours of $v_j$ in the Farey graph $\mathscr{F}$. The set of vertices of $\mathscr{A}$ consists of all the vertices $v_j$, and, for each integer $j$, all the neighbours of $v_j$ in $\mathscr{F}$ that lie outside $I_j$ (some of which may well be other vertices of $\gamma$). The edges of $\mathscr{A}$ are those edges of $\mathscr{F}$ that connect the vertices of $\mathscr{A}$. See Figures~\ref{fig15} and~\ref{fig16} for illustrations of this construction.

It follows quickly from the definition that $\mathscr{A}$ is a triangulated apeirogon. The vertices of $\gamma$ occur in clockwise order around $\mathscr{A}$, with no other vertices of $\mathscr{A}$ in between. The remaining vertices of $\mathscr{A}$ themselves form a path, as Theorem~\ref{thmXYZ}, to follow shortly, confirms. Before stating this theorem, we introduce an elementary lemma.

\begin{lemma}\label{lemT}
Given any irrational number $\alpha$ there is an increasing sequence of rationals $x_1,x_2,\dotsc$ with limit $\alpha$ and a decreasing sequence of rationals $y_1,y_2,\dotsc$ with limit $\alpha$ such that $x_n$ and $y_n$ are adjacent in $\mathscr{F}$, for each $n=1,2,\dotsc$. 
\end{lemma}

This lemma could be proved, for example, by choosing $(x_n)$ and $(y_n)$ to be the odd and even subsequences of the sequence of convergents of the regular continued fraction expansion of $\alpha$. We omit the details.

In the next theorem we describe a clockwise path as \emph{forward-finite} if it has a final vertex and \emph{forward-infinite} if it has no final vertex and so continues indefinitely in the forward direction.

\begin{theorem}\label{thmXYZ}
Let $\gamma$ be a clockwise bi-infinite path and let $\mathscr{A}$ be the associated triangulated apeirogon. Those vertices of $\mathscr{A}$ that are not in $\gamma$ form a clockwise path $\delta$.
\begin{enumerate}
\item Suppose that $\gamma_{-\infty}=\gamma_\infty$. If $\gamma_{-\infty}$ is rational, then $\delta$ comprises the single vertex $\gamma_{-\infty}$, and if $\gamma_{-\infty}$ is irrational then $\delta$ is the empty set.
\item Suppose that $\gamma_{-\infty}\neq\gamma_\infty$. If $\gamma_{-\infty}$ is rational, then $\delta$ is forward-finite with final vertex $\gamma_{-\infty}$, and if $\gamma_{-\infty}$ is irrational then $\delta$ is forward-infinite with forward limit $\gamma_{-\infty}$.

An analogous statement holds for the backward end of $\delta$.
\end{enumerate}
\end{theorem}

\begin{proof}
Let $\gamma=\langle\dotsc,v_{-1},v_0,v_1,v_2,\dotsc\rangle$, and let $\mathscr{B}$ denote the set of those vertices of $\mathscr{A}$ that are not vertices of $\gamma$.

Suppose first that $\gamma_{-\infty}=\gamma_\infty$. Since there are no vertices of $\mathscr{A}$ in-between any two  vertices $v_{i-1}$ and $v_i$ of $\gamma$ (in a clockwise sense), we see that $\mathscr{B}$ comprises at most one point, $\gamma_{-\infty}$ itself. If $\gamma_{-\infty}$ is irrational then it cannot be a vertex of $\mathscr{A}$. If $\gamma_{-\infty}$ is rational, then, after applying an element of $\text{SL}_2(\mathbb{Z})$ we can assume that $\gamma_{-\infty}=\infty$. In this case we see from Lemma~\ref{lemA} (with $x=0$ and $y=\infty$) that $\gamma$ must pass through 0, so $\infty\in\mathscr{B}$, because $\infty$ is adjacent to $0$ in $\mathscr{F}$. This proves statement (i).

Suppose now that $\gamma_{-\infty}\neq\gamma_\infty$. The set $\mathscr{B}$ lies within the closed interval that runs clockwise from $\gamma_\infty$ to $\gamma_{-\infty}$. An argument similar to that just applied shows that $\gamma_{-\infty}\in\mathscr{B}$ if (and only if) it is rational. Next, the set $\mathscr{B}$ cannot accumulate at any points of $\mathbb{R}_\infty$ other than $\gamma_{-\infty}$ and $\gamma_\infty$. This follows from the fact that $\gamma$ accumulates only at $\gamma_{-\infty}$ and $\gamma_\infty$, and, in the disc model of the Farey graph, there are only finitely many edges of the Farey graph of Euclidean diameter greater than any given positive number. As a consequence, we can write the elements of $\mathscr{B}$ as $\dotsc,w_{-1},w_0,w_1,w_2,\dotsc$, where this sequence is in clockwise order and it may be finite, half-infinite or bi-infinite.

Consider now two consecutive terms $w_{i-1}$ and $w_{i}$; we will prove that they are adjacent in $\mathscr{F}$. Let $j$ be the largest index such that $v_j$ is adjacent to $w_i$ in $\mathscr{F}$. After applying an element of $\text{SL}_2(\mathbb{Z})$ we can assume that $v_j=0$ and $w_i=\infty$; in which case, by the clockwise ordering, $v_{j+1}$ is equal to $-1/n$, for some positive integer $n$, and $w_{i-1}$ is a negative integer. Suppose that $\gamma_\infty<-1$. Then, by Lemma~\ref{lemA} (with $x=-1$ and $y=\infty$), we must have $v_k=-1$ for some $k>j$. But $-1$ is adjacent to $\infty$, so this contradicts the definition of $j$. Hence $\gamma_\infty\geq -1$. Now, since $0$ is a vertex of $\gamma$, and $-1$ is adjacent to $0$, we see that $-1$ is a vertex of $\mathscr{B}$. On the other hand, none of the integers $-2,-3,\dotsc$ are adjacent to any vertex of $\gamma$, so they are not vertices of $\mathscr{B}$. Hence $w_{i-1}=-1$, and $w_{i-1}$ and $w_{i}$ are indeed adjacent in $\mathscr{F}$. It follows that $\delta=\langle\,\dotsc,w_{-1},w_0,w_1,w_2,\dotsc\rangle$ is a clockwise path in $\mathscr{F}$.

If $\gamma_{-\infty}$ is rational, then (as we have seen) it is a vertex of $\delta$, and it must be the final vertex, since $\delta$ is a clockwise path. If $\gamma_{-\infty}$ is equal to some irrational $\alpha$, then we can choose increasing and decreasing sequences of rationals $(x_n)$ and $(y_n)$ both with limit $\alpha$ such that $x_n$ and $y_n$ are adjacent in $\mathscr{F}$ for each index $n$. By Lemma~\ref{lemA}, the path $\gamma$ must pass through one of $x_n$ or $y_n$ for all but finitely many indices $n$. In fact, since $x_n<\alpha<y_n$ we see that $\gamma$ must pass through $x_n$ (for large $n$). Hence $y_n$ is a vertex of $\delta$, because it is adjacent to $x_n$ in $\mathscr{F}$. In this way we see that $\delta$ is forward-infinite, with forward limit $\delta_\infty=\gamma_{-\infty}$. This proves statement (ii).
\end{proof}

The path $\delta$ in Theorem~\ref{thmXYZ} is in a sense dual to $\gamma$. This concept of a dual path in this context could be formalised, but we will not do so here.

We illustrate Theorem~\ref{thmXYZ} with two examples, shown in Figure~\ref{fig15}. The first of these two figures is repeated from Figure~\ref{fig11}(a).

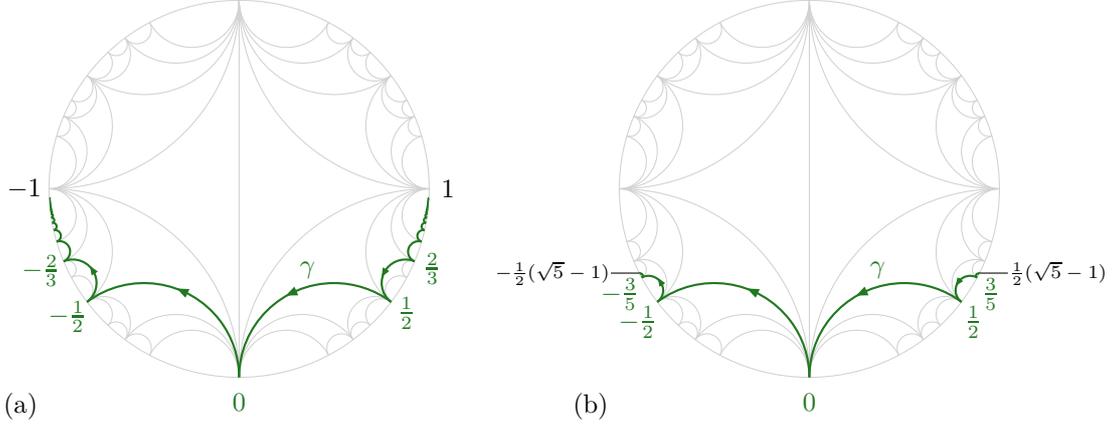
\begin{figure}[ht]
\centering
\begin{tikzpicture}[scale=2.5]

\begin{scope}
\discfareygraph[thin,grey]{4};

\node[] at (0:\radplusminus) {$1$};
\node[] at (180:\radplus) {$-1$};

\node[gammacol] at (-90:\radplus) {$0$};
\node[gammacol] at ({2*atan(1/2)-90}:\radplusminus) {$\tfrac12$};
\node[gammacol] at ({-2*atan(1/2)-90}:\radplus) {$-\tfrac12$};
\node[gammacol] at ({2*atan(2/3)-90}:\radplusminus) {$\tfrac23$};
\node[gammacol] at ({-2*atan(2/3)-90}:\radplus) {$-\tfrac23$};

\foreach \n in {1,...,20}{
	\hgline[gammacol,thick](0,0)(2*atan((\n-1)/\n)-90:2*atan(\n/(\n+1))-90:\rad);
	\hgline[gammacol,thick](0,0)(-2*atan(\n/(\n+1))-90:-2*atan((\n-1)/\n)-90:\rad);
}

\dirhgline[gammacol,draw=none,thick](0.5,0.8,204)(-90:2*atan(1/2)-90:\rad);
\dirhgline[gammacol,draw=none,thick](0.5,0.4,240)(2*atan(1/2)-90:2*atan(2/3)-90:\rad);
\dirhgline[gammacol,draw=none,thick](0.48,0.8,155)(-2*atan(1/2)-90:-90:\rad);
\dirhgline[gammacol,draw=none,thick](0.5,0.4,122)(-2*atan(2/3)-90:-2*atan(1/2)-90:\rad);

 \node[gammacol] at (-50:0.56*\rad) {$\gamma$};

\node at (-1.15,-1.15) {(a)}; 
\end{scope}

\begin{scope}[xshift=3cm]
\discfareygraph[thin,grey]{4};

\draw[thin] (-26.565:\rad) -- ++(0.15,0) node[right,xshift=-3pt,yshift=0pt] {{\scriptsize$\tfrac12(\sqrt5-1)$}};
\draw[thin] (206.565:\rad) -- ++(-0.15,0) node[left,xshift=3pt,yshift=0pt] {{\scriptsize$-\tfrac12(\sqrt5-1)$}};

\node[gammacol] at (-90:\radplus) {$0$};
\node[gammacol,xshift=-3pt,yshift=-2pt] at ({2*atan(1/2)-90}:\radplus) {$\tfrac12$};
\node[gammacol,yshift=-2pt] at ({-2*atan(1/2)-90}:\radplus) {$-\tfrac12$};
\node[gammacol,xshift=-3pt,yshift=-2pt] at ({2*atan(3/5)-90}:\radplus) {$\tfrac35$};
\node[gammacol,yshift=-2pt] at ({-2*atan(3/5)-90}:\radplus) {$-\tfrac35$};

\def\fib{{0,1,1,2,3,5,8,13,21,34,55,89,144,233}}
\foreach \n in {0,...,4}{
	\hgline[gammacol,thick](0,0)(2*atan(\fib[2*\n]/\fib[2*\n+1])-90:2*atan(\fib[2*\n+2]/\fib[2*\n+3])-90:\rad);
	\hgline[gammacol,thick](0,0)(-2*atan(\fib[2*\n+2]/\fib[2*\n+3])-90:-2*atan(\fib[2*\n]/\fib[2*\n+1])-90:\rad);
}

\dirhgline[gammacol,draw=none,thick](0.5,0.8,204)(-90:2*atan(1/2)-90:\rad);
\dirhgline[gammacol,draw=none,thick](0.55,0.3,235)(2*atan(1/2)-90:2*atan(3/5)-90:\rad);
\dirhgline[gammacol,draw=none,thick](0.48,0.8,155)(-2*atan(1/2)-90:-90:\rad);
\dirhgline[gammacol,draw=none,thick](0.52,0.3,122)(-2*atan(3/5)-90:-2*atan(1/2)-90:\rad);

 \node[gammacol] at (-50:0.56*\rad) {$\gamma$};

\node at (-1.15,-1.15) {(b)}; 
\end{scope}

\end{tikzpicture}
\caption{Two pairs of clockwise bi-infinite paths in the Farey graph}
\label{fig15}
\end{figure}

Figure~\ref{fig15}(a) illustrates the path
\[
\gamma =\left\langle\,\dotsc ,\tfrac34,\tfrac23,\tfrac12,0,-\tfrac12,-\tfrac23,-\tfrac34,\dotsc \right\rangle 
\]
with vertices $a_i/b_i$, where $a_i=-i$ and $b_i=|i|+1$, for $i\in\mathbb{Z}$. For this path, the limits $\gamma_{-\infty}=1$ and $\gamma_\infty=-1$  are both rational. The triangulated apeirogon constructed from $\gamma$ is shown in Figure~\ref{fig16}(a). The dual path $\delta=\langle -1,\infty,1\rangle$ is a finite path.

Figure~\ref{fig15}(b) illustrates the path
\[
\gamma =\left\langle\,\dotsc ,\tfrac8{13},\tfrac35,\tfrac12,0,-\tfrac12,-\tfrac35,-\tfrac8{13},\dotsc \right\rangle
\]
with vertices $a_i/b_i$, where $a_i=-\sgn(i)F_{2|i|}$ and $b_i=F_{2|i|+1}$, for $i\in\mathbb{Z}$ (here $\sgn(i)$ is $-1$ if $i<0$, 0 if $i=0$, and $1$ if $i>0$, and $F_n$ is the $n$th Fibonacci number). In this case the limits $\gamma_{-\infty} = \tfrac12(\sqrt5-1)$ and $\gamma_\infty = -\tfrac12(\sqrt5-1)$ are both irrational. The triangulated apeirogon constructed from $\gamma$ is shown in Figure~\ref{fig16}(b). The dual path
\(
\delta = \left\langle \dotsc,-\tfrac58,-\tfrac23,-1,\infty,1,\tfrac23,\tfrac58,\dotsc\right\rangle
\)
is a bi-infinite path. 

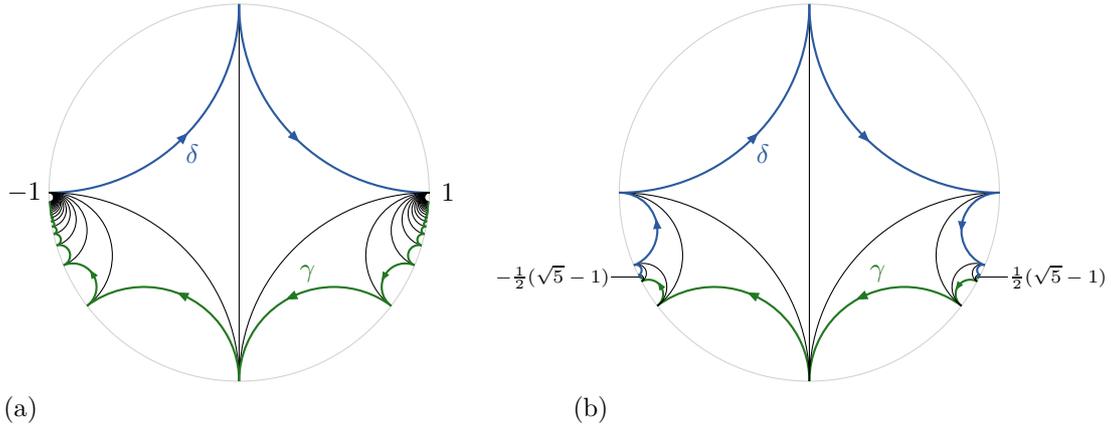
\begin{figure}[ht]
\centering
\begin{tikzpicture}[scale=2.5]

\begin{scope}
\draw[thin,grey] (0,0) circle(\rad);
\node[] at (0:\radplusminus) {$1$};
\node[] at (180:\radplus) {$-1$};

\hgline[thin,black](0,0)(-90:90:\rad);
\hgline[deltacol,thick](0,0)(0:90:\rad);
\hgline[deltacol,thick](0,0)(90:180:\rad);
\dirhgline[deltacol,draw=none,thick](0.5,0.8,45)(90:180:\rad);
\dirhgline[deltacol,draw=none,thick](0.5,0.8,-45)(0:90:\rad);

\node[deltacol] at (140:0.32*\rad) {$\delta$};

\foreach \n in {0,...,20}{
	\hgline[black](0,0)(2*atan(\n/(\n+1))-90:0:\rad);
	\hgline[black](0,0)(180:2*atan(\n/(\n+1))+90:\rad);
}

\foreach \n in {1,...,20}{
	\hgline[gammacol,thick](0,0)(2*atan((\n-1)/\n)-90:2*atan(\n/(\n+1))-90:\rad);
	\hgline[gammacol,thick](0,0)(-2*atan(\n/(\n+1))-90:-2*atan((\n-1)/\n)-90:\rad);
}

\dirhgline[gammacol,draw=none,thick](0.5,0.8,204)(-90:2*atan(1/2)-90:\rad);
\dirhgline[gammacol,draw=none,thick](0.5,0.4,240)(2*atan(1/2)-90:2*atan(2/3)-90:\rad);
\dirhgline[gammacol,draw=none,thick](0.48,0.8,155)(-2*atan(1/2)-90:-90:\rad);
\dirhgline[gammacol,draw=none,thick](0.5,0.4,122)(-2*atan(2/3)-90:-2*atan(1/2)-90:\rad);

 \node[gammacol] at (-50:0.56*\rad) {$\gamma$};

\node at (-1.15,-1.15) {(a)}; 
\end{scope}
\begin{scope}[xshift=3cm]
\draw[thin,grey] (0,0) circle(\rad);


\def\fib{{0,1,1,2,3,5,8,13,21,34,55,89,144,233}}
\foreach \n in {0,...,4}{
	\hgline[gammacol,thick](0,0)(2*atan(\fib[2*\n]/\fib[2*\n+1])-90:2*atan(\fib[2*\n+2]/\fib[2*\n+3])-90:\rad);
	\hgline[gammacol,thick](0,0)(-2*atan(\fib[2*\n+2]/\fib[2*\n+3])-90:-2*atan(\fib[2*\n]/\fib[2*\n+1])-90:\rad);
}

\foreach \a/\b in {0/270,90/-90,180/270,-36.87/0,-36.87/-22.62,180/216.87,202.62/216.87}{
	\hgline[black](0,0)(\a:\b:\rad);
}
\foreach \a/\b in {-28.07/-25.99,-28.07/-22.62,205.99/208.07,202.62/208.07}{
	\hgline[black](0,0)(\a:\b:\rad);
}

\dirhgline[gammacol,draw=none,thick](0.5,0.8,204)(-90:2*atan(1/2)-90:\rad);
\dirhgline[gammacol,draw=none,thick](0.55,0.3,235)(2*atan(1/2)-90:2*atan(3/5)-90:\rad);
\dirhgline[gammacol,draw=none,thick](0.48,0.8,155)(-2*atan(1/2)-90:-90:\rad);
\dirhgline[gammacol,draw=none,thick](0.52,0.3,122)(-2*atan(3/5)-90:-2*atan(1/2)-90:\rad);

 \node[gammacol] at (-50:0.56*\rad) {$\gamma$};

\hgline[deltacol,thick](0,0)(0:90:\rad);
\hgline[deltacol,thick](0,0)(90:180:\rad);
\hgline[deltacol,thick](0,0)(-22.65:0:\rad);
\hgline[deltacol,thick](0,0)(180:202.65:\rad);
\hgline[deltacol,thick](0,0)(-25.99:-22.62:\rad);
\hgline[deltacol,thick](0,0)(202.62:205.99:\rad);
\hgline[deltacol,thick](0,0)(-26.78:-25.99:\rad);
\hgline[deltacol,thick](0,0)(205.99:206.78:\rad);
\dirhgline[deltacol,draw=none,thick](0.5,0.8,45)(90:180:\rad);
\dirhgline[deltacol,draw=none,thick](0.5,0.8,-45)(0:90:\rad);
\dirhgline[deltacol,draw=none,thick](0.48,0.6,265)(-22.65:0:\rad);
\dirhgline[deltacol,draw=none,thick](0.52,0.6,95)(180:202.65:\rad);

\node[deltacol] at (140:0.32*\rad) {$\delta$};

\draw[thin] (-26.565:\rad) -- ++(0.15,0) node[right,xshift=-3pt,yshift=0pt] {{\scriptsize$\tfrac12(\sqrt5-1)$}};
\draw[thin] (206.565:\rad) -- ++(-0.15,0) node[left,xshift=3pt,yshift=0pt] {{\scriptsize$-\tfrac12(\sqrt5-1)$}};

\node at (-1.15,-1.15) {(b)}; 
\end{scope}

\end{tikzpicture}
\caption{Triangulated apeirogons obtained from the clockwise bi-infinite paths of Figure~\ref{fig15}}
\label{fig16}
\end{figure}

Triangulations of essentially the same type as those considered here appear in \cites{BaPaTs2016,BaPaTs2018}, where they are defined combinatorially; however, in that work the underlying geometrical object is an infinite strip rather than a disc (or half-plane). This is not a significant difference because any triangulated apeirogon  can be transported between a disc and a strip using a conformal map. Indeed, an infinite strip is a particularly suitable space for a triangulation $\mathscr{A}$ constructed from a clockwise bi-infinite path $\gamma$ when the limits $\gamma_{-\infty}$ and $\gamma_\infty$  are distinct, because those limit points can then be chosen to correspond to the two infinite ends of the strips. And in that case, the vertices of $\gamma$ and $\delta$ occupy opposite sides of the strip boundary. From our geometric perspective, the strip model is less desirable when the backward and forward limits coincide.

It was observed in \cite{BaPaTs2016}*{Theorem~5.2} that the quiddity sequence of a positive infinite frieze can be read off from the corresponding triangulation. We can explain this in our terminology as follows. 

Consider a triangulated apeirogon $\mathscr{A}$ constructed from a clockwise bi-infinite path $\gamma$. Let $\mathbf{M}$ be the positive infinite frieze corresponding to $\gamma$. By Theorem~\ref{thmZ}, the itinerary for $\gamma$ is equal to the quiddity sequence of $\mathbf{M}$. But the $i$th term of the itinerary just records the number of triangles in $\mathscr{A}$ that are incident to the $i$th vertex of $\gamma$. Thus we can obtain the quiddity sequence from an infinite form of the triangle-counting procedure originally employed by Conway and Coxeter.

We finish this section with a theorem that characterises clockwise bi-infinite paths for which $\gamma_{-\infty}$ and $\gamma_\infty$ are equal and irrational by using the corresponding positive infinite frieze $\mathbf{M}$. One can obtain similar results to classify the other types of paths that we have considered here (with rational/irrational limits which may or may not coincide) using the corresponding friezes, but to go into this now would draw out this discussion beyond reasonable bounds.

\begin{theorem}\label{thm99}
Let $\gamma$ be a clockwise bi-infinite path and let $\mathbf{M}$ be the positive infinite frieze such that $\widetilde{\Phi}(\gamma)=\pm \mathbf{M}$. The following statements are equivalent:
\begin{enumerate}
\item\label{thm99a} $\gamma_{-\infty}$ and $\gamma_{\infty}$ coincide and are equal to an irrational number
\item\label{thm99b} there are increasing sequences of positive integers $i_1,i_2,\dotsc$ and $j_1, j_2,\dotsc$ with $m_{i_k,-j_k}=1$, for $k=1,2,\dotsc$.
\end{enumerate}
\end{theorem}

\begin{proof}
We use the usual notation $\gamma=\langle\,\dotsc,v_{-1},v_0,v_1,v_2\dotsc\rangle$, with $v_i=a_i/b_i$, for $i\in\mathbb{Z}$, where $a_ib_{i+1}-a_{i+1}b_i=1$.

Suppose first that $\gamma_{-\infty}$ and $\gamma_{\infty}$ coincide and are equal to an irrational number $\alpha$. By Lemma~\ref{lemT} we can choose increasing and decreasing sequences of rationals $(x_n)$ and $(y_n)$ both with limit $\alpha$ such that $x_n$ and $y_n$ are adjacent in $\mathscr{F}$ for each index $n$. Now, the sequence $v_0,v_{-1},v_{-2},\dotsc$ is eventually increasing, with limit $\alpha$, and the sequence $v_0,v_1,v_2,\dotsc$ is eventually decreasing, again with limit $\alpha$. By Lemma~\ref{lemA}, both sequences must pass through one of the points $x_n$ and $y_n$, for all but finitely many indices $n$, and, evidently, because of the ordering, the first sequence passes through $x_n$ and the second through $y_n$. Given any sufficiently large positive integer $n$, then, there are positive integers $i$ and $j$ for which $v_{i}=y_n$ and $v_{-j}=x_n$, so $v_{i}$ and $v_{-j}$ are adjacent and
\[
m_{i,-j}=a_{-j}b_i-b_{-j}a_i=1.
\]
Since this is so for all but finitely many pairs $x_n$ and $y_n$, we see that there are infinitely many positive integers $i$ and $j$ (all distinct) with $m_{i,-j}=1$.

For the converse,  suppose there are increasing sequences of positive integers $i_1,i_2,\dotsc$ and $j_1, j_2,\dotsc$ with $m_{i_k,-j_k}=1$, for $k=1,2,\dotsc$. It follows that $v_{i_k}$ is adjacent to $v_{-j_k}$, for each $k$. Now, for any positive constant $\varepsilon$ there are only finitely many edges of the disc model of the Farey graph of Euclidean diameter greater than $\varepsilon$, so we see that the convergent sequences $v_{i_1},v_{i_2},\dotsc$ and $v_{-j_1},v_{-j_2},\dotsc$ converge to the same limit $\alpha$. Hence $\gamma_{-\infty}=\gamma_{\infty}=\alpha$.

Suppose $\alpha\in\mathbb{Q}\cup\{\infty\}$. By applying an element of $\text{SL}_2(\mathbb{Z})$, we can assume that $\alpha=\infty$. Then $v_{i_1},v_{i_2},\dotsc$ is a decreasing sequence, unbounded below, and $v_{-j_1},v_{-j_2},\dotsc$ is an  increasing sequence, unbounded above. However, this is impossible, because $v_{i_k}$ cannot be adjacent to $v_{-j_k}$ for all $k=1,2,\dotsc$ under these circumstances. Therefore $\alpha$ is irrational.
\end{proof}

For an example of Theorem~\ref{thm99} in action, refer back to the path $\gamma$ illustrated in Figure~\ref{fig11}(b) and the corresponding positive infinite frieze $\mathbf{M}$ of Figure~\ref{fig12}(b). This frieze satisfies
\[
m_{i,-i} = F_{2i-1}F_{2i+1}-F_{2i}^2=1,
\]
for $i=1,2,\dotsc$, by Cassini's identity for Fibonacci numbers, so there is a 1 in every row of the frieze in standard form.


\section{Combinatorics of positive $\text{SL}_2$-tilings}\label{sec9}

A widely-used tool in the theory of $\text{SL}_2$-tilings is \emph{Conway--Coxeter counting}, which, for a hyperbolic triangulated apeirogon $\mathscr{A}$, can be described as follows. Choose any vertex $u$ of $\mathscr{A}$. We assign the value $0$ to that vertex. To each neighbour of $u$ in $\mathscr{A}$ we assign the value $1$. Next we carry out the following recursive process to assign a value to any particular vertex $v$ of $\mathscr{A}$. If $v$ is adjacent to two vertices of $\mathscr{A}$ that have already been assigned the values $b$ and $d$, then we assign the value $b+d$ to $v$. We denote the value at $v$ from this counting procedure by  $\kappa(u,v)$. 

Every vertex $v$ of $\mathscr{A}$ can be assigned a value $\kappa(u,v)$ that is uniquely specified by this process. We do not prove this fact, although it is not difficult. Actually, it can be established using similar ideas to those used to prove the next lemma. This lemma demonstrates how the arithmetic of the Farey graph neatly encapsulates Conway--Coxeter counting. We recall from the introduction that for reduced rationals $a/b$ and $c/d$, $\Delta(a/b,c/d)=|ad-bc|$.

\begin{lemma}\label{lemS}
Let $u$ and $v$ be vertices of a triangulated apeirogon. Then $\kappa(u,v)=\Delta(u,v)$.
\end{lemma}
\begin{proof}
Observe first that both $\kappa$ and $\Delta$ are invariant under $\text{SL}_2(\mathbb{Z})$, in the sense that 
\[
\kappa(g(u),g(v))=\kappa(u,v)\quad\text{and}\quad \Delta(g(u),g(v))=\Delta(u,v),
\]
for $g\in\text{SL}_2(\mathbb{Z})$, where $u$ and $v$ are vertices of the triangulated apeirogon $\mathscr{A}$. Hence we can assume that $u=\infty$.

Let us write $v=a/b$ as a reduced rational with $b>0$. Then $\Delta(\infty,v)=b$. We will prove that $\kappa(\infty,v)=b$ also.

All vertices adjacent to $u=\infty$ have the form $m/1$, for $m\in\mathbb{Z}$. Reduced rationals for other vertices in the Farey graph can be obtained by Farey addition: if $p/q$ and $r/s$ are neighbours in $\mathscr{F}$, then the unique vertex $w$ of $\mathscr{F}$ between $p/q$ and $r/s$ on the real line for which $w$, $p/q$ and $r/s$ form a triangle in $\mathscr{F}$ is $w=(p+r)/(q+s)$. It is well known (and reasonably straightforward to prove) that all vertices of the Farey graph can be obtained in this way. Conway--Coxeter counting simply records the denominators of this repeated process of Farey addition, so $\kappa(\infty,v)$ is equal to the denominator of $a/b$, namely $b$.
\end{proof}

In the previous section we described how to construct a triangulated apeirogon from a clockwise bi-infinite path. Now we discuss how to obtain a triangulated apeirogon from an element $(\gamma,\delta)$ of $(\mathscr{P}\times\mathscr{P})^+$ -- a pair of clockwise bi-infinite paths. First we carry out the procedure from the  previous section to give two triangulated apeirogons $\mathscr{A}_\gamma$ and $\mathscr{A}_\delta$, the first for $\gamma$ and the second for $\delta$. It will help to think of these apeirogons as infinite connected subtrees of the dual graph of the Farey graph. Consider the unique minimal path in the dual graph connecting $\mathscr{A}_\gamma$ to $\mathscr{A}_\delta$ (which is empty if $\mathscr{A}_\gamma$ and $\mathscr{A}_\delta$ share a common triangle). We define a new triangulated apeirogon $\mathscr{A}$ to consist of the vertices and edges from the collection of all triangles in $\mathscr{A}_\gamma$, $\mathscr{A}_\delta$ and this minimal path. One can verify that $\mathscr{A}$ is indeed a triangulated apeirogon, and none of the vertices of $\mathscr{A}$ lie in-between any two consecutive vertices of $\gamma$ or any two consecutive vertices of $\delta$ (in the clockwise sense). 

Using a result similar to Theorem~\ref{thmXYZ} it can be shown that the vertices of the triangulated apeirogon $\mathscr{A}$ can be split into four clockwise paths: $\gamma$ and $\delta$, and two other paths $\alpha$ and $\beta$, each of which may be empty, finite, half-infinite or bi-infinite. The paths themselves occur in the order $\gamma, \alpha,\delta,\beta$ clockwise in $\mathbb{R}_\infty$, and the only accumulation points of vertices of $\mathscr{A}$ are the four limit points $\gamma_{-\infty}$, $\gamma_\infty$, $\delta_{-\infty}$ and $\delta_\infty$. Precisely which configuration occurs depends on which if any of the four limit points are equal, and whether each one is rational or irrational.

There are many cases, and rather than discussing them all in detail, we supply a single illustrative example. (See \cite{BeHoJo2017} for a more-detailed look at the combinatorial configurations, from a different perspective.) Consider the paths
\[
\gamma =\left\langle\,\dotsc ,-\tfrac4{3},-\tfrac32,-2,\infty,2,\tfrac32,\tfrac4{3},\dotsc \right\rangle
\quad\text{and}\quad 
\delta = \left\langle\,\dotsc ,\tfrac8{13},\tfrac35,\tfrac12,0,-\tfrac12,-\tfrac35,-\tfrac8{13},\dotsc \right\rangle
\]
in Figure~\ref{fig17}(a). Here $\gamma_{-\infty}=-1$ and $\gamma_\infty=1$ are both rational, and $\delta_{-\infty}=\tfrac12(\sqrt5-1)$ and $\delta_\infty=-\tfrac12(\sqrt5-1)$ are both irrational. Figure~\ref{fig17}(b) shows the corresponding triangulated apeirogon $\mathscr{A}$. Since $\delta_\infty$ is irrational and $\gamma_{-\infty}$ is rational, the vertices of $\mathscr{A}$ in between $\delta_\infty$ and $\gamma_{-\infty}$ form a half-infinite path with backward limit $\delta_\infty$ and (forward) final vertex $\gamma_{-\infty}$.  Likewise, the vertices of $\mathscr{A}$ in-between $\gamma_\infty$ and $\delta_{-\infty}$  form a half-infinite path with initial vertex $\gamma_\infty$ and forward limit~$\delta_{-\infty}$.

\begin{figure}[ht]
\centering
\begin{tikzpicture}[scale=2.5]

\begin{scope}
\discfareygraph[thin,grey]{4};

\node[xshift=-2pt] at (0:\radplus) {$1$};
\node[] at (180:\radplus) {$-1$};
\node[gammacol,yshift=-3pt] at (90:\radplus) {$\infty$};

\node[gammacol] at ({2*atan(2)-90}:\radplus) {$2$};
\node[gammacol] at ({-2*atan(2)-90}:\radplus) {$-2$};
\node[gammacol] at ({2*atan(3/2)-90}:\radplus) {$\tfrac32$};
\node[gammacol] at ({-2*atan(3/2)-90}:\radplus) {$-\tfrac32$};

\foreach \n in {1,...,20}{
	\hgline[gammacol,thick](0,0)(2*atan((\n-1)/\n)+90:2*atan(\n/(\n+1))+90:\rad);
	\hgline[gammacol,thick](0,0)(-2*atan(\n/(\n+1))+90:-2*atan((\n-1)/\n)+90:\rad);
}

\dirhgline[gammacol,draw=none,thick](0.5,0.8,24)(90:2*atan(1/2)+90:\rad);
\dirhgline[gammacol,draw=none,thick](0.5,0.4,60)(2*atan(1/2)+90:2*atan(2/3)+90:\rad);
\dirhgline[gammacol,draw=none,thick](0.48,0.8,-25)(-2*atan(1/2)+90:90:\rad);
\dirhgline[gammacol,draw=none,thick](0.5,0.4,-58)(-2*atan(2/3)+90:-2*atan(1/2)+90:\rad);

 \node[gammacol] at (130:0.56*\rad) {$\gamma$};

\draw[thin] (-26.565:\rad) -- ++(0.15,0) node[right,xshift=-3pt,yshift=0pt] {{\scriptsize$\tfrac12(\sqrt5-1)$}};
\draw[thin] (206.565:\rad) -- ++(-0.15,0) node[left,xshift=3pt,yshift=0pt] {{\scriptsize$-\tfrac12(\sqrt5-1)$}};

\node[deltacol,yshift=1pt] at (-90:\radplus) {$0$};
\node[deltacol,xshift=-3pt,yshift=-2pt] at ({2*atan(1/2)-90}:\radplus) {$\tfrac12$};
\node[deltacol,yshift=-2pt] at ({-2*atan(1/2)-90}:\radplus) {$-\tfrac12$};
\node[deltacol,xshift=-3pt,yshift=-2pt] at ({2*atan(3/5)-90}:\radplus) {$\tfrac35$};
\node[deltacol,yshift=-2pt] at ({-2*atan(3/5)-90}:\radplus) {$-\tfrac35$};

\def\fib{{0,1,1,2,3,5,8,13,21,34,55,89,144,233}}
\foreach \n in {0,...,4}{
	\hgline[deltacol,thick](0,0)(2*atan(\fib[2*\n]/\fib[2*\n+1])-90:2*atan(\fib[2*\n+2]/\fib[2*\n+3])-90:\rad);
	\hgline[deltacol,thick](0,0)(-2*atan(\fib[2*\n+2]/\fib[2*\n+3])-90:-2*atan(\fib[2*\n]/\fib[2*\n+1])-90:\rad);
}

\dirhgline[deltacol,draw=none,thick](0.5,0.8,204)(-90:2*atan(1/2)-90:\rad);
\dirhgline[deltacol,draw=none,thick](0.55,0.3,235)(2*atan(1/2)-90:2*atan(3/5)-90:\rad);
\dirhgline[deltacol,draw=none,thick](0.48,0.8,155)(-2*atan(1/2)-90:-90:\rad);
\dirhgline[deltacol,draw=none,thick](0.52,0.3,122)(-2*atan(3/5)-90:-2*atan(1/2)-90:\rad);

\node[deltacol] at (-50:0.56*\rad) {$\delta$};

\node at (-1.45,-1.15) {(a)}; 
\end{scope}

\begin{scope}[xshift=2.9cm]
\draw[thin,grey] (0,0) circle(\rad);

\node[xshift=-2pt] at (0:\radplus) {$1$};
\node[] at (180:\radplus) {$-1$};

\foreach \n in {1,...,20}{
	\hgline[black](0,0)(2*atan((\n-1)/\n)+90:180:\rad);
	\hgline[black](0,0)(0:-2*atan((\n-1)/\n)+90:\rad);
}

\foreach \a/\b in {0/90,0/270,90/-90,90/180,180/270,-36.87/0,-36.87/-22.62,-22.62/0,-28.07/-22.62,180/216.87,180/202.62,202.62/216.87,202.62/208.07}{
	\hgline[black](0,0)(\a:\b:\rad);
}
\foreach \a/\b in {-28.07/-25.99,-26.78/-25.99,-25.99/-22.62,205.99/208.07,205.99/206.78,202.62/205.99}{
	\hgline[black](0,0)(\a:\b:\rad);
}

\foreach \n in {1,...,20}{
	\hgline[gammacol,thick](0,0)(2*atan((\n-1)/\n)+90:2*atan(\n/(\n+1))+90:\rad);
	\hgline[gammacol,thick](0,0)(-2*atan(\n/(\n+1))+90:-2*atan((\n-1)/\n)+90:\rad);
}
\dirhgline[gammacol,draw=none,thick](0.5,0.8,24)(90:2*atan(1/2)+90:\rad);
\dirhgline[gammacol,draw=none,thick](0.5,0.4,60)(2*atan(1/2)+90:2*atan(2/3)+90:\rad);
\dirhgline[gammacol,draw=none,thick](0.48,0.8,-25)(-2*atan(1/2)+90:90:\rad);
\dirhgline[gammacol,draw=none,thick](0.5,0.4,-58)(-2*atan(2/3)+90:-2*atan(1/2)+90:\rad);

\def\fib{{0,1,1,2,3,5,8,13,21,34,55,89,144,233}}
\foreach \n in {0,...,4}{
	\hgline[deltacol,thick](0,0)(2*atan(\fib[2*\n]/\fib[2*\n+1])-90:2*atan(\fib[2*\n+2]/\fib[2*\n+3])-90:\rad);
	\hgline[deltacol,thick](0,0)(-2*atan(\fib[2*\n+2]/\fib[2*\n+3])-90:-2*atan(\fib[2*\n]/\fib[2*\n+1])-90:\rad);
}
\dirhgline[deltacol,draw=none,thick](0.5,0.8,204)(-90:2*atan(1/2)-90:\rad);
\dirhgline[deltacol,draw=none,thick](0.55,0.3,235)(2*atan(1/2)-90:2*atan(3/5)-90:\rad);
\dirhgline[deltacol,draw=none,thick](0.48,0.8,155)(-2*atan(1/2)-90:-90:\rad);
\dirhgline[deltacol,draw=none,thick](0.52,0.3,122)(-2*atan(3/5)-90:-2*atan(1/2)-90:\rad);

\draw[thin] (-26.565:\rad) -- ++(0.05,-0.2) node[below,xshift=6pt,yshift=2pt] {{\scriptsize$\tfrac12(\sqrt5-1)$}};
\draw[thin] (206.565:\rad) -- ++(-0.05,-0.2) node[below,xshift=-8pt,yshift=2pt] {{\scriptsize$-\tfrac12(\sqrt5-1)$}};

\hgline[alphacol,thick](0,0)(-22.65:0:\rad);
\hgline[betacol,thick](0,0)(180:202.65:\rad);
\hgline[alphacol,thick](0,0)(-25.99:-22.62:\rad);
\hgline[betacol,thick](0,0)(202.62:205.99:\rad);
\hgline[alphacol,thick](0,0)(-26.78:-25.99:\rad);
\hgline[betacol,thick](0,0)(205.99:206.78:\rad);
\dirhgline[alphacol,draw=none,thick](0.48,0.6,265)(-22.65:0:\rad);
\dirhgline[betacol,draw=none,thick](0.52,0.6,95)(180:202.65:\rad);

 \node[gammacol] at (130:0.56*\rad) {$\gamma$};
\node[deltacol] at (-50:0.56*\rad) {$\delta$};
\node at (-1.25,-1.15) {(b)}; 
\end{scope}

\end{tikzpicture}
\caption{A pair of clockwise bi-infinite paths in $(\mathscr{P}\times\mathscr{P})^+$ and the corresponding triangulated apeirogon}
\label{fig17}
\end{figure}
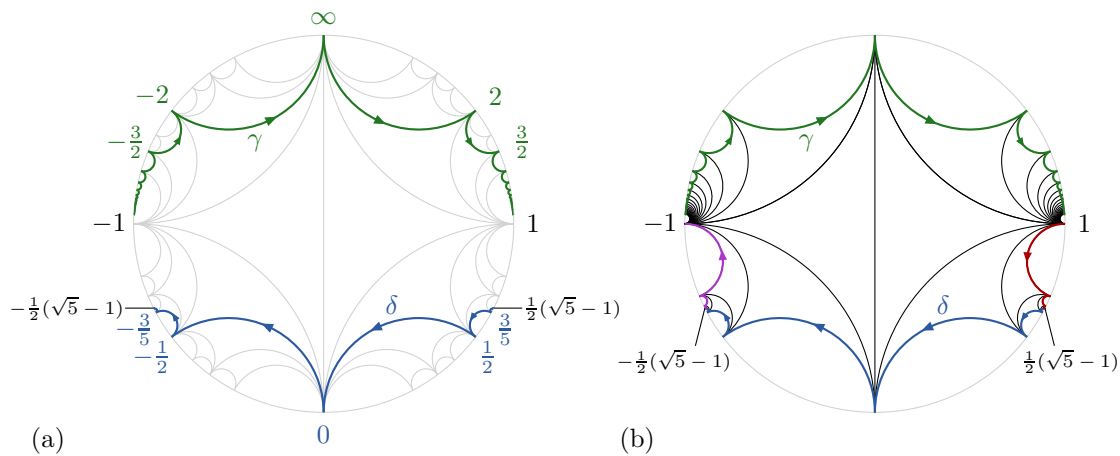

These triangulated apeirogons constructed from pairs of clockwise bi-infinite paths are of the same type as the triangulations considered by Bessenrodt, Holm and J{\o}rgensen in \cite{BeHoJo2017}. Those authors proved that every positive $\text{SL}_2$-tiling can be obtained by choosing such a triangulation and applying Conway--Coxeter counting. Using Lemma~\ref{lemS} and Theorem~\ref{thm2} we can see (in essence, if not with full rigour) that the method for classifying positive  $\text{SL}_2$-tilings using the Farey graph is equivalent to that of \cite{BeHoJo2017}; the geometry and numerics of the Farey graph provide a short cut for simplifying combinatorial arguments.

It was observed in \cite{BeHoJo2017} that any positive $\text{SL}_2$-tiling without any entries equal to $1$ has a unique smallest entry. We reprove that result here using the Farey graph. The following lemma is at the heart of the proof.

\begin{lemma}\label{lemZ}
Let $x$ and $y$ be adjacent vertices of $\mathscr{F}$, and let $u$ and $v$ be vertices that lie in distinct components of $\mathbb{R}_\infty-\{x,y\}$. Then $\Delta(u,x)<\Delta(u,v)$.
\end{lemma}
\begin{proof}
By applying an element of $\text{SL}_2(\mathbb{Z})$ we can assume that $u=\infty$. Then $x$ and $y$ are rationals, and let us suppose that $x<y$, so $v$ lies inside the interval $(x,y)$ (the argument for $y<x$ is much the same). We remarked earlier that every rational inside $(x,y)$ can be obtained by repeatedly applying Farey addition, starting from the pair $x$ and $y$. Observe that the Farey sum of two rationals with positive denominators is a rational with a larger denominator. Thus, writing $x=a/b$ and $v=c/d$ (in reduced form, with $b,d>0$) we see that $b<d$. The result follows, since $\Delta(u,x)=b$ and $\Delta(u,v)=d$. 
\end{proof}

\begin{theorem}\label{thmJ}
Let $\mathbf{M}$ be a positive $\textnormal{SL}_2$-tiling without any entries equal to 1. Then $\mathbf{M}$ has a unique smallest entry.
\end{theorem}

\begin{proof}
By shifting the indices of entries of $\mathbf{M}$ we can assume that $\mathbf{M}$ assumes its least value at $m_{0,0}$. We use the usual notation for paths $\gamma$ and $\delta$ with $\widetilde{\Phi}(\gamma,\delta)=\pm\mathbf{M}$. By applying an element of $\text{SL}_2(\mathbb{Z})$ we can assume that the zeroth vertex $w_0$ of $\delta$ is $\infty$.

Let $v_0$ be the zeroth vertex of $\gamma$. Since $m_{0,0}>1$, it follows that $v_0$ and $w_0$ are not adjacent in $\mathscr{F}$, so $v_0$ is not an integer. Hence $v_0$ has precisely two neighbours $x$ and $y$ in $\mathscr{F}$ with denominators of smaller magnitude than $v_0$. These two rationals are called the \emph{Farey parents} of $v_0$ in some sources, such as \cite{BeHoSh2012}. They lie on either side of $v_0$ -- let us say that $x<v_0<y$ -- and they are themselves neighbours in $\mathscr{F}$.
 
Neither of the paths $\gamma$ nor $\delta$ can pass through $x$ or $y$, because if (say) $\gamma$ did then there would be a vertex $v_i$ with $\Delta(v_i,w_0)<\Delta(v_0,w_0)$, by Lemma~\ref{lemZ}. It follows from Lemma~\ref{lemA} that $\gamma$ lies within the interval $(x,y)$ and $\delta$ lies in the complement in $\mathbb{R}_\infty$ of this interval.

Consider now any vertex $v_i$ of $\gamma$ other than $v_0$ and any vertex $w_j$ of $\delta$. Then $v_i$ lies in either  $(x,v_0)$ or $(v_0,y)$. In both cases we can apply Lemma~\ref{lemZ} to see that $\Delta(v_i,w_j)>\Delta(v_0,w_j)$. This shows that the minimum value of $\Delta(v_i,w_j)$ is only achieved when $i=0$. Applying the same argument with the roles of $v_i$ and $w_j$ reversed shows that the minimum is only achieved at $i=j=0$, as required.
\end{proof}

In contrast to Theorem~\ref{thmJ}, a positive $\text{SL}_2$-tiling can have infinitely many entries 1. We characterise the $\text{SL}_2$-tilings of that type using the Farey graph in the next two results.

\begin{lemma}\label{lemU}
Let $\mathbf{M}$ be a positive $\textnormal{SL}_2$-tiling, and suppose that $m_{r,s}=1$ for some integers $r$ and $s$. If $i>r$ and $j>s$, or $i<r$ and $j<s$, then $m_{i,j}\neq 1$.
\end{lemma}
\begin{proof}
Let $(\gamma,\delta)$ be a pair of bi-infinite clockwise paths in $(\mathscr{P}\times\mathscr{P})^+$ (with the usual notation) such that $m_{i,j}=a_id_j-b_ic_j$. Then $v_r$ and $w_s$ are adjacent in $\mathscr{F}$. Suppose that $i>r$ and $j>s$. Observe that $v_r$, $v_i$, $w_s$, $w_j$ are in clockwise order in $\mathbb{R}_\infty$. It follows that $v_i$ and $w_j$ cannot be adjacent in $\mathscr{F}$, for if they were then the edge of $\mathscr{F}$ between those two vertices and  the edge between $v_r$ and $w_s$ would intersect. Hence $m_{i,j}\neq 1$. For similar reasons we see that $m_{i,j}\neq 1$ when $i<r$ and $j<s$.
\end{proof} 

Lemma~\ref{lemU} says that if there is an entry 1 at position $(r,s)$, then there are no other entries 1 either upwards and leftwards, or downwards and rightwards, of that position. We also observe  that there are only finitely many entries 1 in any row or column of $\mathbf{M}$, for if there were infinitely many then the vertices of one of the paths $\gamma$ and $\delta$ would have to accumulate at a single vertex of the other. 

A consequence of these observations is that a positive $\textnormal{SL}_2$-tiling $\mathbf{M}$ has infinitely many entries 1 if and only if 
 there are increasing sequences of positive integers $i_1,i_2,\dotsc$ and $j_1,j_2,\dotsc$ such that either $v_{i_k}\sim w_{-j_k}$ for  $k=1,2,\dotsc$, or $v_{-i_k}\sim w_{j_k}$ for $k=1,2,\dotsc$. The class of $\text{SL}_2$-tilings satisfying both these conditions was studied in \cite{HoJo2013}.

\begin{theorem}\label{thm88}
Let $\mathbf{M}$ be a positive $\textnormal{SL}_2$-tiling. The tiling $\mathbf{M}$ has infinitely many entries $1$ if and only if either $\gamma_\infty$ and $\delta_{-\infty}$ are equal and irrational, or $\gamma_{-\infty}$ and $\delta_\infty$ are equal and irrational.
\end{theorem}

The proof is similar to that of Theorem~\ref{thm99}, so we hasten through some of the details.

\begin{proof}
Suppose first that $\mathbf{M}$ has infinitely many entries 1. By Lemma~\ref{lemU} (and the fact that there can be only finitely many 1s in any row or column) we see that there are increasing sequences of positive integers $i_1,i_2,\dotsc$ and $j_1,j_2,\dotsc$ such that either $v_{i_k}\sim w_{-j_k}$ for $k=1,2,\dotsc$, or $v_{-i_k}\sim w_{j_k}$ for $k=1,2,\dotsc$. In the former case the sequences $(v_{i_k})$ and $(w_{-j_k})$ must converge to the same irrational limit, so $\gamma_{\infty}$ and $\delta_{-\infty}$ are equal and irrational. In the latter case $\gamma_{-\infty}$ and $\delta_\infty$ are equal and irrational.

For the converse, suppose that $\gamma_\infty$ and $\delta_{-\infty}$ are equal to some irrational $\alpha$ (the other case can be dealt with in a similar manner). Since $\alpha$ is irrational, we can find infinitely many pairs of neighbouring vertices of $\mathscr{F}$ that lie on either side of $\alpha$ and accumulate at $\alpha$. The paths $\gamma$ and $\delta$ must pass through almost every one of these vertices, so we can find increasing sequences of positive integers $i_1,i_2,\dotsc$ and $j_1,j_2,\dotsc$ with $v_{i_k}\sim w_{-j_k}$, for $k=1,2,\dotsc$. Then $m_{i_k,-j_k}=\Delta(v_{i_k},w_{-j_k})=1$, for each positive integer $k$, so $\mathbf{M}$ has infinitely many entries $1$.
\end{proof}






\end{document}